\documentclass[twoside]{article}
\RequirePackage{amsmath}
\RequirePackage{enumerate}
\usepackage{epsfig}
\usepackage{mor}

\received{}                              %
\revised{}                               %
\pubyear{0x}                             %
\pubmonth{Xxxxxxx}                       %
\volume{xx}                              %
\issue{x}                                %
\pages{xxx--xxx}                         %
\firstpage{0xxx}                         %
\DOI{10.1287/moor.xxxx.xxxx}             %
\startpagenumber{1}                      %

\newcommand{\RR}{{\mathbb R}}
\newcommand{\ZZ}{{\mathbb Z}}

\newcommand{\D}{{\mathcal D}}
\newcommand{\C}{{\mathcal C}}
\newcommand{\T}{{\mathcal T}}
\newcommand{\rS}{{\mathbb S}}

\newcommand{\sD}{{\cal D}}

\newcommand{\sF}{{\cal F}}
\newcommand{\sP}{{\cal P}}

\newcommand{\sL}{{\cal L}}

\newcommand{\eeq}{\end{equation}}

\newcommand{\ra}{\rightarrow}
\newcommand{\Ra}{\Rightarrow}

\newcommand{\deq}{\stackrel{\rm d}{=}}
\newcommand{\beql}[1]{\begin{equation}\label{#1}}
\newcommand{\eqn}[1]{(\ref{#1})}
\newcommand{\beq}{\begin{displaymath}}
\newcommand{\eeqno}{\end{displaymath}}

\newcommand{\af}{\alpha}

\newcommand{\lm}{\lambda}
\newcommand{\ep}{\epsilon}

\newcommand{\qandq}{\quad\mbox{and}\quad}

\newcommand{\qforq}{\quad\mbox{for}\quad}
\newcommand{\qorq}{\quad\mbox{or}\quad}
\newcommand{\qasq}{\quad\mbox{as}\quad}

\newcommand{\qinq}{\quad\mbox{in}\quad}
\newcommand{\qforallq}{\quad\mbox{for all}\quad}

\newcommand{\bes}{\begin{equation*}}
\newcommand{\ees}{\end{equation*}}
\newcommand{\bequ}{\begin{equation}}
\newcommand{\bi}{\begin{itemize}}
\newcommand{\ei}{\end{itemize}}
\newcommand{\bsplit}{\begin{split}}
\newcommand{\esplit}{\end{split}}
\newcommand{\bea}{\begin{eqnarray}}
\newcommand{\eea}{\end{eqnarray}}
\newcommand{\beas}{\begin{eqnarray*}}
\newcommand{\eeas}{\end{eqnarray*}}
\newcommand{\btab}{\begin{tabular}}
\newcommand{\etab}{\end{tabular}}

\newcommand{\barq}{\bar{Q}}
\newcommand{\barz}{\bar{Z}}

\newcommand{\barx}{\bar{X}}

\newcommand{\hatlm}{\hat{\lambda}}
\newcommand{\hatmu}{\hat{\mu}}

\def\lm{\lambda}
\def\tinf{\rightarrow\infty}

\def\equalDist{\stackrel{d}{=}}

\def\AA{\mathbb{A}}
\def\SS{\mathbb{S}}

\title{
A Fluid Limit for an Overloaded X Model Via an Averaging Principle}
\ShortTitle{An Averaging Principle}
\ShortAuthors{} \NumberOfAuthors{2}
\FirstAuthor{Ohad Perry}
\FirstAuthorAddress{Department of Industrial Engineering and Management Sciences, Northwestern University, Evanston, IL 60208}
\FirstAuthorEmail{ohad.perry@northwestern.edu}
\FirstAuthorURL{http://users.iems.northwestern.edu/~perry/}
\SecondAuthor{Ward Whitt}
\SecondAuthorAddress{Department of Industrial Engineering and Operations Research, Columbia University, New York, NY 10027}
\SecondAuthorEmail{ww2040@columbia.edu}
\SecondAuthorURL{http://www.columbia.edu/~ww2040/}
\ThirdAuthor{}
\ThirdAuthorAddress{} \ThirdAuthorEmail{} \ThirdAuthorURL{}
\FourthAuthor{} \FourthAuthorAddress{} \FourthAuthorEmail{}
\FourthAuthorURL{} \FifthAuthor{} \FifthAuthorAddress{}
\FifthAuthorEmail{} \FifthAuthorURL{}

\keywords{many-server queues; averaging principle; separation of time scales;
state-space collapse; heavy-traffic fluid limit; overload control}

\MSCcodes{Primary: 60F17, 60K25  ; Secondary: 60G70, 90B22} 

\ORMScodes{Primary: Queues   ; Secondary: Limit Theorems, Transient Results}

\begin{document}
\maketitle

\begin{abstract}
We prove a many-server heavy-traffic fluid limit
for an overloaded Markovian queueing system having two customer
classes and two service pools, known in the call-center
literature as the X model. The system uses the
fixed-queue-ratio-with-thresholds (FQR-T) control, which we
proposed in a recent paper as a way for one service system to
help another in face of an unexpected overload.
Under FQR-T, customers are served by their own service pool until
a threshold is exceeded. Then, one-way sharing is activated
with customers from one class allowed to be served in both
pools. After the control is activated, it aims to keep the two queues at a
pre-specified fixed ratio. For large systems that fixed ratio
is achieved approximately. For the fluid limit, or FWLLN, we
consider a sequence of properly scaled X models in overload operating under
FQR-T. Our proof of the FWLLN follows the compactness approach, i.e., we
show that the sequence of scaled processes is tight, and then show that all converging subsequences have
the specified limit. The characterization step is complicated because
the queue-difference processes, which determine the
customer-server assignments, need to be considered without spatial scaling.
Asymptotically, these queue-difference processes operate on
a faster time scale than the fluid-scaled processes. In the
limit, due to a separation of time scales, the
driving processes
 converge to a time-dependent
steady state (or local average) of a time-varying
fast-time-scale process (FTSP).  This averaging principle
allows us to replace the driving processes with the long-run
average behavior of the FTSP.
\end{abstract}
\normalsize

\section{Introduction.}\label{secIntro}

In this paper we prove that the deterministic fluid approximation for the overloaded $X$ call-center model,
suggested in \cite{PeW09b} and analyzed in \cite{PeW09c},
arises as the many-server heavy-traffic  fluid limit of a properly scaled sequence of
overloaded Markovian X models under the {\em fixed-queue-ratio-with-thresholds} (FQR-T) control.
(A list of all the acronyms appears in \S \ref{AppAcro} in the appendix.)
The $X$ model has two classes of customers and two service pools, one for each class, but with both pools capable of
handling customers from either class.  The service-time distributions depend on both the class and the pool.
The FQR-T control was suggested in \cite{PeW09a} as a way to automatically initiate sharing (i.e., sending customers from one class
to the other service pool) when the system encounters an unexpected overload, while ensuring that sharing
does not take place when it is not needed.

\subsection{A Series of Papers.}

This paper is the fourth in a series. First, in \cite{PeW09a}
we heuristically derived a stationary fluid approximation,
whose purpose was to approximate the steady-state of a large
many-server X system operating under FQR-T during the overload
incident. More specifically, in \cite{PeW09a} we assumed that a
convex holding cost is incurred on both queues whenever the
system is overloaded, and our aim was to develop a control
designed to minimize that cost. (That deterministic cost
approximates the long-run average cost during the overload
incident in the stochastic model.)  We further assumed that the
system becomes overloaded due to a sudden, unexpected shift in
the arrival rates, with new levels that may not be known to the
system managers, and that the staffing of the service pools
cannot be changed quickly enough to respond to that sudden
overload.

Under the heuristic stationary fluid approximation,
in \cite{PeW09a} we proved that a queue-ratio control with thresholds (QR-T) is optimal,
and showed how to calculate the optimal control parameters.  We also showed that the QR-T
control outperforms the optimal static control when the arrival rates are known.
In general, that optimal QR-T control is a function of the arrival rates during the overload
incident, which are assumed to be unknown.  In the special
case of a separable quadratic cost, i.e., for $C(Q_1, Q_2) = c_1 Q_1^2 + c_2 Q_2^2$, with $c_1, c_2$ being two constants,
we proved that the FQR-T control is optimal, so that two queue ratios -- one for each direction of overload --
are optimal for all possible overload scenarios.
More generally, we found that a FQR-T control was approximately optimal,
as illustrated by Figure 4 of \cite{PeW09a}. Thus, in all subsequent work we have focused on the FQR-T control.

Second, in \cite{PeW09b} we applied a heavy-traffic {\em averaging principle} (AP) as
an engineering principle to describe the transient (time-dependent) behavior of a large overloaded X system operating under FQR-T.
The suggested fluid approximation was expressed via an {\em ordinary differential equation} (ODE), which is driven by a stochastic
process. Specifically, the expression of the fluid ODE as a function of time involves the
local steady state of a stochastic process at each time point $t \ge 0$,
which we named the {\em fast-time-scale process} (FTSP).
As the name suggests, the FTSP operates on (an infinitely) faster time scale than
the processes approximated by the ODE, thus converges
to its local steady state instantaneously at every time $t \ge 0$.
Extensive simulation experiments showed that our approximations work remarkably well, even for surprisingly small systems,
having as few as $25$ servers in each pool.

Third, in \cite{PeW09c} we investigated the ODE suggested in \cite{PeW09b} using a dynamical-system approach.
The dynamical-system framework could not be applied directly, since the ODE is driven by a stochastic process,
and its state space depends on the distributional characteristics of the FTSP. Nevertheless,
we showed that a unique solution to the ODE exists over the positive halfline $[0, \infty)$ for each specified initial condition.
The stationary fluid approximation, derived heuristically in \cite{PeW09a}, was shown to exist as the unique
fixed point (or stationary point) for the fluid approximation.
Moreover, we proved that the solution to the ODE converges
to this stationary point, with the convergence being exponentially fast.  (That supports the steady-state approximation used in \cite{PeW09a}.)
In addition, a numerical algorithm to solve the ODE was developed, based on a combination of a matrix-geometric
algorithm and the classical forward Euler method for solving ODE's.

\subsection{Overview.} \label{secOverview}

In this fourth paper, we will prove that the solution to the ODE in \cite{PeW09b, PeW09c} for specified initial condition is indeed
the many-server heavy-traffic fluid limit of the overloaded X model,
which we also call a {\em functional weak law of large numbers} (FWLLN);
see Theorem \ref{th1}; see \S \ref{secAssump} for the key assumptions.
In doing so, we will prove a strong version of
{\em state-space collapse} (SSC) for the server-assignment processes;
see Corollary \ref{thSSCweak} and Theorem \ref{thSSCextend}.
We will also prove a strong SSC result for the two-dimensional queue process in Corollary \ref{corSSCfull}.
In a subsequent paper \cite{PeW10} we prove a {\em functional central limit theorem} (FCLT) refinement of the FWLLN here,
which describes the stochastic fluctuations about the fluid path.

We only consider the $X$ model {\em during} the overload incident, once sharing has begun;
that will be captured by our main assumptions, Assumptions \ref{assA} and \ref{assC}.
As a consequence, the model is stationary (without time-varying arrival rates), but the evolution is transient, because the system does not start in steady state.
Because of customer abandonment,
the stochastic models will all be stable, approaching proper steady-state distributions.
As a further consequence, during the overload incident sharing will occur in only one direction, so that the overloaded $X$ model
actually behaves as an overloaded $N$ model, but that requires proof; that follows from Corollary \ref{thSSC1} and Theorem \ref{thSSCextend}.
Our FWLLN serves as an approximation for the time-dependent behavior of the model, as it approaches steady state.
In addition, we prove a weak law of large numbers for the stationary distributions, showing that the unique fixed point of the fluid limit
is also the limit of the scaled stationary model. Proving that latter result builds on a novel limit-interchange argument,
which requires the established FWLLN.

Convergence to the fluid limit will be established in roughly three steps:
($i$) representing the sequence of systems (\S\S \ref{secRepGen} and \ref{secSSCserv}), ($ii$) proving that the sequence considered is $\C$-tight (\S \ref{secTight}), and
($iii$) uniquely characterizing the limit (\cite{PeW09c} and \S \ref{secProofs}).
The first representation step in \S \ref{secRepGen} starts out in the usual way, involving rate-$1$ Poisson processes,
 as reviewed in \cite{PTW07}.  However, the SSC part in \S \ref{secSSCserv}
requires a delicate analysis of the unscaled sequence.
The second tightness step in \S \ref{secTight} is routine, but the final
characterization step is challenging.  These last two steps are part of the standard compactness approach to proving stochastic-process limits;
 see \cite{B99}, \cite{EK86}, \cite{PTW07} and \S
11.6 in \cite{W02}. As reviewed in \cite{EK86} and
\cite{PTW07}, uniquely
characterizing the limit is usually the most challenging part of the
proof, but it is especially so here.  Characterizing the limit is difficult because the FQR-T
control is driven by a queue-difference process which is not being scaled and
hence does not converge to a deterministic quantity with
spatial scaling. However, the driving process operates in a
different time scale than the fluid-scaled processes, asymptotically achieving
a (time-dependent) steady state at each instant of time,
yielding the AP.

As was shown in \cite{PeW09c}, the AP and the FTSP also complicate the analysis of the limiting ODE.
First, it requires that the steady state of a {\em continuous-time Markov chain} (CTMC), whose distribution depends on the solution
to the ODE, be computed at every instant of time.
(As explained in \cite{PeW09c}, this argument may seem circular at first, since the distribution of the FTSP
is determined by the solution to the ODE, while the evolution of the solution to the ODE is determined by the behavior of the FTSP.
However, the separation of time scales explains why this construction is consistent.)
The second complication is that the AP produces a singularity region in the state space, causing
the ODE to be discontinuous in its full state space.
Hence, both the convergence to the many-server heavy-traffic fluid limit, and the
analysis of the solution to the ODE depend heavily on the state space of the ODE,
which is characterized in terms of the FTSP.
For that reason, many of the results in \cite{PeW09c}
are needed for proving convergence.

\subsection{Literature.}

Our previous papers discuss related literature;
see especially \S 2 of \cite{PeW09a}.
Our FQR-T control extends the FQR control and other queue-and-idleness ratio controls suggested and studied in \cite{GW09a, GW09b, GW10},
but the limits there were established for a different regime under different conditions.
Here we study the FQR-T control and establish limits for overloaded systems.  Unlike that previous work,
here the service rates may depend on both the customer class and the service pool in a very general way.  In particular,
our $X$ model does not satisfy the conditions of the previous theorems even under normal loads.

There is a substantial literature on averaging principles, e.g., see \cite{KL05} and references therein,
but there is not one unified framework that can easily be applied to any model.  Moreover, it is common practice
to use averaging principles as direct approximations, i.e., to simply replace a fast process by its long-term average behavior.
That is the classic approach for deterministic dynamical systems, e.g., see Chapters 10 and 11 of \cite{K02}.
We ourselves took that approach in \cite{PeW09b}.

Averaging principles are relatively rare in operations research.  See p. 71 of \cite{W02} for discussion related to the queueing literature.
Two notable papers in the queueing literature are Coffman et al. \cite{CPR95},
which considers the diffusion limit of a polling system with zero switch-over times,
and Hunt and Kurtz \cite{HK94}, which considers large loss networks under a large family of controls.
The limits via an AP in Hunt and Kurtz \cite{HK94} are the basis for other papers studying loss networks. We refer to
\cite{AFRT06} and \cite{ZZ02}, and references therein.
The work in \cite{HK94} is also closely related to our work since it considers the fluid limits of such loss systems,
with the control-driving process moving on a faster time scale than the other processes considered.

For the important characterization step, we give two proofs, one in the main paper and the other in the appendix.
The shorter proof in the main paper closely follows \cite{HK94},
exploiting martingales and random measures, building on Kurtz \cite{K92}.
In contrast, our second approach exploits stochastic bounds, which we also use in the
important preliminary step establishing state space collapse.

There is now a substantial literature on fluid limits for queueing models, some of which is reviewed in \cite{W02}.
For recent work on many-server queues, see \cite{KR08,KR07}.
Because of the separation of time scales here,
our work is in the spirit of fluid limits for networks of many-server queues in \cite{BHZ05, BHZ06}, but again the specifics are quite different.
Their separation of time scales justifies using a pointwise stationary approximation asymptotically, as in
\cite{MW98,W91}.

\section{Preliminaries.} \label{secModel}

In this section we specify the queueing model, which we refer to
as the X model.  We then specify the FQR-T control. We then
provide a short summary of the many-server heavy-traffic scaling and the different
regimes.  We conclude with our conventions about notation.

\subsection{The Original Queueing Model.}

The Markovian X model has two classes of customers,
initially arriving according to independent Poisson processes with rates $\tilde{\lm}_1$ and $\tilde{\lm}_2$.
There are two queues, one for each class, in which customers that are not routed to service immediately upon arrival wait
to be served.  Customers are served from each queue in order of arrival.
Each class-$i$ customer has limited patience, which is assumed to be exponentially distributed with rate $\theta_i$,
$i = 1,2$. If a customer does not enter service before he runs out of patience, then he abandons the queue.
The abandonment keep the system stable for all arrival and service rates.

There are two service pools, with pool $j$ having $m_j$ homogenous servers (or agents) working in parallel.
This X model was introduced to study two large systems that
are designed to operate independently under normal loads, but can help each other in face of
unanticipated overloads.  We assume that all servers are cross-trained, so that they can serve both classes.
The service times depend on both the customer class $i$ and the server type $j$, and are exponentially distributed;
the mean service time for each class-$i$ customer by each pool-$j$ agent is $1/\mu_{i,j}$.
All service times, abandonment times and arrival processes are assumed to be mutually independent.
The FQR-T control described below assigns customers to servers.

We assume that, at some unanticipated time,
the arrival rates change instantaneously, with at least one increasing.  At this time the overload incident has begun.
{\em  We consider the system only after the overload incident has begun, assuming that it is in effect at the initial time $0$.}
We further assume that the staffing cannot be changed (in the time scale under consideration) to respond
to this unexpected change of arrival rates.  Hence, the arrival processes change from Poisson
with rates $\tilde{\lm}_1$ and $\tilde{\lm}_2$ to Poisson processes with rates $\lm_1$ and $\lm_2$,
where $\tilde{\lm}_i < m_i / \mu_{i,i}$, $i = 1,2$ (normal loading),
but $\lm_i > \mu_{i,i} m_i$ for at least one $i$ (the unanticipated overload).
(These new arrival rates may not be known by the system manager.)
Without loss of generality, we assume that pool $1$ (and class-$1$) is the overloaded (or more overloaded) pool.
The fluid model (ODE) is an approximation for the system performance during the overload incident,
so that we start with the new arrival rate pair $(\lm_1,\lm_2)$.
(The overload control makes sense much more generally; we study its performance in this specific scenario.)

The two service systems may be designed to operate independently under normal conditions (without any overload)
for various reasons.  In \cite{PeW09a, PeW09b} we considered the common case
in which there is no efficiency gain from service by cross-trained agents.
Specifically, in
 \cite{PeW09a} we assumed the {\em strong inefficient sharing condition}
\bequ \label{strong}
\mu_{1,1} > \mu_{1,2} \qandq \mu_{2,2} > \mu_{2,1}.
\eeq
Under condition \eqref{strong}, customers are served at a faster rate when served in their own service pool
than when they are being served in the other-class pool.
However, many results in \cite{PeW09a} hold under the weaker {\em basic inefficient sharing condition:}
$\mu_{1,1} \mu_{2,2} \ge \mu_{1,2} \mu_{2,1}$.

It is easy to see that some sharing can be beneficial if one system is overloaded, while the other is underloaded (has
some slack), but sharing may not be desirable if both systems are overloaded.
In order to motivate the need for sharing when both systems are overloaded, in \cite{PeW09a}
we considered a convex-cost framework.  With that framework, in \cite{PeW09a}
we showed that sharing can be optimal in the fluid approximation, even if it causes the total queue length
(queue $1$ plus queue $2$) to increase. Despite the optimality of the control in the framework of \cite{PeW09a},
in this paper {\em we do not assume that either the strong or the weak inefficient
sharing condition holds}, since the FWLLN holds regardless of the service rates.
We mention that there can be other operational reasons for not sharing customers between pools during normal loads
(e.g., to avoid too much agent distraction), but to share during overloads (e.g., to provide some minimal
service-level constraints for both classes).

Let $Q_i(t)$ be the number of customers in the class-$i$ queue at time $t$, and let $Z_{i,j}(t)$ be the number of class-$i$
customers being served in pool $j$ at time $t$, $i,j = 1,2$.  Given a stationary (state-dependent) routing policy,
the six-dimensional stochastic process
\bequ \label{vecX}
X_6 (t) \equiv (Q_1(t), Q_2(t), Z_{1,1}(t), Z_{1,2}(t), Z_{2,1}(t), Z_{2,2}(t)), \quad t \ge 0,
\eeq
becomes a six-dimensional CTMC.  ($\equiv$ means equality by definition.)
In principle, the optimal control could be found from the theory of Markov decision processes,
but that approach seems prohibitively difficult.  For a complete analysis, we would need to consider
the unknown transient interval over which the overload occurs, and the random initial conditions,
depending on the model parameters under normal loading.  In summary, there is a genuine need for the
simplifying approximation we develop.

\subsection{The FQR-T Control for the Original Queueing Model.}\label{secFQRorig}

The purpose of FQR-T is to prevent sharing when the system is not overloaded, and to rapidly start sharing
when the arrival rates shift.
If sharing is elected, then we allow
sharing in only one direction.
\begin{assumption}{$($one-way sharing$)$}\label{assShare}
Sharing is allowed in only one direction at any one time.
\end{assumption}

When sharing takes place, FQR-T aims to keep the two queues at a certain ratio,
depending on the direction of sharing. Thus, there is one ratio, $r_{1,2}$,
which is the target ratio if class $1$ is being helped by pool $2$, and another target ratio, $r_{2,1}$,
when class $2$ is being helped by pool $1$.  As explained in \cite{PeW09a}, appropriate ratios can be found using the steady-state fluid
approximation.
 In particular, the specific FQR-T control is optimal in the special case
of a separable quadratic cost function.  More generally, fixed ratios are often approximately optimal.

We now describe the control.
The FQR-T control is based on two positive thresholds, $k_{1,2}$ and $k_{2,1}$, and the two queue-ratio parameters, $r_{1,2}$ and
$r_{2,1}$. We define two queue-difference stochastic processes $D_{1,2}(t) \equiv Q_1(t) - r_{1,2} Q_2(t)$ and
$D_{2,1} \equiv r_{2,1} Q_2(t) - Q_1(t)$.  As shown in \cite{PeW09a}, in that convex cost framework
there is no incentive for sharing simultaneously in both directions, implying that these ratio parameters
should satisfy $r_{1,2} \ge r_{2,1}$; see Proposition EC.2 and (EC.11) of \cite{PeW09a}.
However, even without the cost framework, we do not want sharing to ever occur in both directions simultaneously.
Hence we make the following assumption.
\begin{assumption}{$($ordered ratio parameters$)$}\label{assRatio}
The ratio parameters are assumed to satisfy $r_{1,2} \ge r_{2,1}$.
\end{assumption}

As long as $D_{1,2}(t) \le k_{1,2}$ and $D_{2,1}(t) \le k_{2,1}$
we consider the system to be normally loaded (i.e., not overloaded)
so that no sharing is allowed. Hence, in that case, the two classes operate independently.
Once one of these inequalities is violated, the system is considered to be overloaded, and sharing is initialized.
For example, if $D_{1,2}(t) > k_{1,2}$ and $Z_{2,1} (t) = 0$, then class $1$ is judged to be overloaded
and service-pool $2$ is allowed to start helping queue $1$:
If a type-$2$ server becomes available at this time $t$,
then it will take its next customer from the head of queue $1$.
When $D_{1,2}(t) - k_{1,2} \le 0$, new sharing is not initiated. Then new sharing stops until one of the thresholds is
next exceeded.  However, sharing in the opposite direction
with pool $1$ servers helping class $2$ is not allowed until both $Z_{1,2} (t) = 0$ and $D_{2,1} (t) > k_{2,1}$.

It can be of interest to consider alternative variants of the FQR-T control just defined.
For example, it may be desirable to relax the one-way sharing rule imposed above.
We might use additional lower thresholds for $Z_{1,2} (t)$ and $Z_{1,2} (t)$
to allow sharing to start more quickly in the opposite direction when the queue lengths indicate that is desirable.
However, we do not discuss such control variants here.

With the FQR-T control just defined, the
six-dimensional stochastic process $X_6 \equiv \{X_6 (t): t \ge 0\}$ in \eqn{vecX}
is a CTMC.  (The control depends only on the process state.)
This is a stationary model, but we are concerned with its transient behavior, because it is not starting in steady state.
We aim to describe that transient behavior.
The control keeps the two queues at
approximately the target ratio, e.g., if queue $1$ is being
helped, then $Q_1(t) \approx r_{1,2} Q_2(t)$. If sharing is done
in the opposite direction, then $r_{2,1} Q_2(t) \approx Q_1(t)$
for all $t \ge 0$.  That is substantiated by simulation
experiments, some of which are reported in \cite{PeW09a, PeW09b}.
In this paper we will prove that the $\approx$ signs are replaced
with equality signs in the fluid limit.

\subsection{Many-Server Heavy-Traffic Scaling.}\label{secHT}

To develop the fluid limit, we consider a sequence of X systems, $\{X^n_6: n \ge 1\}$ defined as in \eqn{vecX},
indexed by $n$ (denoted by superscript), using the standard many-server heavy-traffic scaling, i.e.,
with arrival rates and number of servers growing proportionally to $n$
\begin{assumption}{$($many-server heavy-traffic scaling$)$}\label{assMS-HT}
As $n \tinf$,
\bequ \label{MS-HTscale}
\bar{\lm}^n_i \equiv \frac{\lm^n_i}{n} \ra \lm_i  \qandq \bar{m}^n_i \equiv \frac{m^n_i}{n} \ra m_i, \quad 0 < \lm_i, m_i < \infty,
\eeq
and the service and abandonment rates held fixed.
\end{assumption}
We then define the associated fluid-scaled stochastic processes
\bequ \label{vec1}
\barx^n_6(t) \equiv n^{-1} X^n_6 (t), \quad t \ge 0.
\eeq
Where $X^n_6$ is defined as in \eqn{vecX} for each $n$.

For each system $n$ there are thresholds $k^n_{1,2}$ and $k^n_{2,1}$ scaled as follows: 
\begin{assumption}{$($scaled thresholds$)$}\label{assThresh}
For $k_{1,2}, k_{2,1} > 0$ and a sequence of positive numbers $\{c_n: n \ge 1\}$,
\bequ \label{thresholds}
k^n_{i,j}/c_n \ra k_{i,j}, \, (i,j) = (1,2) \qorq (2,1), \, \mbox{where} \quad c_n/n \ra 0 \qandq c_n/\sqrt{n} \ra \infty \qasq n \tinf.
\eeq
\end{assumption}
The first scaling by $n$ in Assumption \ref{assThresh} is chosen to make the thresholds asymptotically negligible
in many-server heavy-traffic fluid scaling, so they have
no asymptotic impact on the steady-state cost (in the cost framework of \cite{PeW09a}).
The second scaling by $\sqrt{n}$ in Assumption \ref{assThresh} is chosen to make
the thresholds asymptotically infinite in many-server heavy-traffic diffusion scaling, so that asymptotically the thresholds will
not be exceeded under normal loading.  It is significant that the scaling shows that we should be able to
simultaneously satisfy both conflicting objectives in large systems.

We consider an overload incident in which class $1$ is overloaded, and more overloaded than class $2$ if both are overloaded.
Hence we primarily focus on the queue difference processes $D^n_{1,2}$.
We redefine the queue-difference process 
by subtracting $k^n_{1,2}$ from $Q^n_1$, i.e.,
\bequ \label{Dprocess}
D^n_{1,2}(t) \equiv Q^n_1(t) - k_{1,2}^n - r_{1,2} Q^n_2(t), \quad t \ge 0.
\eeq
(Similarly, we write $D^n_{2,1} (t) \equiv r_{2,1} Q^n_2(t) - k_{2,1}^n - Q^n_1(t)$.)
We now apply FQR using the process $D^n_{1,2}$ in \eqref{Dprocess}: if $D^n_{1,2}(t) > 0$, then every newly available agent (in either pool) takes
his new customer from the head of the class-$1$ queue. If $D^n_{1,2}(t) \le 0$, then every newly available agent takes his new customer
from the head of his own queue.

Let
\bequ \label{alone}
\rho^n_i \equiv \frac{\lm^n_i} {\mu_{i,i}m^n_i}, \qandq \rho_i \equiv \lim_{n\tinf} \rho^n_i = \frac{\lm_i}{\mu_{i,i}m_i}, \quad i = 1,2.
\eeq
Then $\rho^n_i$ is the traffic intensity of class $i$ to pool $i$, and $\rho_i$ can be thought of as its fluid counterpart.

Our results depend on the system being overloaded. However, in our case, a system can be overloaded even if one of the classes is not
overloaded by itself. We define the following quantities:
\bequ \label{Qalone}
q_i^a \equiv \frac{(\lm_i - \mu_{i,i} m_i)^+}{\theta_i} \qandq s^a_i \equiv \left( m_i - \frac{\lm_i}{\mu_{i,i}} \right)^+, \quad i = 1,2,
\eeq
where $(x)^+ \equiv \max\{ x, 0 \}$.
It is easy to see that
$q^a_i s^a_i = 0$, $i = 1,2$.
Note that $q^a_i$ is the steady-state of the class-$i$ fluid-limit queue when there is no sharing, i.e., when both classes
operate independently. Similarly, $s^a_i$ is the steady state of the class-$i$ fluid-limit idleness process.
For the derivation of the quantities in \eqref{Qalone}
see Theorem 2.3 in \cite{W04}, especially equation (2.19), and \S 5.1 in \cite{PeW09a}. See also \S 6 in \cite{PeW09c}.

\subsection{Conventions About Notation.}\label{secNotation}

We use capital letters to denote random variables and stochastic processes,
and lower-case letters for deterministic counterparts.
For a sequence $\{Y^n : n \ge 1\}$ (of stochastic processes or random variables)
we denote its fluid-scaled version by $\bar{Y}^n \equiv Y^n / n$.
 The fluid limit of stochastic processes $\bar{Y}^n$ will be denoted by a lower case letter $y$, and sometimes by $\bar{Y}$.
Let $\Rightarrow$ denote convergence in distribution.  Then the fluid limits will be expressed as
 $\bar{Y}^n \Rightarrow y$ as $n \ra \infty$.

We use the usual $\RR$, $\ZZ$ and $\ZZ_{+}$ notation for the real numbers, integers and nonnegative integers, respectively.
Let $\RR_k$ denote all $k$-dimensional vectors with components in $\RR$.
For a subinterval $I$ of $[0,\infty)$,
let $\D_k(I) \equiv \D(I, \RR_k)$ be the space of all right-continuous $\RR_k$ valued functions on $I$ with limits from the left everywhere,
endowed with the familiar Skorohod $J_1$ topology. We let $d_{J_1}$ denote the metric on $\D_k(I)$.
Since we will be considering continuous limits, the topology is equivalent to uniform convergence on compact subintervals of $I$.
If $I$ is an arbitrary compact interval, we simply write $\D_k$.
Let $\C_k$ be the subset of continuous functions in $\D_k$.  Let $e$ be the identity function in $\D \equiv \D_1$, i.e.,
$e (t) \equiv t$, $t \in I$. 

We use the familiar big-$O$ and small-$o$ notations for deterministic functions:
For two real functions $f$ and $g$, we write
\bes
\bsplit
f(x) & = O(g(x)) \quad \mbox{whenever} \quad \limsup_{x \ra \infty} |f(x) / g(x)| < \infty, \\
f(x) & = o(g(x)) \quad \mbox{whenever} \quad \limsup_{x\tinf} |f(x) / g(x)| = 0.
\end{split}
\ees
The same notation is used for sequences, replacing
$x$ with $n \in \ZZ_{+}$.

For $a, b \in \RR$, let $a \wedge b \equiv \min{\{a,b\}}$, $a \vee b \equiv \max{\{a,b\}}$, $(a)^+
\equiv a \vee 0$ and $(a)^- \equiv 0 \vee -a$. For a function
$x : [0, \infty) \ra \RR$ and $0 < t < \infty$, let \bes
\|x\|_t \equiv \sup_{0 \le s \le t} | x(s)|. \ees Let $Y \equiv
\{Y(t) : t \ge 0\}$ be a stochastic process, and let $f : [0,
\infty) \ra [0, \infty)$ be a deterministic function. We say
that $Y$ is $O_P(f(t))$, and write $Y = O_P(f)$, if $\|Y\|_t /
f(t)$ is {\em stochastically bounded}, i.e., if \bes \lim_{a
\tinf} \limsup_{t \tinf} P \left(\|Y\|_t/f(t) > a \right) = 0.
\ees We say that $Y$ is $o_P(f(t))$ if $\|Y\|_t / f(t)$
converges in probability (and thus, in distribution) to $0$,
i.e., if $\|Y\|_t/f(t) \Rightarrow 0 \qasq t \tinf$. If $f(t)
\equiv 1$, then $Y = O_P(1)$ if it is stochastically bounded,
and $Y = o_P(1)$ if $\|Y\|_t \Rightarrow 0$. We define
$O_P(f(n))$ and $o_P(f(n))$ in a similar way, but with the
domain of $f$ being $\ZZ_{+}$, i.e., $f : \ZZ_{+} \ra [0,
\infty)$. These properties extend to sequences of random
variables and processes indexed by $n$ if the property holds
uniformly in $n$.

\section{Representation of the Fluid Limit.} \label{secFluid}

In this section we represent the fluid limit as a solution to
an ODE which is driven by a FSTP.
In contrast to the six-dimensional scaled process $\barx^n_6$ in \eqn{vec1},
the ODE is only three-dimensional.
 Hence, we start by briefly discussing the dimension reduction in \S \ref{secDimRed}.
 Afterwards, we
define the FTSP in \S \ref{secFTSP} and then present the ODE
in \S \ref{secODE}. We have studied the FTSP and the ODE in \cite{PeW09c}, to
which we refer for more details. In \S \ref{secAssump} we state
three main assumptions.

\subsection{Dimension Reduction.}\label{secDimRed}

We will making assumptions implying that we consider the system during an overload incident
in which class $1$ is overloaded, and more so than class $2$ if it is also overloaded.
We will thus be considering sharing in which only pool $2$ may help class $1$.
We will thus have both service pools fully occupied, with service pool $1$ serving only class $1$.
We will thus have $P(B^n_T) \ra 1$ as $n \ra \infty$, for all $T$, $0 < T < \infty$, where $B^n_T$ is the subset of the underlying
probability space defined by
\bequ \label{red0}
B^n_T \equiv \{Z^n_{1,1} (t) = m^n_1, Z^n_{2,1} (t) = 0 , Z^n_{1,2} (t) + Z^n_{2,2} (t) = m^n_2 \qforq 0 \le t \le T\}.
\eeq
On the set $B^n_T$ the effective dimension is reduced from six to three.
Carefully justifying this SSC
 will be the topic of \S \ref{secSSCserv}.
Thus, in addition to the process $X_6^n$ in \eqref{vecX} for each $n$, we also consider the six-dimensional processes
\bequ \label{red1}
X^{n,*}_6 \equiv (Q^n_1, Q^n_2, m^n_1 e, Z^n_{1,2}, 0e, m^n_2 e - Z^n_{1,2}) \qinq \D_6
\eeq
and the associated three-dimensional processes
\bequ \label{red2}
X^n_3 \equiv (Q^n_1, Q^n_2, Z^n_{1,2}) \qinq \D_3,
\eeq
obtained by truncating the process $X^{n,*}_6$, keeping only the essential first, second and fourth coordinates.
(Note that $P(X^{n,*}_6 = X^{n}_6 \qinq \D_6 ([0,T])) = P(B^n_T) \ra 1$ as $n \ra \infty$.)
We obtain a further alternative representation for the associated three-dimensional fluid-scaled processes $\bar{X}^n_3$, denoted by $\barx^n$ in \S \ref{secSSCserv};
see \eqn{FluidScaled}.

\subsection{The Fast-Time-Scale Process (FTSP).}\label{secFTSP}

Since we consider the system during an overload incident
in which class $1$ is overloaded, and more so than class $2$ if it is also overloaded,
we will primarily consider only the one queue difference processes $D^n_{1,2}$ in \eqn{Dprocess}.
The FTSP can perhaps be
best understood as being the limit of a family of {\em time-expanded queue-difference processes}, defined for each $n \ge 1$ by
\bequ \label{fast102}
D^n_e (\Gamma^n, s) \equiv D^n_{1,2} (t_0 + s / n), \quad s \ge 0.
\eeq
where $X^n$ is the three-dimensional process in \eqn{red2} and we condition
on $X^n (t_0) = \Gamma^n$ for some deterministic vector
$\Gamma^n$ assuming possible values of $X^n (t_0) \equiv(Q^n_1
(t_0), Q^n_2 (t_0), Z^n_{1,2} (t_0))$. (The time $t_0$ is an
arbitrary initial time.) We choose $\Gamma^n$ so that $\Gamma^n/n \ra \gamma$ as $n \ra
\infty$, where $\gamma \equiv (q_1, q_2, z_{1,2})$ is an
appropriate fixed state (in three dimensions, because we will have sharing in only one direction.
The formal statement of the limit for $D^n_e$ in \eqn{fast102} is Theorem
\ref{thConvToFast}.
Since we divide $s$ in \eqn{fast102} by $n$, we are effectively
dividing the rates by $n$. (See \eqref{BDfrozen1}-\eqref{BDfrozen2} for the
transition rates of $D^n_{1,2}$ itself.) We are applying a
``microscope'' to ``expand time'' and look at the behavior
after the initial time more closely.  That is in contrast to
the usual time contraction with conventional heavy-traffic
limits. See \cite{W84} for a previous limit using time
expansion.

Let $r \equiv r_{1,2}$ and let $\gamma \equiv (q_1, q_2, z_{1,2})$ be a possible state
in the three-dimensional state space $\rS \equiv [0, \infty)^2 \times [0, m_2]$.
Directly, we let the FTSP $\{D(\gamma, s) : s \ge 0\}$ be a
pure-jump Markov process with transition rates
$\lm^{(r)}_-(\gamma)$, $\lm^{(1)}_- (\gamma)$, $\mu^{(r)}_-
(\gamma)$ and $\mu^{(1)}_- (\gamma)$ for transitions of $+r$,
$+1$, $-r$ and $-1$, respectively, when $D (\gamma, s)\le 0$.
Similarly, let the transition rates be $\lm^{(r)}_+(\gamma)$,
$\lm^{(1)}_+ (\gamma)$, $\mu^{(r)}_+ (\gamma)$ and $\mu^{(1)}_+
(\gamma)$ for transitions of $+r$, $+1$, $-r$ and $-1$,
respectively, when $D (\gamma, s) > 0$.

We define the transition rates for $D(\gamma)$ as follows:
First, for $D (\gamma, s) \in (-\infty, 0]$ with $\gamma \equiv
(q_1, q_2, z_{1,2})$, the upward rates are
\bequ \label{bd1}
\bsplit \lm^{(1)}_- (\gamma)  \equiv  \lm_1, \qandq \lm_-^{(r)}
(\gamma)  \equiv  \mu_{1,2}z_{1,2} + \mu_{2,2} (m_2 - z_{1,2})
+ \theta_2 q_2.
\end{split}
\eeq
  Similarly, the downward rates are
\bequ \label{bd2}
\mu^{(1)}_- (\gamma) \equiv  \mu_{1,1}m_1  +
\theta_1 q_1 \qandq \mu_-^{(r)} (\gamma) \equiv \lm_2 \eeq

Next, for $D (\gamma, s) \in (0, \infty)$, we have upward rates
\bequ \label{bd3}
\lm^{(1)}_+ (\gamma) \equiv \lm_1 \qandq
\lm_{+}^{(r)} (\gamma) \equiv \theta_2 q_2. \eeq
 The downward rates are
\bequ \label{bd4}
\bsplit \mu^{(1)}_+ (\gamma) \equiv
\mu_{1,1}m_1  + \mu_{1,2}z_{1,2} + \mu_{2,2}(m_2 - z_{1,2}) +
\theta_1 q_1 \qandq \mu_+^{(r)} (\gamma) \equiv \lm_2.
\end{split}
\eeq

As in \S 7.1 of \cite{PeW09c}, we identify important
subsets of the state space $\rS \equiv [0, \infty)^2 \times [0, m_2]$:
\bequ \label{space}
\bsplit
\mathbf{\rS^b} \equiv \{ q_1 - r q_2 = 0 \}, \quad
\mathbf{\rS^+} \equiv \{ q_1 - r q_2 > 0 \}, \quad
\mathbf{\rS^-} \equiv \{ q_1 - r q_2 < 0 \}.
\end{split}
\eeq
Let $D (\gamma, \infty)$ be a random variable that has the steady-state limiting distribution
of the FTSP $D (\gamma, s)$ as $s \ra \infty$ and let
\bequ \label{piDef}
\pi_{1,2} (\gamma) \equiv P(D(\gamma, \infty) > 0).
\eeq
That is, $\pi_{1,2}(\gamma)$ is the probability that the stationary FTSP associated with $\gamma \in \rS$ is strictly positive.

It turns out that $D (\gamma, \infty)$  and $\pi_{1,2} (\gamma)$ are well defined throughout $\rS$.
In $\rS^b$ the function $\pi_{1,2}$ can assume its full range of values, $0 \le \pi_{1,2} (\gamma) \le 1$;
the boundary subset $\rS^b$ is where the AP is taking place.
For all $\gamma \in \rS^+$, $\pi_{1,2}(\gamma) = 1$;
for all $\gamma \in \rS^-$, $\pi_{1,2}(\gamma) = 0$.
In order for $\rS^-$ to be a proper subspace of $\rS$,
both service pools must be constantly full.
Thus, if
$\gamma \in \rS^-$, then $z_{1,1} = m_1$
and $z_{1,2} + z_{2,2} = m_2$, but $q_1$ and $q_2$ are allowed to be equal to zero.

An important role will be played by the subset $\AA$ of $\rS^b$ such that the FTSP is
positive recurrent.  The AP takes place only in the set $\AA$.
In Theorem
6.1 of \cite{PeW09c} we showed that positive recurrence of the
FTSP, and thus the set $\AA$, depends only on the constant
drift rates in the two regions:
\begin{eqnarray}\label{drifts}
\delta_{+} (\gamma) & \equiv & r\left(\lambda^{(r)}_{+} (\gamma) - \mu^{(r)}_{+} (\gamma)\right)
+ \left(\lambda^{(1)}_{+} (\gamma) - \mu^{(1)}_{+} (\gamma)\right) \nonumber \\
\delta_{-} (\gamma) & \equiv & r\left(\lambda^{(r)}_{-} (\gamma) - \mu^{(r)}_{-} (\gamma)\right)
+ \left(\lambda^{(1)}_{-} (\gamma) - \mu^{(1)}_{-} (\gamma)\right).
\end{eqnarray}
The FTSP $\{D(\gamma, s) : s \ge 0\}$ is positive recurrent if (and only if) the state $\gamma$ belongs to the set
\bequ \label{Aset2}
\AA \equiv \{\gamma \in \SS
: \delta_{-} (\gamma) > 0 > \delta_{+} (\gamma)\}.
\eeq
Let the other two subsets of $\rS^b$ be
\bequ \label{AplusSet}
\AA^{+} \equiv \{x \in \rS^b \; \mid \; \delta_{+}(x) \ge 0\} \qandq
\AA^{-} \equiv \{x \in \rS^b \; \mid \; \delta_{-}(x) \le 0\}.
\eeq

From Theorem 6.2 of \cite{PeW09c}, we obtain the following lemma, giving the limiting behavior of the FTSP for any state in $\rS$.
\begin{lemma}{$($limiting behavior of the FTSP$)$}\label{lmLimFTSP}
For all $\gamma \in \rS$ and $x \in \RR$,
\bequ \label{FTSPdistLim}
\lim_{s \ra \infty} P(D(\gamma, s) \le x) = F(\gamma, x) \equiv P(\gamma, (-\infty, x]),
\eeq
where $F(\gamma, \cdot)$ is a cdf associated with a possibly defective probability measure $P(\gamma, \cdot)$ depending on the state $\gamma$.
Moreover,
\begin{eqnarray}
 \qforallq \gamma \in \AA \qandq x \in \RR, &&\quad 0 < F(\gamma, x) < 1 \qandq 0 < \pi_{1,2}(\gamma) < 1; \nonumber \\
 \qforallq \gamma \in \rS^+ \cup \AA^{+} \qandq x \in \RR, &&\quad  F(\gamma, x) = 0 \qandq \pi_{1,2}(\gamma) = 1;  \\
 \qforallq \gamma \in \rS^- \cup \AA^{-} \qandq x \in \RR, && \quad  F(\gamma, x) = 1 \qandq \pi_{1,2}(\gamma) = 0. \nonumber
\end{eqnarray}
\end{lemma}
Later in \S \ref{secHK}, we obtain a proper limiting steady-state distribution for the FTSP for all $\gamma$ in $\rS$
by appending states $+\infty$ and $-\infty$ to the state space $\RR$ of the FTSP $\{D(\gamma, s) : s \ge 0\}$.
Lemma \ref{lmLimFTSP} then implies that $P(\gamma, \RR) = 1$ for $\gamma \in \AA$,
$P(\gamma, \{+\infty\}) = 1$ for $\gamma \in \rS^+ \cup \AA^{+}$ and $P(\gamma, \{-\infty\}) = 1$ for $\gamma \in \rS^- \cup \AA^{-}$.

\subsection{The Ordinary Differential Equation (ODE).}\label{secODE}

We can now present the three-dimensional ODE in terms of the
FTSP $D$. Let $x (t)  \equiv (q_1 (t), q_2 (t), z_{1,2} (t))$ be the solution to the ODE at time $t$;
let $\dot{x} \equiv (\dot{q}_1, \dot{q}_2, \dot{z}_{1,2})$,
where $\dot{x}(t)$ is the derivative evaluated
at time $t$ and
\bequ \label{odeDetails}
\bsplit \dot{q}_{1}
(t)   & \equiv  \lambda_1  - m_1 \mu_{1,1} - \pi_{1,2}
(x(t))\left[z_{1,2} (t) \mu_{1,2} + z_{2,2} (t)
\mu_{2,2}\right]
- \theta_1 q_1 (t) \\
\dot{q}_{2} (t)   & \equiv  \lambda_2   - (1 - \pi_{1,2}(x(t)))
\left[z_{2,2} (t) \mu_{2,2} + z_{1,2} (t) \mu_{1,2}\right]
- \theta_2 q_2 (t) \\
\dot{z}_{1,2} (t) & \equiv  \pi_{1,2}(x(t)) z_{2,2} (t)
\mu_{2,2} - (1 - \pi_{1,2}(x(t))) z_{1,2} (t) \mu_{1,2},
\end{split}
\eeq
with $\pi_{1,2}(x(t)) \equiv P(D(x(t), \infty) > 0)$ for
each $t \ge 0$, where $D (x(t), \infty)$ has the limiting
steady-state distribution as $s \ra \infty$ of the FTSP $D(\gamma, s)$ for $\gamma = x(t)$.
(Recall also that $z_{2,2} = m_2 - z_{1,2}$.)
Theorem 5.2 of \cite{PeW09c}
shows that the ODE has a unique solution as a continuous function mapping $[0, \infty)$ into $\rS$ for any initial value
in $\rS$.  Lemma \ref{lmLimFTSP} shows that $\pi_{1,2} (x(t))$ is well defined for any $x(t)$ in $\rS$.

Equivalently, we have the following integral representation of
the ODE in \eqn{odeDetails}:
\bequ \label{FluidScaledLim}
\bsplit z_{1,2} (t) & \equiv z_{1,2}(0) + \mu_{2,2}
\int_{0}^{t} \pi_{1,2}(x(s))(m_{2} - z_{1,2} (s)) \, ds
 - \mu_{1,2} \int_{0}^{t} (1 - \pi_{1,2}(x(s))) z_{1,2} (s) \, ds, \\
q_1 (t) & \equiv q_1 (0) + \lm_1 t - m_1 t - \mu_{1,2} \int_{0}^{t} \pi_{1,2}(x(s)) z_{1,2} (s)) \, ds \\
& \quad - \mu_{2,2} \int_{0}^{t} \pi_{1,2}(x(s)) (m_2 - z_{1,2}
(s)) \, ds
- \theta_{1} \int_{0}^{t} q_{1} (s) \, ds, \\
q_2 (t) & \equiv q_2(0) + \lm_2 t - \mu_{2,2} \int_{0}^{t}(1 - \pi_{1,2}(x(s))) (m_2 - z_{1,2} (s)) \, ds \\
& \quad - \mu_{1,2} \int_{0}^{t} (1 - \pi_{1,2}(x(s))) z_{1,2}
(s) \, ds - \theta_{2} \int_{0}^{t} q_{2} (s)) \, ds.
\end{split}
\eeq
We will see that the integral representation in
\eqn{FluidScaledLim} is closely related to an associated
 integral representation of $\bar{X}^n \equiv (\bar{Q}^n_1, \bar{Q}^n_2, \bar{Z}^n_{1,2})$;
see \eqn{FluidScaled}; $\bar{X}^n$ is replaced by the deterministic state $x$ and the
indicators $1_{\{D^n_{1,2} (s) > 0\}}$ are replaced by
$\pi_{1,2} (x(s))$.

It is easy to see that the right-hand side of the ODE is not a continuous function of $x$ and, in particular, is not
locally Lipschitz continuous in $x$. Thus, proving that the ODE posses a unique solution is not straightforward.
The proof of that statement is the main result in \cite{PeW09c}
and builds on matrix-geometric methods, as well as heavy-traffic limit theorems for the FTSP; see Theorems 5.2 and 7.1 there.
The matrix-geometric representation of the FTSP also provides key tools for developing an algorithm to compute that unique solution.


\subsection{Three Main Assumptions.} \label{secAssump}

We now introduce three main assumptions:  Assumptions
\ref{assA}-\ref{assE} below.  {\em All three assumptions are
assumed to hold throughout the paper, unless explicitly stated otherwise.}
These assumptions are in addition to the four assumptions made in \S \ref{secModel}:
Assumptions \ref{assShare}-\ref{assThresh}. (Here we do not require \eqn{strong}.)
Our first
new assumption is on the asymptotic behavior of the model parameters; it
specifies the essential form of the overload. For the
statement, recall the definitions in \eqn{MS-HTscale},
\eqn{thresholds} and \eqn{Qalone}, which describe the asymptotic
behavior of the parameters.

\begin{assumption}{$($system overload, with class $1$ more overloaded$)$}\label{assA} \\ \\
The rates in the overload are such that the limiting rates satisfy
\begin{enumerate} [$(1)$]
\item  \ $\theta_1 q^a_1  > \mu_{1,2} s^a_2$.
\item  \ $q^a_1  > r_{1,2} q^a_2$.
\end{enumerate}
\end{assumption}

Condition $(1)$ in Assumption \ref{assA} ensures that class $1$
is asymptotically overloaded, even after receiving help from
pool $2$. To see why, first observe that, since $s^a_2 \ge 0$,
$q^a_1 > 0$, so that $\lm_1 > \mu_{1,1} m_1$
and $\rho_1 > 1$. Hence, class $1$ is overloaded. Next observe
that $\mu_{1,2} s^a_2 = \mu_{1,2}(1 - \rho_2)^+$, and that $(1
- \rho_2)^+$ is the amount of (steady-state fluid) extra
service capacity in pool $2$, if it were to serve only
class-$2$ customers. Thus, Condition $(1)$ in Assumption
\ref{assA} implies that enough class-$1$ customers are routed
to pool $2$ to ensure that pool $2$ is overloaded when sharing
is taking place. This conclusion will be demonstrated in \S
\ref{secSSCserv}. Condition $(1)$ in Assumption \ref{assA} is
slightly stronger than Condition $(I)$ of Assumption A in
\cite{PeW09c}. because here there is a strong inequality
instead of a weak inequality.

Condition $(2)$ in Assumption \ref{assA} ensures that class $1$
is more overloaded than class $2$ if class $2$ is also overloaded.
This condition helps ensure that there is no incentive for
pool $1$ to help pool $2$, so that
$Z^n_{2,1}$ should remain at $0$.

Our second assumption is about the initial conditions.
For the initial conditions, we assume that the overload,
whose asymptotic character is specified by Assumption \ref{assA}, is ongoing or is about to begin.
In addition, sharing with pool $2$ allowed to help class $1$ has been activated by having the threshold $k^n_{1,2}$ exceeded by
the queue difference process $D^n_{1,2}$ and is in process.  Thus actual sharing is being controlled by
the difference process $D^n_{1,2}$ in \eqn{Dprocess}.  Here is our specific assumption.

\begin{assumption} {$($initial conditions$)$}\label{assC}
For each $n \ge 1$, $P(Z^n_{2,1} (0) = 0, Q^n_i(0) > a_n, i = 1,2) = 1$,
\bes
\barx^n(0) \Rightarrow x(0) \in \AA \cup \AA^{+} \cup \rS^{+} \qasq n \ra \infty, \qandq D^n_{1,2} (0) \Rightarrow L \mbox{ if } x(0) \in \AA \cup \AA^+,
\ees
where $\{a_n : n \ge 1\}$ is a sequence of real numbers satisfying $a_n/c_n \ra a$, $0 < a \le \infty$, for $c_n$ in
\eqref{thresholds}; $D^n_{1,2}$ is defined in {\em \eqn{Dprocess}};
$\barx^n \equiv (\barq^n_1, \barq^n_2, \barz^n_{1,2})$;
$x(0)$ is a deterministic element of $\RR_3$; $\AA$, $\AA^{+}$ and $\rS^{+}$ are the subsets of $\rS$ in {\em \eqn{Aset2}},
{\em \eqn{AplusSet}} and {\em \eqn{space}}; and $L$ is a proper random variable, i.e., $P(|L| < \infty) = 1$.
\end{assumption}
Since we are interested in times when sharing occurs with pool $2$ helping class $1$,
in Assumption \ref{assC} we assume that the scale of $Q^n_1(0)$ is at least as large as that of the threshold $k^n_{1,2}$
(so either the threshold has already been crossed, or it is about to be crossed).
Note that we also assume that $D^n_{1,2}(0) \Ra L$, so that
it is natural to assume that $Q^n_2(0)$ has the same order as $Q^n_1(0)$;
we elaborate in Remark \ref{remTightInA} and Appendix \ref{appPosQ} below.

In Assumption \ref{assC} we do not allow $x(0)$ in $\rS^{-}$, because such an initial condition may activate sharing in the wrong direction,
with pool $1$ helping class $2$, causing the system to leave the state space $\rS$; see Remark \ref{rmLeaveS}.

As noted in \S \ref{secODE}, in \cite{PeW09c} we required that the queue ratio parameter
be rational in order to establish results about the FTSP and the ODE.
\begin{assumption} {$($rational queue ratio parameter$)$}\label{assE}
The queue ratio parameters $r_{1,2}$ and $r_{2,1}$ are rational positive numbers.
\end{assumption}
Given Assumption \ref{assE}, without loss of generality, we let the thresholds be rational (of the form $k^n_{1,2} = m^n/k$
where $r_{1,2} = j/k$).
We conjecture that Assumption \ref{assE} can be removed, but that condition has
been used in \cite{PeW09c} to make the pure-jump Markov FTSP a {\em quasi-birth-and-death} (QBD) process,
which in turn was used to establish critical properties of the FTSP and the ODE. We use some of these properties in this paper as well.
By Assumption \ref{assE}, $r_{1,2} = j/k$ for positive integers $j$ and $k$.
The computational efficiency of the algorithm to solve the ODE developed in \S11 of \cite{PeW09c}
actually depends on $j$ and $k$ not being too large as well,
because the QBD matrices are $2m \times 2m$, where $m \equiv \max{\{j, k\}}$, see \S 6.2 of \cite{PeW09c}, and the steady-state of that QBD must be calculated
at each discretization step in solving the ODE.  Fortunately, simulations show that the system performance is not very sensitive to small changes in $r_{1,2}$, so that having
$m$ be $5$ or $10$ seems adequate for practical purposes.

Relaxing Assumption \ref{assE} will have practical value only if an efficient algorithm for solving the ODE is developed. We remark that
computing the stationary distribution of a pure-jump Markov process can in general be hard and time consuming, and that we need to compute
the stationary distribution of a large number of such processes in order to solve the ODE.
Hence, the ability to analyze the FTSP as a QBD has an important advantage, even if Assumption \ref{assE} is relaxed.

\section{Main Results.} \label{secMain}

In this section we state the main results of the paper. In \S \ref{secStatement} we state the main theorem, establishing the FWLLN via the AP,
proving that the (unique) solution to the ODE \eqref{odeDetails} is indeed the fluid limit of $\barx^n_6$.
In \S \ref{secStationary} we establish convergence of the stationary distributions,
showing that the order of the two limits $n \ra \infty$ and $t \ra \infty$ can be interchanged in great generality.
In \S \ref{secQD} we establish asymptotic results about the queue-difference stochastic process.
We conclude in \S \ref{secProofOver} by giving a brief overview of the following proofs.

\subsection{The Fluid Limit.} \label{secStatement}
We are now ready to state our main result in this paper, which
is a FWLLN for scaled versions of the vector stochastic process
$(X^n_6, Y^n_8)$, where $X^n_6 \equiv (Q^n_i, Z^n_{i,j}) \in
\D_6$ as in \eqn{vecX} and $Y^n_8 \equiv  (A^n_i, S^n_{i,j},
U^n_i) \in \D_8$, $i,j = 1,2$,
where $A^n_i (t)$ counts the number of class-$i$ customer arrivals,
$S^n_{i,j} (t)$ counts the number of service completions of class-$i$ customers by agents in pool $j$,
and $U^n_i (t)$ counts the number of class-$i$ customers to abandon from queue, all in model $n$ during the time interval $[0,t]$.
For the FWLLN, we focus on the
scaled vector process
 \bequ \label{scale}
(\bar{X}^n_6, \bar{Y}^n_8) \equiv n^{-1}(X^n_6, Y^n_8),
\eeq
as in \eqref{vec1}.  To explicitly state the AP, we also consider the functions
\bequ \label{nuAP}
{\Theta}^n(t) \equiv \int_0^t 1_{\{D^n_{1,2}(s) > 0\}}\, ds \qandq \vartheta(t) \equiv \int_0^t \pi_{1,2}(x(s))\, ds, \quad t \ge 0,
\eeq
where $\pi_{1,2}(\cdot)$ is defined in \eqref{piDef}.
In particular, $\pi_{1,2}(x(s))$ is the probability that the {\em stationary} FTSP
$D(x(s), \cdot)$, associated with $x(s)$, is strictly positive, where $x(s)$ is the value of the fluid limit of $\barx^n(s)$ at time $s$, $s \ge 0$.

Recall that Assumptions \ref{assShare}-\ref{assThresh} and \ref{assA}-\ref{assE}
are assumed
to be in force throughout the paper.

\begin{theorem} \label{th1}{$(${\em FWLLN via the averaging principle}$)$}
As $n \ra \infty$,
\beql{limit}
(\bar{X}^n_6, \bar{Y}^n_8, {\Theta}^n) \Rightarrow (x_6,y_8,\vartheta) \qinq \D_{15}([0, \infty)),
\eeq
where
$(x_6,y_8,\vartheta)$ is a deterministic element of $\C_{15}([0, \infty))$,
$x_6 \equiv (q_i, z_{i,j})$, $y_8 \equiv (a_i, s_{i,j}, u_i)$, i = 1, 2; j = 1, 2; $\vartheta$ in \eqref{nuAP};
$z_{2,1} = s_{2,1} = m_1 - z_{1,1}  = m_2 - z_{2,2} - z_{1,2} = 0 e$; $x \equiv (q_1, q_2, z_{1,2})$
is the unique solution to the three-dimensional ODE in \eqref{odeDetails} mapping $[0, \infty)$
into $\rS$.
The remaining limit function $y_8$ is defined in terms of $x_6$:
\begin{eqnarray}\label{y}
a_i (t) \equiv \lambda_i t, \quad s_{i,j} (t) \equiv \mu_{i,j} \int_{0}^{t} z_{i,j} (s) \, ds,
u_i (t) \equiv \theta_i \int_{0}^{t} q_i (s) \, ds \qforq t
\ge 0, \quad i = 1,2; \quad j = 1,2.
\end{eqnarray}
\end{theorem}

We prove Theorem \ref{th1} by showing in \S\ref{secTight} that the sequence $\{(\bar{X}^n_6, \bar{Y}^n_8, \Theta^n): n \ge 1\}$ is $\C$-tight in $\D_{15} ([0,\infty))$
and by showing subsequently that the limit of every convergent subsequence of $\barx^n_6$ must take values in $\rS$
and be a solution to the ODE \eqref{odeDetails},
which has a unique solution by Theorem 5.2 of \cite{PeW09c}.

\subsection{Limit Interchange Result.}\label{secStationary}

Under the FQR-T control operating during a single overload incident of unlimited duration,
the six-dimensional stochastic process $X^n_6 \equiv (Q_i^n, Z_{i,j}^n; i, j = 1,2)$ is
a positive recurrent irreducible CTMC for each $n$.  Hence,
$\bar{X}^n_6 \equiv n^{-1} X^n_6$ has a unique steady-state (limiting and stationary)
distribution $\bar{X}^n_6 (\infty)$ for each $n$.

Theorem 8.2 of \cite{PeW09c} implies that
there exists a unique stationary point $x^* \equiv (q_1^*, q_2^*, z_{1,2}^*)$ in the state space $\rS$
to the three-dimensional limiting ODE in
\eqref{odeDetails}, where
\bequ \label{statPt}
\bsplit
z_{1,2}^* & =
\frac{\theta_2(\lm_1 - m_1\mu_{1,1}) - r_{1,2} \theta_1(\lm_2 - m_2\mu_{2,2})} {r_{1,2} \theta_1\mu_{2,2} + \theta_2\mu_{1,2}} \wedge m_2, \\
q_1^* & = \frac{\lm_1 - m_1\mu_{1,1} - \mu_{1,2} z_{1,2}^*}{\theta_1} \qandq
q_2^* = \frac{\lm_2 - \mu_{2,2} (m_2 - z_{1,2}^*)}{\theta_2}.
\end{split}
\eeq
Let $x^*_6$ be the six-dimensional version of $x^* \equiv (q^*_1, q^*_2, z^*_{1,2})$ in \eqref{statPt}, i.e.,
\bequ \label{fluid6}
x^*_6 \equiv (q^*_1, q^*_2, m_1, z^*_{1,2}, 0, m_2 - z^*_{1,2}) \qforq x^* = (q^*_1, q^*_2, z^*_{1,2}).
\eeq
Observe that, if $\lm_2 - \mu_{2,2}m_2 > 0$, then the numerator in the expression of $z^*_{1,2}$ is equal to $\theta_1 \theta_2 (q_1^a - r_{1,2} q^a_2)$ and is strictly positive
by Condition (2) in Assumption \ref{assA}, so that $0 < z^*_{1,2} \le m_2$.
Moreover, by Corollary 8.2 in \cite{PeW09c}, the two conditions in Assumption \ref{assA} guarantee that
$x^* \in \SS^b \cup \SS^+$, and in particular, that $x^* \in \SS$.

We now establish a limit interchange result.
\begin{theorem} {$($interchange of limits$)$} \label{thIntchange}
For each continuous bounded function $f: \RR_6 \ra \RR$,
\bes
\lim_{n \tinf} \lim_{t \tinf} E[f(\barx^n(t))] = \lim_{t \tinf} \lim_{n \tinf} E[f(\barx^n(t)]  = f(x^*_6),
\ees
where $x^*_6$ is defined in {\em \eqn{fluid6}}.
\end{theorem}

We will prove Theorem \ref{thIntchange} by first proving the limit on the left side.
For that, we can relax the assumptions.
In particular, we will show that the sequence of stationary distributions
converges to the unique stationary point of the ODE, without requiring
Assumptions \ref{assC} and \ref{assE}.  Of course, Assumption \ref{assC} plays no role because it concerns the initial conditions.

The current proof of Theorem 8.2 of \cite{PeW09c} used for \eqn{statPt} above does apply Theorem 5.2 of
Theorem 5.2 of \cite{PeW09c}, which depends on Assumption \ref{assE}, the technical assumption that $r_{1,2}$ and $r_{2,1}$ are rational numbers.
However, we now
show that Theorem 8.2 of \cite{PeW09c} actually does not depend on Assumption \ref{assE}.

\begin{lemma}  Under the conditions of Theorem {\em 8.2} of {\em \cite{PeW09c}}, excluding Assumptions {\em \ref{assC}} and {\em \ref{assE}} here,
$x^*$ is the unique stationary point of the ODE.
\end{lemma}

\begin{proof}  Assume that the conditions of Theorem 8.2 of \cite{PeW09c} are satisfied with an irrational $r \equiv r_{1,2}$.
Construct a sequence of rational numbers $\{r_n: n \ge 1\}$ with $r_n \ra r$ as $n \ra \infty$.  Then, for all $n$ sufficiently large, the conditions of
Theorem 8.2 of \cite{PeW09c} are satisfied with $r_n$.  Let $x^*_n$ be the unique stationary point associated with $r_n$.  Then, by
Theorem 8.1 of \cite{PeW09c}, $x^*_n \ra x^*$ as $n \ra \infty$.
\end{proof}

The existence of a stationary point of an ODE necessarily implies the existence of a (constant) solution to the ODE,
but it does not require the existence of a unique solution to the ODE.  Thus, the existence of a unique solution provided by Theorem 5.2 of \cite{PeW09c}, which does
use Assumptions \ref{assC} and \ref{assE},
is not needed.
Moreover, Theorems 8.3 and 9.2 of \cite{PeW09c} imply that  $x^*$ is globally asymptotically stable and $x(t)$ converges to $x^*$ exponentially fast as $t \tinf$.
These too do not depend on Assumptions \ref{assC} and \ref{assE}.

We now show that $x^*_6$ is the limit of the stationary sequence $\{\barx^n_6(\infty) : n \ge 1\}$ {\em without assuming
Assumptions} \ref{assC} and \ref{assE}.
The proof of Theorem \ref{thStatLim} appears in \S \ref{secStatProofs}.

\begin{theorem}{$($WLLN for the stationary distributions$)$}\label{thStatLim}
Under the assumptions here, excluding Assumptions {\em \ref{assC}} and {\em \ref{assE}},
$\bar{X}^n_6 (\infty) \Ra x^*_6$ in $\RR_6$ as $n \ra \infty$, for $x^*_6$ in {\em \eqn{fluid6}}.
\end{theorem}

\begin{proofof}{Theorem \ref{thIntchange}}
The iterated limit on the left holds by virtue of Theorem \ref{thStatLim}.
The iterated limit on the right holds because of Theorem
\ref{th1} together with the fact that $x^*_6$ is a globally asymptotically stable
stationary point for the fluid limit, by \eqn{fluid6} and Theorem 8.3 of \cite{PeW09c}.
\end{proofof}

\begin{remark}{$($starting in $\rS^{-})$}\label{rmLeaveS}
{\em
It is significant that the limit interchange in Theorem \ref{thIntchange} is not valid
throughout $\rS$.  If Assumption \ref{assC} holds, except that $x(0) \in \rS^{-}$, then $q_1 (0) - r_{1,2} q_2 (0) < 0$.
Together with Assumption \ref{assRatio}, that implies that, in some regions of $\SS^-$,
$d_{2,1} \equiv r_{2,1} q_2 (0) - q_1 (0) > 0$;
that can hold in $\SS^-$ because $r_{1,2} q_2(0)$ can be larger than $q_1(0)$. In those cases
we have $P(D^n_{1,2} (0) < 0) \ra 1$ and $P(D^n_{2,1} (0) > k^n_{2,1}) \ra 1$ as $n \ra \infty$.
If we assume that $P(Z^n_{1,2} (0) = Z^n_{2,1} (0) = 0) = 1$ for all $n \ge 1$, which is consistent with Assumption \ref{assC},
then, asymptotically, we will initially have sharing the wrong way, with pool $1$ helping class $2$.
By the continuity, there will be an interval $[0, \delta]$
for which
\beq
\inf_{\{0 \le t \le \delta\}}{\{d_{2,1} (t)\}} > 0.
\eeqno
Hence, asymptotically as $n \ra \infty$, there will rapidly be sharing with pool $1$ helping class $2$.
It can be shown that
 there exists $\delta > 0$ and $\ep > 0$ such that $P(\bar{Z}^n_{2,1} (\delta) > \ep) \ra 1$
as $n \ra \infty$.  This shows that the limit interchange is not valid for every initial condition in $\rS^{-}$.
}
\end{remark}

\subsection{The Limiting Behavior of the Queue Difference Process.}\label{secQD}

In this section we present important supplementary results that help ``explain'' the AP, which takes place in $\AA$.
The following results are not applied in the proof of Theorem \ref{th1}, but are also not immediate corollaries of the FWLLN;
their proofs are given in \S \ref{secAuxProofs}.

For each $n \ge 1$, let $\Gamma^n_6$ be a random state of $X^n_6$ that
is independent of subsequent arrival, service and abandonment processes, and let $\Gamma^n$ be the random state of $X^n_3$ associated with
$\Gamma^n_6$ as in \eqref{red2}.
\begin{theorem} \label{thConvToFast}
If $\Gamma^n_6/n \Ra \gamma_6$, where $\gamma_6 \equiv (q_1, q_2, m_1, z_{1,2}, 0, m_2 - z_{1,2})$
with $\gamma \equiv (q_1, q_2, z_{1,2}) \in \AA \subset \RR_3$ for $\AA$ in {\em \eqn{Aset2}}
and $D^n_e (\Gamma^n, 0) \Rightarrow D (\gamma, 0)$ in $\RR$ as $n \ra \infty$,
where $D^n_e$ is the time-expanded queue-difference process in {\em \eqn{fast102}}
and $D$ is the FTSP in \S {\em \ref{secFTSP}}, then
\bequ \label{FTSPlim}
\{D^n_e (\Gamma^n, s): s \ge 0\} \Rightarrow \{D (\gamma, s): s \ge 0\} \qinq \D \qasq n \tinf;
\eeq
i.e., we have convergence of the sequence 
of time-inhomogeneous non-Markov processes $\{D^n_e (\Gamma^n) : n \ge 1\}$ to the limiting
time-homogeneous pure-jump Markov process $D(\gamma)$.
\end{theorem}

The next results are about the queue-difference process
$D^n_{1,2}$ itself (as opposed to the expanded queue difference
process $D^n_e$). Recall the definition of stochastic
boundedness in \S \ref{secNotation}. Recall also that tightness
in $\RR$ is equivalent to stochastic boundedness in $\RR$, but
not in $\D$.

\begin{theorem}{$($stochastic boundedness of $D^n_{1,2}$$)$} \label{thDSB}
If $x(t_0) \in \AA$ for some $t_0 \ge 0$, then there exists $t_2 > t_0$ such that $x(t) \in \AA$ for all $t \in [t_0, t_2]$
and for all $t_1$ satisfying $t_0 < t_1 \le t_2$ the following hold:
\begin{enumerate}[(i)]
\item $\{D^n_{1,2} (t): n \ge 1\}$ is stochastically bounded in $\RR$ for each $t$ satisfying $t_1 \le t \le t_2$.
\item $\{\{D^n_{1,2}(t) : t \in I\} : n \ge 1\}$ is neither tight nor stochastically bounded in $\D(I)$, $I \subseteq [t_0, t_2]$.
\item For any sequence $\{c_n : n \ge 1\}$ satisfying $c_n/\log{n} \ra \infty$ as $n \ra \infty$, it holds that
\bequ \label{DlogBd}
\sup_{t_1 \le t \le t_2}{\{D^n_{1,2} (t)/c_n\}} \Ra 0 \qasq n \ra \infty.
\eeq
\end{enumerate}
If $x(t) \in \AA$ for all $t \in [t_0, \infty)$ the above statements hold for any finite $t_2 > t_0$.
\end{theorem}

As an immediate corollary to \eqref{DlogBd} in Theorem \ref{thDSB}, we have the following SSC of the queues.
In particular, that claim implies SSC of the {\em fluid and diffusion scaled queues} when the fluid limit $x$ is in $\AA$.

\begin{coro}{$($SSC of queue process in $\AA$$)$}\label{corSSCfull}
For the interval $[t_1, t_2]$ in Theorem {\em \ref{thDSB}}, $d_{J_1}(Q^n_1, r_{1,2}Q^n_2)/c_n \Ra 0$
in $\D([t_1, t_2])$ as $n \tinf$, for every sequence $\{c_n : n \ge 1\}$ satisfying $c_n/\log{n} \ra \infty$ as $n \tinf$.
If $x(0) \in \AA$ and we consider the interval $[0, t_2]$, then the result is strengthened to hold on $[t_0, t_2] \equiv [0, t_2]$.
\end{coro}

Since the sequence of queue-difference processes is not $\D$ tight, by virtue of Theorem \ref{thDSB},
we cannot have convergence of these processes in $\D$.
However, we can obtain a proper limit for the tight sequence of random variables $\{D^n_{1,2} (t): n \ge 1\}$ in $\RR$
for each fixed $t \in [t_1, t_2]$ by exploiting the AP.
See \cite{W91} for a similar result.

\begin{theorem}{$($pointwise AP$)$}\label{thAPlocal}
Consider the interval $[t_1, t_2]$ in Theorem \ref{thDSB}. Then
$D^n_{1,2}(t) \Ra D(x(t), \infty)$ in $\RR$ as $n \tinf$ for each $t$, $t_1 \le t \le t_2$,
where $D(x(t), \infty)$ has the limiting steady-state distribution of the FTSP $D (\gamma, s)$ for $\gamma = x(t)$.
\end{theorem}

\begin{remark}{$($hitting times of $\AA)$} \label{remTightInA}
{\em
First, the stochastic boundedness in Theorem \ref{thDSB} above actually holds at time $t_0$ and thus in
 the larger interval $[t_0, t_2]$ if $t_0 = 0$ and $x(0) \in \AA$, because of the assumed convergence
of $D^n_{1,2}(0)$ in Assumption \ref{assC}.  However, we cannot get the full convergence in Theorem \ref{thAPlocal}
at $t_0 = 0$ because the limit $L$ in Assumption \ref{assC} need not be distributed the same as $D(x(0), \infty)$.
Second, we may also have
 $t_0 > 0$ because $t_0$ is a hitting time of $\AA$ from $\rS - \AA$.
Even if $x(0) \in \AA$, the fluid limit might leave $\AA$ eventually, and later return to $\AA$ at some time $t_0$;
then $x(t_0) \in \AA$ but
$x(s) \notin \AA$ for all $s \in (t_0 - \ep, t_0)$ for some $\ep > 0$.
If $t_0$ is such a hitting time of $\AA$, then we cannot obtain even a stochastic boundedness result at time $t_0$,
but we obtain the pointwise convergence in Theorem \ref{thAPlocal} in the interval $(t_0, t_2]$, open on the left.
}
\end{remark}

Finally, Theorem \ref{thAPlocal} can be applied to strengthen the conclusion of Theorem \ref{thConvToFast} by showing that
$D^n_e(X^n(t), \cdot)$ converges to a {\em stationary} FTSP $D(x(t), \cdot)$,
with $X^n(t) \equiv (Q^n_1(t), Q^n_2(t), Z^n_{1,2}(t))$, and $x(t) \equiv (q_1(t), q_2(t), z_{1,2}(t))$ is the limit of $\barx^n(t)$
at the fixed time $t$.

\begin{coro}
Suppose that the condition of Theorem {\em\ref{thDSB}} holds. For each $t$ such that the conclusion of Theorem {\em \ref{thDSB}} $(i)$ holds for an interval $[t_1, t_2]$,
$t_1 \le t \le t_2$,
\bes
\{D^n_e(X^n(t), s) : s \ge 0\} \Ra \{D(x(t), s) : s \ge 0\} \quad \mbox{in $\D$ as } n \tinf,
\ees
where the limiting FTSP $D(x(t), \cdot)$ is a stationary process, i.e., $D(x(t), s) \deq D(x(t), \infty)$ for all $s \ge 0$.
\end{coro}

\begin{proof}
First, for $X^n_6(t)$ as in Theorem \ref{th1} and $x_6(t) \equiv (q_1(t), q_2(t), m_1, z_{1,2}(t), 0, m_2 - z_{1,2}(t))$,
we have $\Gamma^n_6/n \Ra \gamma_6$ by Theorem \ref{th1}, where $\Gamma^n_6 \equiv X^n_6(t)$, $\gamma_6 \equiv x_6(t)$
and $\gamma \equiv x(t) \equiv (q_1(t), q_2(t), z_{1,2}(t))$ is in $\AA$ (because of our choice of $t$).
Moreover,
\bes
D^n_e (X^n(t), 0) = D^n_{1,2}(t) \Ra D(x(t), 0) \deq D(x(t), \infty) \quad \mbox{in $\RR$ as $n \tinf$},
\ees
where the first equality holds by the definition of $D^n_e$, and the limit holds by applying Theorem \ref{thAPlocal}.
Hence, the conditions in Theorem \ref{thConvToFast} hold, so that we have convergence in $\D$ of the process
$D^n_e (X^n(t), \cdot)$ to the FTSP $D(x(t), \cdot)$.
Since $D(x(t), 0) \deq D(x(t), \infty)$, the limiting FTSP is stationary as claimed.
\end{proof}

\subsection{Overview of the Proofs.}\label{secProofOver}

The rest of this paper is devoted to proving the six theorems above:
Theorems \ref{th1}-\ref{thAPlocal}.
We prove Theorem \ref{th1} in \S\S \ref{secPrelim}-\ref{secProofs}.
Toward that end, in \S \ref{secPrelim} we establish structural results for the
sequence $\{(\bar{X}^n_6, \bar{Y}^n_8): n \ge 1\}$,
 where $X^n_6 \equiv (Q^n_i, Z^n_{i,j}) \in
\D_6$ as in \eqn{vecX} and $Y^n_8 \equiv  (A^n_i, S^n_{i,j}, U^n_i) \in \D_8$, $i,j = 1,2$
and the associated fluid-scaled process $(\bar{X}^n_6, \bar{Y}^n_8)$ in \eqref{scale}.
In \S \ref{secRepGen} we construct the stochastic processes $(X^n_6, Y^n_8)$ in terms of rate-$1$ Poisson processes.
In \S \ref{secTight} we show that the sequence of stochastic processes
$\{(\bar{X}^n_6, \bar{Y}^n_8): n \ge 1\}$ is $C$-tight in $\D_{14}$ and, consequently, there are convergent subsequences
with smooth limits.
In \S \ref{secSSCserv} we show that the representation established in
\S \ref{secRepGen} can be simplified under Assumptions \ref{assA}-\ref{assE}, reducing
the essential dimension from $6$ to $3$.
The final three-dimensional representation $\barx^n$ in \eqn{FluidScaled} there
explains the form of the ODE in \eqn{odeDetails}.

Given the tightness established in \S \ref{secTight}, we prove the main Theorem \ref{th1} by characterizing the limit of all convergent subsequences
in \S \ref{secProofs}.  Given the SSC established in \S \ref{secSSCserv} and given that the three-dimensional
ODE in \S \ref{secODE} has been shown to have a unique solution in \cite{PeW09c}, it suffices to show that
the limit of any subsequence must almost surely be a solution to the ODE.  For that last step,
our proof in \S \ref{secProofs} follows Hunt and Kurtz \cite{HK94}, which draws heavily upon Kurtz \cite{K92}.
It exploits our martingale representation in theorem \ref{thFluidScaled}
and basic properties of random measures from \cite{K92}.
We also have developed an alternative proof exploiting stochastic bounds.  It is given in \S \ref{secAltProof} in the appendix.
Finally, in \S \ref{secAuxProofs} we prove Theorems \ref{thStatLim}- \ref{thAPlocal}. 

There is more in the appendix.  In \S \ref{secSupport} we present supporting technical results to prove the SSC results in \S \ref{secSSCserv}.
We start by introducing auxiliary frozen queue difference processes in \S \ref{secAux}.
We construct useful bounding processes in \S\S \ref{secStoBd}, \ref{secRateOrder} and \ref{secAbanBd}.
These are primarily for quasi-birth-and-death (QBD) processes, because we exploit a QBD representation for the FTSP;
see \S 6 of \cite{PeW09c} for background.
We establish extreme value limits for QBD processes
in \S \ref{secQBDextreme}. 
In \S \ref{secProofsSSCserv} we exploit the technical results in \S \ref{secSupport}
to prove three theorems stated in \S \ref{secSSCserv}.

\section{Preliminary Results for $X^n_6$.}\label{secPrelim}

In this section we establish preliminary structural results for the vector stochastic process
$(X^n_6, Y^n_8)$, where $X^n_6 \equiv (Q^n_i, Z^n_{i,j}) \in
\D_6$ as in \eqn{vecX} and $Y^n_8 \equiv  (A^n_i, S^n_{i,j}, U^n_i) \in \D_8$, $i,j = 1,2$
and the associated fluid-scaled process $(\bar{X}^n_6, \bar{Y}^n_8)$ in \eqref{scale}.
{\em The results in this section do not depend on Assumptions \ref{assA}-\ref{assE}.}
We do impose the many-server heavy-traffic scaling in \S \ref{secHT}.

In \S \ref{secRepGen} we construct the stochastic processes $(X^n_6, Y^n_8)$ in terms of rate-$1$ Poisson processes.
In \S \ref{secTight} we show that the sequence of stochastic processes
$\{(\bar{X}^n_6, \bar{Y}^n_8): n \ge 1\}$ is $C$-tight in $\D_{14}$.
In Corollary \ref{corLip} we apply the tightness to deduce smoothness properties for the limits of convergent subsequences.

\subsection{Representation of $X^n_6$.}\label{secRepGen}

In this section we develop representations for the basic CTMC $X^n_6$ with the FQR-T control.
At first in this section we do not require Assumptions \ref{assA}-\ref{assE}, so that we can have sharing in either direction,
but in only one direction at any time.  Let $(\lambda^n_1, \lambda^n_2)$ be the pair of fixed positive arrival rates in model $n$,
which here are unconstrained.


Following common practice, as reviewed in \S 2 of \cite{PTW07}, we represent the counting processes in
terms of mutually independent rate-$1$ Poisson processes.  We represent the counting processes
$A^{n}_i$, $S^{n}_{i,j}$ and $U^{n}_{i}$ introduced in the beginning of \S \ref{secStatement} as
\bequ \label{rep1}
\bsplit
A^{n}_i (t)     & \equiv N^{a}_i (\lambda^{n}_i t), \quad
S^{n}_{i,j} (t)  \equiv N^{s}_{i,j} \left(\mu_{i,j} \int_{0}^{t} Z^{n}_{i,j} (s) \, ds\right), \quad
U^{n}_{i} (t)    \equiv N^{u}_i \left(\theta_i \int_{0}^{t} Q^{n}_{i} (s) \, ds\right),
\end{split}
\eeq
for $t \ge 0$, where $N^{a}_i$, $N^{s}_{i,j}$ and $N^{u}_i$ for $i = 1,2; j = 1,2$ are eight mutually independent rate-$1$ Poisson processes.

We can then obtain a general representation of the CTMC $X^n_6$, which is actually valid
for general arrival processes with arrivals one at a time.
Let $S^n_1 \equiv S^{n}_{1,1} +  S^{n}_{2,1}$, $S^n_2 \equiv S^{n}_{1,2} +  S^{n}_{2,2}$
and $S^n \equiv S^{n}_{1} +  S^{n}_{2}$.  Paralleling \eqn{Dprocess}, let
$D^n_{2,1} (t) \equiv r_{2,1} Q^n_{2} (t) - k^n_{2,1} - Q^n_1 (t)$.

\begin{theorem}{$($general representation of $X^n_6)$} \label{thRep}
For each $n \ge 1$, the stochastic process $X^n_6$ is well defined as a random element of $\D_6$ by {\em \eqn{rep1}} and
\begin{eqnarray}
Q^{n}_{1} (t) & \equiv & Q^{n}_{1} (0)  +  \int_{0}^{t} 1_{\{Z^n_{1,1}(s-) + Z^n_{2,1}(s-) = m^n_1\}} \, dA^n_1 (s)
- \int_{0}^{t} 1_{\{D^n_{1,2}(s-) > 0, Z^n_{2,1} (s-) = 0, Q^n_1 (s-) > 0\}} \, dS^n (s) \nonumber \\
&&  - \int_{0}^{t} 1_{\{ \{Q^n_1 (s-) > 0\} \cap (\{Z^n_{2,1}(s-) > 0, D^n_{2,1} (s-) \le 0\}
\cup  \{Z^n_{2,1}(s-) = 0, D^n_{1,2} (s-) \le 0\}) \}} \, dS^n_{1} (s) - U^n_1 (t),  \nonumber \\
Z^{n}_{1,1} (t)  & \equiv & Z^{n}_{1,1} (0)  +  \int_{0}^{t} 1_{\{Z^n_{1,1} (s-) + Z^n_{2,1} (s-) < m^n_1 \}} \, dA^n_{1} (s) \nonumber \\
&& - \int_{0}^{t} 1_{\{(\{Z^n_{1,2}(s-) > 0\} \cup \{ D^n_{2,1} (s-) \le 0\}) \cap \{Q^n_1 (s-) = 0\}\}} \, dS^n_{1,1} (s)
 -  \int_{0}^{t} 1_{\{D^n_{2,1}(s-) > 0, Z^n_{1,2} (s-) = 0\}} \, dS^n_{1,1} (s), \nonumber \\
Z^{n}_{1,2} (t)  & \equiv & Z^{n}_{1,2} (0)  +  \int_{0}^{t} 1_{\{D^n_{1,2}(s-) > 0, Z^n_{2,1} (s-) = 0, Q^n_1 (s-) > 0\}} \, dS^n_{2,2} (s) \nonumber \\
&& - \int_{0}^{t} 1_{\{\{D^n_{1,2}(s-) \le 0\} \cup \{Z^n_{2,1} (s-) >0\} \}} \, dS^n_{1,2} (s).   \nonumber
\end{eqnarray}
Symmetry yields the parallel definitions of $Q^{n}_{2} (t)$, $Z^{n}_{2,2} (t)$ and $Z^{n}_{2,1} (t)$ from
$Q^{n}_{1} (t)$, $Z^{n}_{1,1} (t)$ and $Z^{n}_{1,2} (t)$ by simply switching the subscripts $1$ and $2$.
\end{theorem}

We remark that
the representation of $X^n_6$ in Theorem \ref{thRep} holds even without Assumptions \ref{assA}-\ref{assE}.

\begin{proof}  Just as in Lemma 2.1 of \cite{PTW07}, we can justify the construction
by conditioning on the initial values (the first term in each display) and the counting processes.
With these sample paths specified, we recursively construct the sample path of $X^n_6$.  By applying mathematical induction over successive transition epochs of $X^n_6$,
we show that the sample paths are right-continuous piecewise-constant functions satisfying the equations given.

To explain $Q^n_1$, the second term represents the increase by $1$ at each class-$1$ arrival epoch when service pool $1$ is fully occupied;
otherwise the arrival would go directly
into service pool $1$.  The third term represents the decrease by $1$ when any server completes service and sharing with pool $2$ helping class $1$ active;
that requires that the class-$1$ queue length be positive ($Q^n_1 (s) > 0$); sharing with pool $2$ helping class $1$ occurs when both
$D^n_{1,2}(s) > 0$ and $Z^n_{2,1} (s) = 0$.
The fourth term represents the decrease by $1$ when any pool-$1$ server completes service,
provided that again the queue length is positive ($Q^n_1 (s) > 0$). There are two scenarios:
(i) $\{Z^n_{2,1}(s) > 0, D^n_{2,1} (s) \le 0\}$ and (ii) $\{Z^n_{1,2}(s) = 0, D^n_{1,2} (s) \le 0\}$.
In the first, pool $1$ is helping class $2$, so type-$1$ servers take from queue $1$ only when $D^n_{2,1}(s) \le 0$.
The second scenario is the relative complement within the event $\{Z^n_{2,1} (s) = 0, Q^n_1 (s) > 0\}$
of the event in the third term, i.e., pool $2$ is allowed to help
class $1$, but $D^n_{1,2}(s) \le 0$, so that only type-$1$ servers take from queue $1$ at time $s$.

To explain $Z^n_{1,1}$, the second term represents the increase by $1$ which occurs at each class-$1$ arrival epoch at which service pool $1$
has spare capacity ($Z^n_{1,1} (s) + Z^n_{2,1} (s) < m^n_1$).  The third term represents the decrease by $1$ that occurs when a server in pool $1$
completes service of a class-$1$ customer, with pool $1$
not helping class $2$ ($(\{Z^n_{1,2}(s) > 0\} \cup \{ D^n_{2,1} (s) \le 0\})$) when the class-$1$ queue is empty ($Q^n_1 (s) = 0$).
The fourth term represents the decrease by $1$ that occurs when a server in pool $1$
completes service of a class-$1$ customer, when pool $1$ is helping class $2$ ($\{D^n_{2,1}(s) > 0, Z^n_{1,2} (s) = 0\}$).

To explain $Z^n_{1,2}$, the second term  represents the decrease by $1$ that occurs when a server in pool $2$
completes service of a class-$2$ customer, when class $2$ is helping class $1$ ($\{D^n_{1,2}(s) > 0, Z^n_{2,1} (s) = 0, Q^n_1 (s) > 0\}$).
The third term represents the decrease by $1$ that occurs when a server in pool $2$
completes service of a class-$1$ customer, when class $2$ is not helping class $1$ ($\{D^n_{1,2}(s) \le 0\} \cup \{Z^n_{2,1} (s) >0\}$).

Since the model is fully symmetric, the processes $Q^n_2$, $Z^n_{2,2}$ and $Z^n_{2,1}$ are the symmetric versions
of $Q^n_2$, $Z^n_{2,2}$ and $Z^n_{2,1}$, respectively, with the indices $1$ and $2$ switched.
\end{proof}

\subsection{Tightness and Smoothness of the Limits.}\label{secTight}

We do part of the proof of Theorem \ref{th1} here by
establishing tightness.  For background on tightness, see \cite{B99, PTW07, W02}.
We recall a few key facts:
Tightness of a sequence of $k$-dimensional
stochastic processes in $\D_{k}$ is equivalent to tightness of
all the one-dimensional component stochastic processes in $\D$.
For a sequence of random elements of $\D_{k}$, $\C$-tightness
implies $\D$-tightness and that the limits of all convergent
subsequences must be in $\C_{k}$; see Theorem 15.5 of the first 1968 edition of \cite{B99}.
Alternatively, Conditions (7.6) and (7.7) of Theorem 7.3 in \cite{B99}
hold for processes in $\D$ if and only if conditions (13.4) and (13.5) of Theorem 13.2 of \cite{B99} hold
and the limits of all convergent subsequences are in $\C$; see the Corollary on p. 179 of \cite{B99} or
Theorem VI.3.26 of \cite{JS87}.

\begin{theorem} \label{lmTight}
The sequence $\{(\bar{X}^n_6,\bar{Y}^n_8): n \ge 1\}$ in {\em \eqn{scale}} is\ $\C$-tight in $\D_{14}$.
\end{theorem}

\begin{proof}
It suffices to verify conditions (6.3) and (6.4) of  Theorem 11.6.3 of \cite{W02}, namely
to show that $\barx^n(0)$ is stochastically bounded (tight in $\RR_6$)
and appropriately control the oscillations, using the modulus of continuity on $\C$.
 We obtain the stochastic boundedness at time $0$ immediately from Assumption \ref{assC}.

 We now show that we can control the oscillations below.
  For that purpose, let
$w(x,\zeta,T)$ is the modulus of continuity of the function $x \in \D$, i.e.,
\bequ \label{modulus}
w(x,\zeta,T) \equiv \sup{\{|x(t_2) - x(t_1)|: 0 \le t_1 \le t_2 \le T, |t_2 - t_1| \le \zeta\}}.
\eeq

 Using the representations in \S \ref{secRepGen}, for $t_2 > t_1 \ge 0$ we have
\bes
\begin{split}
\left| \barq^n_1(t_2) - \barq^n_1(t_1) \right| \le \frac{A^n_1(t_2) - A^n_1(t_1)}{n} +
\frac{S^n (t_2) - S^n (t_1)}{n}
+ \frac{S^n_{1,1} (t_2) - S^n_{1,1}(t_1)}{n}
+ \frac{U^n_1(t_2) - U^n_1(t_1)}{n}.
\end{split}
\ees
and similarly for $\barq^n_2$.
Hence, for any $\zeta >0$ and $T > 0$,
\beq
w(Q^n_1/n,\zeta,T) \le w(A^n_1/n,\zeta,T)+ w(S^n/n,\zeta,T) + w(S^n_{1,1}/n,\zeta,T) + w(U^n_1/n,\zeta,T).
\eeqno
Then observe that we can bound the oscillations of the service processes $S^n_{i,j}$ by the oscillations in the
scaled Poisson process $N^s_{i,j} (n \cdot)$.  In particular, by \eqn{rep1},
\beql{serviceBD}
w(S^n_{i,j}/n,\zeta,T) \le w(N^s_{i,j} (n \mu_{i,j} m_j \cdot)/n, \zeta, T) \le w(N^s_{i,j} (n \cdot)/n, c\zeta, T)
\eeq
for some constant $c > 0$.
Next for the abandonment process $U^n_i$, we use the elementary bounds
\begin{eqnarray}
Q^{n}_{i} (t) & \le & Q^n_i (0) + A^n_i (t),  \nonumber \\
|U^n_i (t_2) - U^n_i (t_1)| & = & |N_i (\theta_i \int_{t_1}^{t_2} Q^n_i (s) \, ds |
\le | N_i (n \theta (\bar{Q}^n_i (0) + \bar{A}^n_i (T)) (t_2 - t_1))|. \nonumber
\end{eqnarray}
Let $q_{bd} \equiv 2(q_i (0) + T)$, where $\bar{Q}^n_i (0) \Rightarrow q_i (0)$ by Assumption \ref{assC2},
and let $B_n$ be the following subset of the underlying probability space:
$$B_n \equiv \{\bar{Q}^n_i (0) + \bar{A}^n_i (T) \le q_{bd} \}.$$
Then $P(B_n) \ra 1$ as $n \ra \infty$ and,
on the set $B_n$, we have
\beql{abanBD}
w(U^n_{i}/n,\zeta,T) \le w(N^u_{i} (n q_{bd} \cdot)/n, \zeta, T) \le w(N^u_{i} (n \cdot)/n, c\zeta, T)
\eeq
for some constant $c > 0$.

Thus, there exists a constant $c > 0$ such that, for any $\eta > 0$, there exists $n_0$ and $\zeta > 0$ such that,
for all $n \ge n_0$, $P(B_n) > 1 - \eta/2$ and on $B_n$
\beq
w(Q^n_i/n,\zeta,T) \le w(N^a_i (n \cdot)/n,c\zeta,T)+ 2\sum_{i=1}^{2} \sum_{j=1}^{2}{w(N^s_{i,j}(n \cdot)/n, c\zeta,T)}
+ w(N^u_i (n \cdot)/n,c \zeta,T).
\eeqno
However, by the FWLLN for the Poisson processes, we know that we can control all these moduli of continuity on the right.
Thus we deduce that, for every $\epsilon > 0$ and $\eta > 0$, there exists $\zeta > 0$ and $n_0$ such that
$$P(w(Q^n_i/n,\zeta,T) \ge \epsilon) \le \eta \qforallq n \ge n_0.$$
Hence, we have shown that
the sequence $\{\bar{Q}^n_i\}$ is tight.

We now turn to the sequence $\{\barz^n_{1,2}\}$.
Let $A^n_{1,2}(t)$ denote the total number of class-$1$ arrivals up to time $t$,
who will eventually be served by type-$2$ servers in system $n$.
Let $\bar{A}^n_{1,2} \equiv A^n_{1,2} / n$ and $\bar{S}^n_{1,2}(t) \equiv S^n_{1,2}(t) / n$, for $S^n_{1,2}(t)$ in \eqref{rep1}.
Since
\bes
Z^n_{1,2}(t) = Z^n_{1,2}(0) + A^n_{1,2}(t) - S^n_{1,2}(t),
\ees
we have
\bes
|\barz^n_{1,2}(t_2) - \barz^n_{1,2}(t_1)| \le \bar{A}^n_{1,2}(t_2) - \bar{A}^n_{1,2}(t_1) + \bar{S}^n_{1,2}(t_2) - \bar{S}^n_{1,2}(t_1).
\ees
However, for $A^n_1$ in \eqref{rep1},
\bes
A^n_{1,2}(t_2) - A^n_{1,2}(t_1) \le A^n_1 (t_2) - A^n_1(t_1).
\ees
Since $\bar{A}^n_1 \Rightarrow \lambda_1 e$ in $\D$, the sequence $\{\bar{A}^n_1\}$ is tight.
Together with \eqn{serviceBD}, that implies that the sequence $\{\barz^n_{1,2}\}$ is tight as well.
Finally, we observe that the tightness of $\{\bar{Y}^n_8\}$ follows from \eqn{serviceBD}, \eqn{abanBD} and the convergence of $\bar{A}^n_i$.
\end{proof}


Since the sequence $\{(\bar{X}^n_6,\bar{Y}^n_8): n \ge 1\}$ in
\eqn{scale} is\ $\C$-tight by Theorem \ref{lmTight}, every
subsequence has a further subsequence which converges to a continuous limit.  We  now
apply the modulus-of-continuity inequalities established in the proof of Theorem \ref{lmTight}
to deduce additional smoothness properties of the limits of all
converging subsequence.

\begin{coro}\label{corLip}
If $(\bar{X}_6, \bar{Y}_8)$ is the limit of a
subsequence of $\{(\bar{X}^n_6,\bar{Y}^n_8): n \ge 1\}$ in $\D_{14}$, then each component in $\D$, say $\bar{X}_i$, has bounded modulus of continuity; i.e.,
for each $T>0$, there exists a constant $c > 0$ such that
\bequ \label{modlip}
w(\bar{X}_i, \zeta, T) \le c \zeta \quad w.p.1
\eeq
for all $\zeta > 0$.  Hence $(\bar{X}_6, \bar{Y}_8)$ is Lipschitz continuous w.p.1, and is thus differentiable almost everywhere.
\end{coro}

\begin{proof}
Apply the bounds on the modulus of continuity involving Poisson processes in the proof of Theorem \ref{lmTight} above.
For a Poisson process $N$, let $\hat{N}^n \equiv \sqrt{n}(\bar{N}^n -e)$, where $\bar{N}^n (t) \equiv N(nt)/n$, $t \ge 0$.
By the triangle inequality, for each $n$, $\zeta$, and $T$,
\bes
w(\bar{N}^n, \zeta, T) \le \frac{w(\hat{N}^n, \zeta, T)}{\sqrt{n}} + w(e, \zeta, T) \Rightarrow \zeta \qasq n \ra \infty.
\ees
Since, $w(x, \zeta, T)$ is a continuous function of $x$ for each fixed $\zeta$ and $T$, we can apply this bound with the inequalities
in the proof of Theorem \ref{lmTight} to deduce \eqn{modlip}.
\end{proof}


We remark in closing this section that Theorem \ref{lmTight} and Corollary \ref{corLip}
also hold with Assumption \ref{assC} replaced by $\barx^n(0) \Rightarrow x(0)$ as $n \ra \infty$,
where $x(0)$ is a deterministic element of $\RR_3$.

\section{Structural Simplification.}\label{secSSCserv}

We now exploit Assumptions \ref{assA}-\ref{assE} to simplify the representation established in
\S \ref{secRepGen} above, reducing the essential dimension from $6$ to $3$, following the plan described in \S \ref{secDimRed}.
We first establish this dimension reduction over an interval $[0,\tau]$ and later, after Theorem \ref{th1} has been proved over the
same interval $[0,\tau]$, we show that all the results here, and thus Theorem \ref{th1} too, can be extended to the interval $[0, \infty)$.

Let
\bequ \label{stopTau}
\T^n_0 \equiv \inf\{t > 0 : Z^n_{2,1}(t) > 0 \, \, \mbox{or} \, \, Q^n_1(t) = 0\, \, \mbox{or} \, \,   Q^n_2(t) = 0\}.
\eeq
By Assumption \ref{assC}, both queues are initially strictly positive (so there is no idleness in either pool) and $Z^n_{2,1}(0) = 0$.
Hence, $\T^n_0 > 0$ for each $n \ge 1$.
Theorem \ref{thRep} with the definitions in \eqn{rep1} implies the following reduction from six dimensions to three over $[0, \T^n_0]$.  
Let $\deq$ denote equality in distribution for processes.

\begin{coro}\label{corRep3}
On the random interval $[0, \T^n_0]$, $X^n_6 = X^{n,**}_6$ w.p.1, where \beq
X^{n,**}_6  \equiv  (Q^n_1, Q^n_2, m^n_{1,1} e, Z^n_{1,2}, 0 e, m^n_2 e - Z^n_{1,2}),
\eeqno
 with
\bea
Q^{n}_{1} (t) & \equiv & Q^{n}_{1} (0)  +  A^n_1 (t) - \int_{0}^{t} 1_{\{D^n_{1,2}(s-) > 0\}} \, dS^n (s)
- \int_{0}^{t} 1_{\{D^n_{1,2}(s-) \le 0\}} \, dS^n_{1,1} (s) - U^n_1 (t) \nonumber \\
Q^{n}_{2} (t) & \equiv & Q^{n}_{2} (0)  +  A^n_2 (t) - \int_{0}^{t} 1_{\{D^n_{1,2}(s-) \le 0\}} \, dS^n_{2,2} (s)
- \int_{0}^{t} 1_{\{D^n_{1,2}(s-) \le 0\}} \, dS^n_{1,2} (s) - U^n_2 (t) \nonumber  \\
Z^{n}_{1,2} (t)  & \equiv & Z^{n}_{1,2} (0)  +  \int_{0}^{t} 1_{\{D^n_{1,2}(s-) > 0\}} \, dS^n_{2,2} (s)
- \int_{0}^{t} 1_{\{D^n_{1,2}(s-) \le 0\}} \, dS^n_{1,2} (s),
\eea
and $X^{n,**}_6 \deq X^{n,*}_6$, where $X^{n,*}_6  \equiv  (Q^{n,*}_{1}, Q^{n,*}_{2}, m^n_{1,1} e, Z^{n,*}_{1,2}, 0 e, m^n_2 e - Z^{n,*}_{1,2})$ with
\bea
Q^{n,*}_{1} (t) & \equiv & Q^{n,*}_{1} (0) + N^{a}_{1} (\lambda^n_1 t) -  N^{s}_{1, 1} (\mu_{1,1} m^n_{1} t)
-  N^{s,2}_{1,2} \left(\mu_{1,2} \int_{0}^{t} 1_{\{D^{n,*}_{1,2}(s) > 0\}} Z^{n,*}_{1,2} (s)) \, ds\right) \nonumber \\
& & - N^{s}_{2,2} \left(\mu_{2,2} \int_{0}^{t} 1_{\{D^{n,*}_{1,2}(s) > 0\}} (m^n_2 - Z^{n,*}_{1,2} (s)) \, ds\right)
- N^u_1 \left(\theta_{1} \int_{0}^{t} Q^{n,*}_{1} (s) \, ds\right), \label{rep4}\\
Q^{n,*}_{2} (t) & \equiv &  Q^{n,*}_{2} (0) + N^{a}_{2} (\lambda^n_2 t)
-  N^{s,2}_{2,2} \left(\mu_{2,2} \int_{0}^{t} 1_{\{D^{n,*}_{1,2}(s) \le 0\}} (m^n_2 - Z^{n,*}_{1,2} (s)) \, ds\right) \nonumber \\
& \quad & - N^{s}_{1,2} \left(\mu_{1,2} \int_{0}^{t} 1_{\{D^{n,*}_{1,2}(s) \le 0\}} Z^{n,*}_{1,2} (s) \, ds\right)
- N^u_2 \left(\theta_{2} \int_{0}^{t} Q^{n,*}_{2} (s) \, ds\right), \label{rep5} \\
Z^{n,*}_{1,2} (t)  & \equiv & Z^{n,*}_{1,2} (0) + N^{s}_{2,2} \left(\mu_{2,2} \int_{0}^{t} 1_{\{D^{n,*}_{1,2}(s) > 0\}} (m^n_2 - Z^{n,*}_{1,2}) (s) \, ds\right) \nonumber \\
& \quad & - N^{s}_{1,2} \left(\mu_{1,2} \int_{0}^{t} 1_{\{D^{n,*}_{1,2}(s) \le 0\}}  Z^{n,*}_{1,2} (s) \, ds\right), \label{rep3}
\eea
where $N^a_i, N^u_i, N^s_{1,1}, N^s_{i,2}, N^{s,2}_{i,2}$ for $i = 1,2$ are mutually independent rate-$1$ Poisson Processes and
$D^{n,*}_{1,2} (t) \equiv Q^{n,*}_{1} (t) - k^n_{1,2} - r_{1,2} Q^{n,*}_{2} (t)$ as in {\em \eqn{Dprocess}}.
\end{coro}
Note that in two places in the three displays \eqn{rep4}-\eqn{rep3} above we have introduced the new independent rate-$1$ Poisson processes $N^{s,2}_{i,2}$.
Note that $Z^{n,*}_{1,2} (t)$ might be equal to zero for some or all $t$ in $[0, \T^n_0]$.


We next prove that $\T^n_0$ is bounded away from $0$ asymptotically, i.e., that there exists a $\tau > 0$ such that
$P(T^n_0 \ge \tau) \ra 1$ as $n\tinf$.  We do so in two parts (both proved in \S \ref{secProofsSSCserv}):
\begin{theorem}{$($no sharing in the opposite direction$)$} \label{thZ21}
There exists $\tau > 0$ such that $\|Z^n_{2,1}\|_{\tau} \Ra 0$ as $n\tinf$.
\end{theorem}
\begin{theorem} {$($positive queue lengths$)$} \label{thQpos}
For $\tau$ in Theorem \ref{thZ21},
\bes
P\left(\inf_{0 \le t \le \tau} \min\{Q^n_1 (t), Q^n_2(t)\} > 0\right) \ra 1 \qasq n \ra \infty.
\ees
\end{theorem}

As an immediate consequence of Theorems \ref{thZ21} and \ref{thQpos} above, we obtain the following SSC result.

\begin{coro} {$($SSC of the service process$)$} \label{thSSCweak}
For $\tau > 0$ in Theorem {\em \ref{thZ21}}, $P(\T^n_0 > \tau) \ra 1$ as $n \tinf$, where $\T^n_0$ is defined in {\em \eqn{stopTau}}; i.e.,
\bes
(m^n_1 e - Z^n_{1,1},\ Z^n_{2,1},\ m^n_2 e - Z^n_{1,2} - Z^n_{2,2}) \Rightarrow (0 e,0 e,0 e) \quad \mbox{in $\D_3([0, \tau])$ as $n \tinf$.}
\ees
\end{coro}

We make two important remarks about Corollary \ref{thSSCweak}:
First, the limit holds {\em without any scaling}.  Second,
here we do not yet show that a limit of $\barz^n_{1,2}$ as $n \tinf$ exists. We only show that, when analyzing the four service processes $Z^n_{i,j}$,
it is sufficient to consider $Z^n_{1,2}$.


Recall that $d_{J_1}$ denotes the standard Skorohod $J_1$ metric and $X^{n,*}_6$ is the essentially three-dimensional process defined in \eqref{red1}.
The following corollary is immediate from Corollary \ref{thSSCweak}.
\begin{coro}{$($Representation via SSC$)$} \label{thSSC1}
As $n \tinf$, $d_{J_1}(X^n_6, X^{n,*}_6) \Ra 0$ in $\D_6([0, \tau])$, for $X^{n,*}_6$ in {\em \eqref{red1}}, $\tau$ in Theorem {\em \ref{thZ21}} and
$(Q^{n,*}_1, Q^{n,*}_2,Z^{n,*}_{1,2})$ in \eqref{rep4}-\eqref{rep3}. 
\end{coro}

We now obtain further simplification using a familiar martingale representation, again see \cite{PTW07}.
{\em Henceforth, we work with the process $X^{n,*}_6$ defined in Corollary {\em \ref{corRep3}}, but omit the asterisks.}
Consider the representation of $X^n_3$ in \eqref{rep4}-\eqref{rep3} above, and let
\bequ \label{martg}
\bsplit
M^{n,a}_{i}(t) & \equiv N^a_i(\lm^n_i t) - \lm^n_i t, \\
M_i^{n,u}(t) & \equiv N^u_i \left( \theta_{i} \int_{0}^{t} Q^{n}_{i} (s)) \, ds \right) - \theta_{i} \int_{0}^{t} Q^{n}_{i} (s) \, ds, \\
M^{n,s}_{i,2}(t) & \equiv N^s_{i,2}(J^n_{i,2}(t)) - J^n_{i,2}(t), \\
M^{n,s,2}_{i,2}(t) & \equiv N^{s,2}_{i,2}(J^n_{i,2}(t)) - J^n_{i,2}(t),
\end{split}
\eeq
where $J^n_{i,2}(t)$ are the compensators of the point processes
in \eqref{rep4}-\eqref{rep3}, $i = 1,2$, e.g.,
\bes
J^n_{1,2}(t) \equiv \mu_{1,2} \int_{0}^{t} 1_{\{D^n_{1,2}(s) < 0 \}}  Z^{n}_{1,2} (s) \, ds.
\ees
The quantities in \eqref{martg} can be shown to be martingales (with respect to an appropriate filtration); See \cite{PTW07}.
The following lemma follows easily from the functional strong law of large numbers (FSLLN) for Poisson processes and the $\C$-tightness
established in Theorem \ref{lmTight}.
\begin{lemma} {$($fluid limit for the martingale terms$)$} \label{lmMartgFluid}
As $n \tinf$,
\bes
n^{-1} (M^{n,a}_1, M^{n,a}_2, M^{n,u}_{1}, M^{n,u}_{2}, M^{n,s}_{1,2}, M^{n,s}_{2,2}, M^{n,s,2}_{1,2}, M^{n,s,2}_{2,2}) \Ra (0e,0e,0e,0e,0e,0e, 0e, 0e) \qinq \D_8 ([0, \tau]).
\ees
\end{lemma}

\begin{proof}  By Theorem \ref{lmTight}, the sequence $\{\bar{X}^n_6: n \ge 1\}$ is tight in $\D$.
Thus any subsequence has a convergent subsequence.
By the proof of Theorem \ref{lmTight}, the sequences $\{J^n_{i,j}/n\}$ are also $\C$-tight, so that $\{J^n_{i,j}/n\}$, $i = 1,2$,
all converge along a converging subsequence as well.
Consider a converging subsequence $\{X^n\}$ and its limit $\barx$, which is continuous by Theorem \ref{lmTight}.
Then the claim of the lemma follows for the converging subsequence from the FSLLN for Poisson processes
and the continuity of the composition map at continuous limits, e.g., Theorem 13.2.1 in \cite{W02}.
In this case, the limit of each fluid-scaled martingale is the zero function $0e \in \D$,
regardless of the converging subsequence we consider, and is thus unique.  Hence we have completed the proof.
\end{proof}

Hence, instead of $\bar{X}^n_3$ (the relevant components of $\bar{X}^{n,*}_6$) in \eqn{rep4}-\eqn{rep3},
we can work with
 $\bar{X}^n \equiv (\barq^n_1, \barq^n_2, \barz^n_{1,2})$ for
\bequ \label{FluidScaled}
\bsplit
\barz_{1,2}^n(t) & \equiv \barz^n_{1,2}(0) + \mu_{2,2} \int_{0}^{t} 1_{\{D^n_{1,2}(s) > 0\}} (\bar{m}^{n}_{2} - \barz^{n}_{1,2} (s)) \, ds
 - \mu_{1,2} \int_{0}^{t} 1_{\{D^n_{1,2}(s) \le 0 \}}  \barz^{n}_{1,2} (s) \, ds, \\
\barq^n_1(t) & \equiv \barq^n_1(0) + \bar{\lm}^n_1 t - \bar{m}^n_1 t - \mu_{1,2} \int_{0}^{t} 1_{\{D^n_{1,2}(s) > 0 \}} \barz^{n}_{1,2} (s)) \, ds \\
& \quad - \mu_{2,2} \int_{0}^{t} 1_{\{D^n_{1,2}(s) > 0 \}} (\bar{m}^n_2 - \barz^{n}_{1,2} (s)) \, ds
- \theta_{1} \int_{0}^{t} \barq^{n}_{1} (s) \, ds, \\
\barq^n_2(t) & \equiv \barq^n_2(0) + \bar{\lm}^n_2 t - \mu_{2,2} \int_{0}^{t} 1_{\{D^n_{1,2}(s) \le 0 \}} (\bar{m}^n_2 - \barz^{n}_{1,2} (s)) \, ds \\
& \quad - \mu_{1,2} \int_{0}^{t} 1_{\{D^n_{1,2}(s) \le 0 \}} \barz^{n}_{1,2} (s) \, ds
- \theta_{2} \int_{0}^{t} \barq^{n}_{2} (s)) \, ds.
\end{split}
\eeq

\begin{theorem} \label{thFluidScaled} As $n \ra \infty$,
$d_{J_1}(\bar{X}^n_3, \bar{X}^n) \Ra 0$ in $\D_3([0, \tau])$ as $n \tinf$, where $\barx^n_3$ is defined in {\em \eqn{rep4}}-{\em \eqn{rep3}},
 $\bar{X}^n$ is defined in {\em \eqn{FluidScaled}} and $\tau$ is as in Theorem {\em \ref{thZ21}}.
\end{theorem}

\begin{proof}
It suffices to show that $\bar{M}^n \Rightarrow (0 e, 0 e, 0e)$ in $\D_3([0, \tau])$ as $n \ra \infty$, where
\beql{rep101}
\bar{M}^n \equiv \bar{X}^n_3 - \barx^n.
\eeq
However, $\bar{M}^n \Rightarrow (0 e, 0 e, 0e)$ by virtue of Lemma \ref{lmMartgFluid} above
by the continuous mapping theorem with addition at continuous limits.
\end{proof}

As a consequence of Theorem \ref{thFluidScaled}, henceforth we can focus on $\bar{X}^n$ in \eqn{FluidScaled}
instead of $\bar{X}^n_3$ in \eqn{rep4}-\eqn{rep3}.
After we prove Theorem \ref{th1}, we can extend the interval $[0, \tau]$ over which all the previous results in this section hold
to the interval $[0, \infty)$.
\begin{theorem} {$($global SSC$)$} \label{thSSCextend}
All of the results above in this section extend from the interval $[0,\tau]$ to the interval $[0, \infty)$.
\end{theorem}

We prove Theorems \ref{thZ21} and \ref{thQpos} in \S \ref{secProofsSSCserv},
after establishing supporting technical results in \S \ref{secSupport}.
In \S \ref{secGlobalProofs}, we prove Theorem \ref{thSSCextend}, under the assumption that Theorem \ref{th1} has been proved over $[0, \tau]$,
which will be done in \S \ref{secProofs}, by showing that the interval over which the conclusion is valid
can be extended from $[0, \tau]$ to $[0, \infty)$,
once Theorem \ref{th1} has been proved over $[0,\tau]$.
That will imply that Theorem \ref{th1} then holds over $[0, \infty)$ as well.

\section{Completing the Proof of Theorem \ref{th1}:  Characterization.}\label{secProofs}

The $\C$-tightness result in Theorem \ref{lmTight} implies that every subsequence of the sequence $\{(\bar{X}^n_6, \bar{Y}^n_8): n \ge 1\}$
in \eqn{scale} has
a further converging subsequence in $\D_{14}([0, \infty))$, whose limit is in the function space $\C_{14}([0, \infty))$.
To establish the convergence of the sequence $\{(\bar{X}^n_6, \bar{Y}^n_8)\}$
we must show that every converging subsequence
converges to the same limit.

Corollary \ref{thSSC1} and Theorem \ref{thFluidScaled} imply that we can simplify the framework over an initial time interval $[0, \tau]$.
In particular,
it suffices to focus on the sequence $\{\bar{X}^n\}$ in $\D_3([0, \tau])$ given in \eqn{FluidScaled}, where the limits of the subsequences will be in $\C_3 ([0, \tau])$,
but we will later show that this restriction can be relaxed.
In particular, we will show that convergence holds in $\D_{14}([0, \infty))$ 

We achieve the desired characterization of $\barx^n$ in $\D_{3}([0, \tau])$ by showing that
 the limit of any converging subsequence is almost surely a solution to
the ODE \eqref{odeDetails} over $[0, \tau]$.
The existence and uniqueness of the solution to the ODE has been established
in Theorem 5.2 in \cite{PeW09c}.
 As indicated in \S \ref{secProofOver}, we have developed two different proofs,
 the first proof exploits random measures and martingales, following \cite{HK94,K92}; it is given here in \S\S \ref{secHK} and \ref{secHK2}.
 The second exploits stochastic bounds as in the proof of Lemma \ref{lmD21extreme}; it is given in Appendix \ref{secAltProof}.
 Both proofs start from the following subsection.

\subsection{Proof of Theorem \ref{th1}:  Reduction to Integral Terms.} \label{secProofAP}

Let $\bar{X}$ be the limit of a converging subsequence of $\{\barx^n: n \ge 1\}$ in \eqn{FluidScaled}
in $\D([0,\tau])$ for $\tau$ in Theorem \ref{thFluidScaled}.
We consider $n \ge 1$ with the understanding that the limit is through a subsequence.
  Many of the terms in \eqn{FluidScaled} converge directly to their counterparts in \eqn{FluidScaledLim}
because of the assumed many-server heavy-traffic scaling in \S \ref{secHT} and the convergence $\bar{X}^n \Rightarrow \barx$ through the subsequence
obtained from the tightness.  Indeed, the only exceptions are the integral terms involving the
indicator functions. Let $\bar{I}^n_{z,i}$, $\bar{I}^n_{q,1,i}$, and $\bar{I}^n_{q,2,i}$ be the $i^{\rm th}$ integral term in the
respective expression for $\barz^n_{1,2}$, $\barq^n_{1}$ and $\barq^n_{2}$ in \eqn{FluidScaled}, $i = 1,2$.
We first observe that these sequences of integral terms are tight.
\begin{lemma}{$($tightness of integral terms$)$}\label{lmIntTight}
The six sequences of integral processes $\{\bar{I}^n_{z,i}: n \ge 1\}$, $\{\bar{I}^n_{q,1,i}: n \ge 1\}$, and $\{\bar{I}^n_{q,2,i}: n \ge 1\}$
involving the indicator functions appearing in {\em \eqn{FluidScaled}} are each $\C$-tight in $\D ([0,\tau])$.
\end{lemma}
\begin{proof}  We consider only the integral term $\bar{I}^n_{q,1,1}$, because the others are treated in the same way.
First, boundedness is elementary:  $0 \le \bar{I}^n_{q,1,1} (t) \le t m^n_2/n$.
Second, the modulus in \eqn{modulus} is easily controlled:
$w(\bar{I}^n_{q,1,1},\zeta,T) \le \zeta m^n/n \ra \zeta m_2.$
\end{proof}
Hence, we can consider a subsequence of our original converging subsequence in which all these integral terms converge
to proper limits as well.  Hence we have the following expression for $\bar{X}$, the limit of the converging subsequence.
\bequ \label{FS}
\bsplit
\barz_{1,2}(t) & = z_{1,2}(0) + \mu_{2,2} \bar{I}_{z,1} (t) - \mu_{1,2} \bar{I}_{z,2} (t) \\
\barq_1(t) & = q_1(0) + \bar{\lm}_1 t - \bar{m}_1 t - \mu_{1,2} \bar{I}_{q,1,1} (t) - \mu_{2,2} \bar{I}_{q,1,2} (t)
- \theta_{1} \int_{0}^{t} \barq_{1} (s) \, ds, \\
\barq_2(t) & = q_2(0) + \bar{\lm}_2 t - \mu_{2,2} \bar{I}_{q,2,1} (t) - \mu_{1,2} \bar{I}_{q,2,2} (t)
- \theta_{2} \int_{0}^{t} \barq_{2} (s)) \, ds.
\end{split}
\eeq
In \eqn{FS}, we have exploited the assumed convergence of the initial conditions in Assumption \ref{assC} to
replace $\bar{X} (0)$ by $x(0)$ in \eqn{FS}.
At this point, it only remains to show that the terms $\bar{I}_{z,i}$, $\bar{I}_{q,1,i}$ and $\bar{I}_{q,2,i}$
appearing in \eqn{FS} necessarily coincide almost surely with the corresponding terms in the integral representation \eqn{FluidScaledLim}
associated with the ODE in \eqn{odeDetails} over the interval $[0,\tau]$.
That will uniquely characterize the limit over that initial interval $[0,\tau]$
because, by Theorem 5.2 of \cite{PeW09c}, there exists a unique solution to the ODE.

Thus, it suffices to establish the following lemma, which we do in two different ways, one in the next two sections
and the other in Appendix \ref{secAltProof}.
\begin{lemma}{$($representation of limiting integral terms$)$}\label{lmKey}
For $\tau$ in Theorem {\em \ref{thFluidScaled}},
the integral terms in {\em \eqn{FS}} necessarily coincide with the corresponding integral terms in {\em \eqn{FluidScaledLim}}
with $\barx$ substituted for $x$ for $0 \le t \le \tau$,
e.g.,
\bequ \label{key}
\bar{I}_{q,1,1} (t) = \int_{0}^{t} \pi_{1,2}(\barx (s)) \barz_{1,2} (s)) \, ds, \quad 0 \le t \le \tau, \quad w.p.1.
\eeq
\end{lemma}

\subsection{Exploiting Random measures.}\label{secHK}

For the rest of our proof of
 Lemma \ref{lmKey}, we closely
follow Hunt and Kurtz \cite{HK94}, which applies Kurtz \cite{K92}.  There are two key steps:  (i) exploiting random measures
and (ii) applying a martingale representation to characterize the limit in terms of the steady-state distribution of the FTSP in \S \ref{secFTSP}.
This subsection is devoted to the first step; the next subsection is devoted to the second step.

We now introduce random measures in order to expose additional structure in the integral term $\bar{I}_{q,1,1}$ in \eqn{FS}.
The random measures will be defined by setting
\bequ \label{rm1}
\nu^n ([0,t] \times B) \equiv \int_{0}^{t} 1_{\{D^n_{1,2} (s) \in B\}} \, ds, \quad t \ge 0,
\eeq
where $B$ is a measurable subset of the set $E$ of possible values of $D^n_{1,2}$.
Given the definition of $D^n_{1,2}$ in \eqn{Dprocess}, we exploit Assumption \ref{assE} stating that $r_{1,2}$ is rational and the assumption after
it that the thresholds are rational as well, so that
we can have the state space $E$ be discrete, independent of $n$.
(This property is not necessary for the analysis at this point, but it is a helpful simplification.
The space $E$ could even be taken to be a subset of $\ZZ$, using the construction in \S 6 of \cite{PeW09c} by renaming the states in $E$.)
Of course, we are especially interested in the case in which $B$ is the set of positive values;
then we focus on the associated random variables $\nu^n ([0,t] \times (0, \infty))$, $n \ge 1$, as in \eqn{FluidScaled}.

As in \cite{HK94,K92}, it is convenient to compactify the space $E$.
We do that here , first, be adding the states $+\infty$ and $-\infty$ to $E$ and, second, by endowing $E$ with
the metric and associated Borel $\sigma$-field from $\RR$ induced by the mapping $\psi: E \ra [-1,1]$,
defined by $\psi (x) = x/(1 + |x|)$.  That makes $E$ a compact metric space.
We then consider the space $M \equiv M(S)$ of (finite) measures $\mu$ on the product space $S \equiv [0, \delta] \times E$
for some $\delta > 0$, such that $\mu([0,t] \times E) = t$
for all $t > 0$.  Moreover, we endow $M$ with the Prohorov metric, as in (1.1) of \cite{K92}.  Since $S$ is compact, the space $M$ inherits the compactness;
i.e., it too is a compact metric space, by virtue of Prohorov's theorem, Theorem 11.6.1 of \cite{W02}.

Let $\sP \equiv \sP(M) \equiv \sP(M(S))$ be the space of probability measures on $M(S)$,
also made into a metric space with the Prohorov metric, so that convergence corresponds to the usual notion of weak convergence of probability measures.
As a compact metric space, $M(S)$ is a complete separable metric space, so that this is a standard abstract setting for
weak convergence of probability measures \cite{B99,EK86,W02}.
By Prohorov's theorem, this space $\sP$ of probability measures on $M(S)$ also is a compact metric space.

Thus we have convergence of random measures $\nu^n \Ra \nu$ as $n \ra \infty$ if and only if $E[f(\nu^n)] \ra E[f(\nu)]$
as $n \ra \infty$ for all continuous bounded real-valued functions $f$ on $M$.
On the other hand, by the continuous mapping theorem, if we have $\nu^n \Ra \nu$ as $n \ra \infty$, then we also have
$f(\nu^n) \Ra f(\nu)$ as $n \ra \infty$ for each continuous function on $M$.
One reason that the random measure framework is convenient is that
each continuous bounded real-valued function $f$ on $S$ corresponds to a continuous real-valued function on $M(S)$ via the integral representation
\beq
f(\mu) \equiv \int_{S} f(s) \, d \mu (s).
\eeqno
As a consequence, if $\nu^n \Ra \nu$ as $n \ra \infty$, then necessarily also
\beq
\int_{S} f(s) \, d \nu^n (s) \Ra \int_{S} f(s) \, d \nu (s) \qinq \RR \qasq n \ra \infty
\eeqno
for all continuous bounded real-valued functions $f$ on $S$.

In our context it is important to observe what are the continuous functions on $E$
after the compactification above.
The new topology requires that the functions have finite limits as $k \ra +\infty$ or as $k \ra -\infty$.
All the functions we consider will be continuous on $E$ because
they take constant values
outside bounded sets.
We will use the following stronger result, which is a special case of Lemma 1.5 of \cite{K92}.
\begin{lemma}{$($extended continuous mapping theorem$)$}\label{lmCMT}
If $f$ is a continuous bounded real-valued function on $S$ and $\{f_n: n \ge 1\}$ is a sequence of measurable real-valued functions on $S$
such that $\| f_n - f\|_{S} \ra 0$ as $n \ra \infty$, then
\beq
\int_{S} f_n (s) \, d \nu^n (s) \Ra \int_{S} f(s) \, d \nu (s) \qinq \RR \qasq n \ra \infty.
\eeqno
\end{lemma}

\begin{proof}  First, starting from the convergence $\nu^n \Ra \nu$, apply the Skorohod representation theorem to obtain versions converging w.p.1,
without changing the notation.  Then,
by the triangle inequality,
\bes
\bsplit
| \int_{S} f_n (s) \, d \nu^n (s) - \int_{S} f(s) \, d \nu (s)| & \le  | \int_{S} f_n (s) \, d \nu^n (s) - \int_{S} f(s) \, d \nu^n (s)| 
+  | \int_{S} f (s) \, d \nu^n (s) - \int_{S} f(s) \, d \nu (s)| \\
& \le \| f_n - f\|_{S} \nu^n (S) + | \int_{S} f (s) \, d \nu^n (s) - \int_{S} f(s) \, d \nu (s)| \ra 0 ~\mbox{ as } n \ra \infty, \nonumber
\end{split}
\ees
using the uniform convergence condition and the limit $\nu^n (S) \ra \nu(S) < \infty$ in the first term and the
continuous mapping in the second term.  This shows convergence w.p.1 for the special versions and thus the claimed convergence in distribution.
\end{proof}

The random measures we consider are random elements of $M(S)$, where $S$ is the product space $[0, \delta] \times E$.
These random measures can be properly defined by giving their definition for all product sets $[0,t] \times B$,
as we have done in \eqn{rm1}.  The usual extension produces full random elements of $M(S)$.

We exploit the compactness of $E$ introduced above to obtain compactness of
$\sP(M)$ and thus relative compactness of the sequence $\{(\barx^n_6, \nu^n); n \ge 1\}$.
\begin{lemma}{$($relative compactness$)$}\label{lmRelC}
The sequence $\{(\barx^n_6, \nu^n); n \ge 1\}$ defined by {\em \eqn{scale}} and {\em \eqn{rm1}}
is relatively compact in $\D_6 ([0, \delta]) \times M(S)$ for any $\delta > 0$.
\end{lemma}

\begin{proof}
We already have observed that, because of the compactness imposed on $E$, the space $\sP(M(S))$ is compact.
By Theorem \ref{lmTight}, the sequence $\{\barx^n_{1,2}: n \ge 1\}$ is tight and thus relatively compact.
However, relative compactness of the components implies relative compactness of the vectors.
Thus the sequence $\{(\barx^n_6, \nu^n); n \ge 1\}$ defined by \eqn{scale} and \eqn{rm1}
is relatively compact in $\D_6 ([0, \delta] \times M(S)$.
\end{proof}

Another crucial property of the random measures on the product space
is that the random measures themselves admit a product representation or factorization,
as indicated by Lemma 1.4 of \cite{K92}; also see Lemma 2 of \cite{HK94}.
This result requires filtrations.  For that, we observe that $X^n_6$ is a Markov process
and $\nu_n$ restricted to $[0,t] \times E$ is a function of the Markov process $X^n_6$
over $[0,t]$.  Thus, we can use the filtrations $\sF^n_t$ generated by $X^n_6$ for each $n \ge 1$.
In our context, we have the following consequence of Lemma 1.4 of \cite{K92}.
\begin{lemma}{$($factorization of the limiting random measures$)$}\label{lmKurtzRep}
Let $(\barx, \nu)$ be the limit of a converging subsequence of
$\{(\barx^n_6, \nu^n); n \ge 1\}$ in $\D_6 ([0, \delta]) \times M(S)$
obtained via Lemma {\em \ref{lmRelC}}.  Then there exists $p_s \equiv p_s (B)$,
a measurable function of $s$ for each measurable subset $B$ in $E$
and a probability measure on $E$ for each $s$ in $[0,\delta]$,
such that, for all measurable subset $B_1$ of $[0, \delta]$
and $B_2$ of $E$,
\bequ \label{rmRep}
\nu (B_1 \times B_2) = \int_{B_1} p_s (B_2) \, ds.
\eeq
\end{lemma}

As a consequence of the three lemmas above, we obtain the following preliminary representation.
\begin{lemma}{$($initial representation of limiting integral terms$)$}\label{lmKey2}
Every subsequence of the sequence $\{(\barx^n_6, \nu^n); n \ge 1\}$ defined by {\em \eqn{scale}} and {\em \eqn{rm1}}
in $\D_6 ([0, \delta]) \times M(S)$ as a further converging subsequence.
Let $(\barx, \nu)$ be a limit of a convergent subsequence.
For any $\delta \le \tau$ for $\tau$ in Theorem {\em \ref{thFluidScaled}},
the integral terms in {\em \eqn{FS}} necessarily coincide with the corresponding integral terms in {\em \eqn{FluidScaled}}
with a probability $p_{1,2} (s)$ substituted for $1_{\{D^n_{1,2} (s) > 0\}}$ and $\barx$ substituted for $x$ for $0 \le t \le \delta$,
in particular,
\bequ \label{key2}
\bar{I}_{q,1,1} (t) = \int_{0}^{t} p_{1,2}(s) \barz_{1,2} (s)) \, ds, \quad 0 \le t \le \delta, \quad w.p.1.
\eeq
where $\barz_{1,2}$ is the component of $\barx$ and $p_{1,2} (s) \equiv p_s ((0, \infty))$ for $p_s$ in {\em \eqn{rmRep}},
so that $p_{1,2} (s)$ is a measurable function of $s$ with $0 \le p_{1,2} (s) \le 1$, $0 \le s \le \delta$.
\end{lemma}

\begin{proof}
By Lemma \ref{lmRelC}, we are justified in focusing on a converging subsequence with limit $(\barx, \nu)$,
where $\barx \equiv (\barq_{i}, \barz_{i,j})$.  For \eqn{key2}, we focus on $(\barz_{1,2}, \nu)$.
As before, apply the Skorohod representation theorem to obtain a version converging w.p.1 along the subsequence,
without changing the notation.
For the corresponding terms indexed by $n$,
\bequ \label{rm5}
\bar{I}^n_{q,1,1} (t) = \mu_{1,2} \int_{0}^{t} \barz^n_{1,2} (s) 1_{\{D^n_{1,2} (s) > 0\}} ds = \int_{0}^{t} \barz^n_{1,2} (s) \, \nu^n (ds \times dy),
\eeq
so that we can apply Lemma \ref{lmCMT} to deduce that
\beq
\bar{I}^n_{q,1,1} (t) \ra \bar{I}_{q,1,1} (t) \equiv \mu_{1,2} \int_{0}^{t} \barz_{1,2} (s) \, \nu (ds \times dy).
\eeqno
Finally, we apply Lemma \ref{lmKurtzRep} to show that the representation of $\nu$ in \eqn{rm5}
is equivalent to \eqn{key2}.
\end{proof}

It now remains to determine the term $p_{1,2} (s)$ in the integrand of the integral \eqn{key2}.
In the next section we will show that we can write $p_{1,2} (s) = P(D(\barx (s), \infty) > 0)$,
thus completing the proof of Lemma \ref{lmKey}.

\subsection{A Martingale Argument to Characterize the Probability in the Integrand.}\label{secHK2}

We now finish the proof of Lemma \ref{lmKey} by characterizing the probability measure $p_s$ in Lemma \ref{lmKurtzRep}.

\begin{proofof}{Lemma \ref{lmKey}}
We will prove that $p_{1,2}(s) = P(D(\barx (s), \infty)) > 0$ for almost all $s$
in the integral \eqn{key2}, where $D(\barx (s), \infty))$ is a random variable with the steady-state distribution of the FTSP in
\S \ref{secFTSP} depending on the state $\barx (s)$, which is the limit of the converging subsequence of the sequence $\{\barx^n: n \ge 1\}$.
This step will make \eqn{key2} reduce to the desired \eqn{key}.

We first comment on the exceptional sets.  We establish the result w.p.1, so that there is an exceptional set, say $\Upsilon$,
in the underlying probability space $\Omega$ with $P(\Upsilon) = 0$, such that we claim the conclusion
of Lemma \ref{lmKey} holds in $\Omega - \Upsilon$.
However, for each $\omega \in \Omega - \Upsilon$, we find an exceptional set $\Psi (\omega)$ in $[0, \delta]$
where the Lebesgue measure of $\Psi (\omega)$ is $0$.  However, the integral in Lemma \ref{lmKey} is unchanged if
we change the definition of the integrand on a set of Lebesgue measure $0$.
Hence, we can assume that $p_{1,2}(s) = P(D(\barx (s), \infty) > 0)$ for {\em all} $s$ in $[0,\delta]$
for each sample point $\omega \in \Omega - \Upsilon$.  After doing that, we obtain the w.p.1 conclusion
in Lemma \ref{lmKey}.

We remark that it is possible to obtain a {\em single} exceptional set $\Psi$ in $[0, \delta]$ such that
$p_{1,2}(s) = P(D(\barx (s), \infty)) > 0$ for all $s \in [0,\delta] - \Psi$ w.p.1.  The construction to achieve that stronger goal is described
in Example 2.3 of Kurtz \cite{K92} on p. 196.  The conditions specified there hold in our context.
Since that property is not required here, we do not elaborate.

Continuing with the main proof, we now aim to characterize the
 entire probability measure $p_s$ on $E$ appearing in Lemma \ref{lmKurtzRep}.
We will do that by showing that $p_s$ satisfies the
equation characterizing the steady-state distribution of the FTSP $D(\barx (s), \cdot)$ for almost all $s$ with respect to Lebesgue measure
(consistent with the notion of an averaging principle).
Since the FTSP $D(\barx (s), \cdot)$, given $\barx (s)$, is a CTMC with the special structure (only four transitions possible from each state,
and only two different cases for these, as shown in \eqn{bd1}-\eqn{bd4}), just as for finite-state CTMC's (in elementary textbooks), it suffices to show that
\bequ \label{mg1}
\sum_{i} p_s (\{i\}) Q_{i,j} (\barx (s)) = 0 \qforallq j
\eeq
for almost all $s$ in $[0, \delta]$ with respect to Lebesgue measure, where $i$ and $j$ are states of the FTSP
and $Q_{i,j} (\barx (s))$ in \eqn{mg1} is the $(i,j)^{th}$ component in the infinitesimal rate matrix (generator)
of the CTMC $D(\barx (s), \cdot)$.

However, we will follow \cite{HK94} and use the framework in \S\S 4.2 and 8.3 of \cite{EK86}.  In particular,
the FTSP satisfies the assumptions in Corollary 8.3.2 on p. 379 in \cite{EK86}. As in (9) of \cite{HK94},
this step corresponds to an application of Proposition 4.9.2 of \cite{EK86}, but the simple CTMC setting
does not require all the structure there.
Following the proof of Theorem 3 in \cite{HK94} (and \S 2 of \cite{K92}),
we now develop a martingale representation for $f(D^n_{1,2})$, where $f$ is a bounded continuous real-valued function on the state space $E$ of $D^n_{1,2}$.
This construction is the standard martingale associated with functions of Markov processes,
just as in Proposition 4.1.7 of \cite{EK86}.  Since, $D^n_{1,2}$ is a simple linear function of the CTMC $X^n_6$ in \eqn{Dprocess},
we can write $f(D^n_{1,2}) = g(X^n_6)$ for some continuous bounded function $g$.  The martingale will be with respect to the filtration $\mathcal{F}^n_t$
generated by the Markov process $X^n_6$ (as in Lemma \ref{lmKurtzRep}).

Recalling that $D^n_{1,2} \equiv Q^n_1 - r_{1,2} Q^n_2$, we can write $f(D^n_{1,2}(t))$ in terms of the
independent rate-$1$ Poisson processes in \eqref{rep1} as follows
\bes
\bsplit
f(D^n_{1,2}(t)) & \equiv  f(D^n_{1,2}(0)) - \int_0^t [f(D^n_{1,2}(s-)+1) - f(D^n_{1,2}(s-))] dN^a_{1}(\lm^n_1 s) \\
& \quad - \int_0^t [f(D^n_{1,2}(s-)-1) - f(D^n_{1,2}(s-))] \Big(dN^u_1(\theta_1 Q^n_1(s)) + dN^s_{1,1}(\mu_{1,1}m^n_1 s)\Big) \\
& \quad - \int_0^t 1_{\{D^n_{1,2}(s) > 0\}}[f(D^n_{1,2}(s-)-1) - f(D^n_{1,2}(s-))] \Big(dN^s_{1,2}(\mu_{1,2}Z^n_{1,2}(s))+dN^s_{2,2}(\mu_{2,2}(Z^n_{2,2}(s)))\Big) \\
& \quad - \int_0^t 1_{\{D^n_{1,2}(s) \le 0\}} [f(D^n_{1,2}(s-)+r) - f(D^n_{1,2}(s-))] \Big(dN^s_{1,2}(\mu_{1,2}Z^n_{1,2}(s))+dN^s_{2,2}(\mu_{2,2}(Z^n_{2,2}(s)))\Big) \\
& \quad - \int_0^t 1_{\{D^n_{1,2}(s) > 0\}} [f(D^n_{1,2}(s-)+r) - f(D^n_{1,2}(s-))] dN^u_s(\theta_2 Q^n_2(s)) \\
& \quad - \int_0^t [f(D^n_{1,2}(s-)-r) - f(D^n_{1,2}(s-))] dN^a_2(\lm^n_2 s).
\end{split}
\ees
We next rewrite the representation for $f(D^n_{1,2}(t))$ above to achieve a martingale representation. To that end,
we add and then subtract the appropriate Riemann integral from each of the integrals above, e.g.,
\bes
\bsplit
& \int_0^t [f(D^n_{1,2}(s-)+1) - f(D^n_{1,2}(s-))] dN^a_{1}(\lm^n_1 s) = \\
& \quad \int_0^t [f(D^n_{1,2}(s-)+1) - f(D^n_{1,2}(s-))] dM^{n,a}_{1}(s) + \int_0^t [f(D^n_{1,2}(s-)+1) - f(D^n_{1,2}(s-))]\lm^n_1 ds,
\end{split}
\ees
for $M^{n,a}_{1}$ in \eqref{martg}. Note that an integral of a predictable process with respect to a martingale is again a martingale.
We thus achieve a modifies representation of $f(D^n_{1,2}(t))$ in terms of martingales and their associated predictable quadratic-variation processes;
see, e.g., \S 3.5 in \cite{PTW07} for more details.
Rearranging terms, so that all the martingales in the modified representation appear on the left-hand side and letting $M^n_f$ denote the
sum of those martingales, we have
\begin{eqnarray}\label{mg2}
M^n_f(t) & \equiv & f(D^n_{1,2}(t))  - f(D^n_{1,2}(0)) - \int_0^t \lm^n_1 [f(D^n_{1,2}(s-)+1) - f(D^n_{1,2}(s-))] ds \nonumber  \\
&& \quad - \int_0^t (m^n_1 + \theta_1 Q^n_1(s))[f(D^n_{1,2}(s-)-1) - f(D^n_{1,2}(s-))] ds \nonumber \\
&& \quad - \int_0^t 1_{\{D^n_{1,2}(s) > 0\}} (\mu_{1,2}Z^n_{1,2}(s) + \mu_{2,2}(m^n_2 - Z^n_{1,2}(s)))[f(D^n_{1,2}(s-)-1) - f(D^n_{1,2}(s-))] ds \nonumber \\
&& \quad - \int_0^t 1_{\{D^n_{1,2}(s) \le 0\}} (\mu_{1,2}Z^n_{1,2}(s) + \mu_{2,2}(m^n_2 - Z^n_{1,2}(s))) [f(D^n_{1,2}(s-)+r) - f(D^n_{1,2}(s-))] ds \nonumber \\
&& \quad - \int_0^t 1_{\{D^n_{1,2}(s) > 0\}} \theta_2 Q^n_2(s) [f(D^n_{1,2}(s-)+r) - f(D^n_{1,2}(s-))] ds \nonumber \\
&& \quad - \int_0^t \lm^n_2 [f(D^n_{1,2}(s-)-r) - f(D^n_{1,2}(s-))] ds.
\end{eqnarray}
Note that $M^n_f$ is a $\mathcal{F}^n_t$-martingale itself.

It follows from essentially the same arguments as in the proof of Lemma \ref{lmMartgFluid} and the continuity of addition at continuous limits, that
$\bar{M}^n_f \equiv M^n_f/n \Ra 0e$ in $\D$ as $n \tinf$, for $M^n_f$ in \eqref{mg2}.
In addition, we have $n^{-1}(f(D^n_{1,2}(t)) - f(D^n_{1,2}(0))) \Ra 0e$ in $\D$ since $f$ is bounded.
We write the remaining terms of $\bar{M}^n_f$ as
\bes 
\bsplit
& \bar{D}^n_f (t) \equiv \bar{M}^n_f (t) - n^{-1}(f(D^n_{1,2}(t)) - f(D^n_{1,2}(0)))
  = \int_{(0,t)\times E} \bar{\lm}^n_1 [f(y+1) - f(y)] \nu^n(ds \times dy) \\
& \quad - \int_{(0,t)\times E} (\bar{m}^n_1 + \theta_1 \barq^n_1(s))[f(y-1) - f(y)] \nu^n(ds \times dy) \\
& \quad - \int_{(0,t)\times E} 1_{\{y > 0\}} (\mu_{1,2}\barz^n_{1,2}(s) + \mu_{2,2}(\bar{m}^n_2 - \barz^n_{1,2}(s)))[f(y-1) - f(y)] \nu^n(ds \times dy) \\
& \quad - \int_{(0,t)\times E} 1_{\{y \le 0\}} (\mu_{1,2}\barz^n_{1,2}(s) + \mu_{2,2}(\bar{m}^n_2 - \barz^n_{1,2}(s))) [f(y+r) - f(y)] \nu^n(ds \times dy) \\
& \quad - \int_{(0,t)\times E} 1_{\{y > 0\}} \theta_2 \barq^n_2(s) [f(y+r) - f(y)] \nu^n(ds \times dy)
 - \int_{(0,t)\times E} \bar{\lm}^n_2 [f(y-r) - f(y)] \nu^n(ds \times dy). \\
\end{split}
\ees
By Lemmas \ref{lmKey2} and \ref{lmCMT}, $\bar{D}^n_f \Ra \bar{D}_f = 0e$ as $n \ra \infty$ along the converging subsequence, where
\bequ \label{mg3}
\bsplit
& \bar{D}_f  \equiv \int_{(0,t)\times E_\Delta} \lm_1 [f(y+1) - f(y)] \nu (ds \times dy)
 - \int_{(0,t)\times E} (m_1 + \theta_1 \barq_1(s))[f(y-1) - f(y)] \nu (ds \times dy) \\
& - \int_{(0,t)\times E} 1_{\{y > 0\}} (\mu_{1,2}\barz_{1,2}(s) + \mu_{2,2}(m_2 - \barz_{1,2}(s)))[f(y-1) - f(y)] \nu (ds \times dy) \\
& - \int_{(0,t)\times E} 1_{\{y \le 0\}} (\mu_{1,2} \barz_{1,2}(s) + \mu_{2,2}(m_2 - \barz_{1,2}(s))) [f(y+r) - f(y)] \nu (ds \times dy) \\
& - \int_{(0,t)\times E} 1_{\{y > 0\}} \theta_2 \barq_2(s) [f(y+r) - f(y)] \nu (ds \times dy)
- \int_{(0,t)\times E} \lm_2 [f(y-r) - f(y)] \nu (ds \times dy). \\
\end{split}
\eeq
However, just as in \cite{HK94}, we can identify the limit $\bar{D}_f$ in \eqn{mg3}
as the integral with respect to the random measure $\nu$
of the infinitesimal generator of the FTSP $D(\barx(s), \cdot)$
applied to the test function $f$.  In particular, for each sample point in the underlying probability space $\Omega$ supporting $(\barx, \nu)$
except for a subset $\Upsilon$ with $P(\Upsilon) = 0$,
from Lemma \ref{lmKurtzRep}, we obtain
\bequ \label{mg4}
\int_{0}^{t} \int_{E} [Q (\barx (s))f](y) p_s(dy) \, ds = 0 \quad \mbox{for all $t \ge 0$.}
\eeq
where $[Q (\barx (s))f] (y)$ is the generator of the CTMC $D(\barx (s), \cdot)$ applied to $f$ as a function of $y$ in $E$.  As a consequence,
\bequ \label{mg5}
\int_{E} [Q (\barx (s)) f](y) p_s (dy) = 0 \quad \mbox{for almost all $s$ with respect to Lebesgue measure.}
\eeq
It follows from Proposition 4.9.2 page 239 in \cite{EK86}, that $p_s$ is the (unique) stationary distribution of the FTSP $D(\barx(s), \cdot)$
for almost all $s$.
(This step is equivalent to \eqn{mg1}.)
We apply Lemma \ref{lmLimFTSP} to conclude that the FTSP $D(\gamma, \cdot)$ has a unique stationary distribution on $E$ for all $\gamma \in \rS$.
\end{proofof}

\subsection{Proof of Theorem \ref{th1}.}

We can now summarize the proof of our main result - the FWLLN via the AP.

\begin{proofof}{Theorem \ref{th1}}
There are two steps:  (i) establishing convergence over an initial interval and (ii) expanding the interval of convergence.
We consider slightly more general settings than in the statement of the theorem, by considering the vector $(\barx^n_6, \bar{Y}^n_8, \nu^n)$
with the random measure $\nu^n$ replacing $\Theta^n$.
Note that ${\Theta}^n(t) \equiv \nu^n([0,t] \times (0,\infty))$ and $\vartheta(t) \equiv \nu([0,t] \times (0,\infty))$ for ${\Theta}^n$ and $\vartheta$
in \eqref{nuAP}, $\nu^n$ in \eqref{rm1} and $\nu$ in \eqref{rmRep}.

{(i) \em Establishing convergence over $[0,\tau]$.}
By Theorem \ref{lmTight}, the sequence $\{(\barx^n_6, \bar{Y}^n_8): n \ge 1\}$ in \eqref{scale} is $\C$-tight in $\D_{14} ([0, \infty))$.
By Lemma \ref{lmRelC}, $(\barx^n_6,\nu^n)$ is relative compact in $\D_6 \times M(S)$.
By Theorem \ref{thFluidScaled},
there exists $\tau > 0$
such that the limit point of a converging subsequence of $\barx^n_6$ in $\D_6([0, \tau])$ is also the limit point of $\barx^{n,*}_6$ in \eqref{red1},
whose representation is specified in \eqref{rep4}-\eqref{rep3}.
Thus, it suffices to next characterize the limit of a converging subsequence of the sequence $(\barx^n, \nu^n)$, for
$\{\barx^n\} \equiv \{1/n (Q^n_1, Q^n_2, Z^n_{1,2})\}$ in \eqref{FluidScaled} over an interval $[0, \tau]$.
The characterization of the limit of $\barx^n_6$ also characterizes the limit of $\nu^n$ since,
as the proof of Lemma \ref{lmKurtzRep} demonstrates,
each limit point of $(\barx^n_6,\nu^n)$ is of the form $(\barx_6, \nu)$, for $\nu$ in \eqref{rmRep}. In particular, for any two Borel sets $B_1, B_2$,
$$\nu^n(B_1, B_2) \equiv \int_{B_1} 1_{\{D^n_{1,2}(s) \in B_2\}} \Ra \nu(B_1 \times B_2) \equiv \int_{B_1} p_s(B_2)\, ds \qinq \D \qasq n \tinf,$$
where $p_s$ is the unique stationary distribution of the FTSP $D(\barx_6(s), \cdot)$ for almost all $s$.
Hence, if $\barx \equiv x_6$, and in particular, the limit of $\barx^n_6$ is unique, then $p_s \equiv \pi_{1,2}(x_6(s))$ for almost all $s$,
so that the limit $\nu$ of $\nu^n$ is unique as well.

With that characterization complete, we obtain the full convergence $(\barx^n_6, \bar{Y}^n_8 , \nu^n) \Ra (x_6,y_8,\nu)$ in $\D_{15} ([0, \tau])$ directly,
exploiting Theorem \ref{thFluidScaled}.  By Lemma \ref{lmKey} above, we complete the characterization step, showing that $P(\barx = x) = 1$ in $\D_3([0, \tau])$,
where $x$ is a solution to the ODE in \eqn{odeDetails} with the initial condition $x(0)$ specified by Assumption \ref{assC}.

{(ii) \em Expanding in the interval of convergence.}
After establishing the convergence over an initial interval $[0, \tau]$,
we can apply Theorem \ref{thSSCextend} (which uses Theorem \ref{th1} over $[0, \tau]$) to conclude that
any limit point of the tight sequence $\barx^n_6$ is again a limit of the tight sequence $\barx^{n,*}_6$ in \eqref{red1} over the entire half line $[0,\infty)$,
showing that $\tau$ places no constraint on expanding the convergence interval.
Moreover, by part (ii) of Theorem 5.2 in \cite{PeW09c}, any solution to the ODE, with a specified initial condition, can be extended indefinitely,
and is unique. Hence that places no constraint either.  Finally, the martingale argument allows us to uniquely characterize the
steady-state distribution of the FTSP $D(\gamma, \cdot)$ in \S \ref{secFTSP} even when the state $\gamma$ is not in $\AA$, provided that we have the SSC provided by
Theorem \ref{thSSCextend}.  In particular, we will have either $\pi_{1,2} (\gamma) = 1$ or $\pi_{1,2} (\gamma) = 0$ if $\gamma \notin \AA$.
\end{proofof}


\section{Remaining Proofs of Theorems in \S \ref{secMain}.}\label{secAuxProofs}

We now provide the remaining proofs for four theorems in \S \ref{secMain}.
At this point, Theorem \ref{th1} has been proved.

\subsection{Proof of Theorem \ref{thStatLim}.}\label{secStatProofs}

We will consider a sequence of stationary Markov processes $\{\{\barx^{n}_6 (t) : t \ge 0\}: n \ge 1\}$, with
$\barx^{n}_6 (0) \equalDist \barx^n_6 (\infty)$ for each $n \ge 1$. That initial condition makes the stochastic processes strictly
stationary. We will refer to these stationary processes as {\em stationary versions} of the processes $\barx^n_6$
and denote them by $\barx^n_s$.  We start by establishing tightness.

\begin{lemma}$($tightness of the sequence of stationary distributions$)$ \label{lmTightSSr.v.}
The sequences $\{\barx^n_6 (\infty): n \ge 1\}$ and $\{\{\barx^n_s (t) : t \ge 0\}: n \ge 1\}$ are tight in $\RR_6$ and $\D_6$,
respectively.
\end{lemma}

\begin{proof}
First, for the tightness of $\{\barx^n_6 (\infty): n \ge 1\}$ in $\RR_6$, it suffices
to treat the six components separately.  The tightness of $\{\barz^n_{i,j} (\infty): n \ge 1\}$ in $\RR$ is
immediate because $0 \le \barz^n_{i,j} (\infty) \le m^n_j/n$, where $m^n_j/n \ra m_j$ as $n \ra \infty$.
The tightness of the queue lengths follows from Lemma \ref{lmQbds2}.  In particular,
since (i) $Q^n_i (t) \le_{st} Q^n_{i,bd} (t)$ for all $t$ and (ii) $Q^n_i (t) \Rightarrow Q^n_i (\infty)$
and $Q^n_{i,bd} (t) \Rightarrow Q^n_{i,bd} (\infty)$ as $t \ra \infty$ for all $n \ge 1$, we necessarily have $Q^n_i (\infty) \le_{st} Q^n_{i,bd} (\infty)$ for all $n$,
because stochastic order is preserved under convergence.  Since $\bar{Q}^n_{i,bd} (\infty) \Ra q_{i,bd} (\infty) \equiv q_i (0) \vee (\lambda_i/\theta_i)$
as $n \ra \infty$,
the sequence $\{\bar{Q}^n_{i,bd} (\infty): n \ge 1\}$ is stochastically bounded, which implies that
the sequence $\{\bar{Q}^n_{i} (\infty): n \ge 1\}$ is stochastically bounded as well.
Since tightness of the marginal distributions implies tightness of vectors, the sequence of
steady-state random vectors $\{\bar{X}^n (\infty): n \ge 1\}$ is tight in $\RR_6$.
Given the tightness of $\{\barx_s^n (0)\}$, the proof of tightness in $\sD_6$ is identical to the proof of Theorem \ref{lmTight}.
\end{proof}

We next establish the analogue of the structural simplification results in \S \ref{secSSCserv}.
Let $\barx (\infty) \equiv (\barq_i (\infty), \barz_{i,j} (\infty))$ be a limit of the stochastically bounded sequence $\{\barx^n_6 (\infty)\}$.
Let $\barx \equiv (\barq_i, \barz_{i,j})$ be a limit of the sequence $\{\{\barx^n_s (t) : t \ge 0\}: n \ge 1\}$. Note that $\barx$ must itself be a stationary process.

\begin{lemma} \label{lmStatLim}
$P(\barz_{1,1} (\infty) = m_1, \barz_{2,1} (\infty) = 0, \barz_{2,2} (\infty) = m_2 - \barz_{1,2} (\infty)) = 1$.
\end{lemma}

\begin{proof}
Let $\bar{I}_j (\infty) \equiv m_j - \barz_{1,j} (\infty) - \barz_{2,j} (\infty)$, $j = 1,2$.
To prove the claim,  we need to show that $P (\barz_{2,1} (\infty) > 0) = P (\bar{I}_j (\infty) > 0) = 0$.
We will consider the sequence of stationary versions $\{\barx^n_s\}$, and a limit $\barx$ of this sequence.

(i) $P(\barz_{2,1} = 0e) = 1$.   We will show that the opposite assumption leads to a contradiction.
Hence, suppose that $P(\barz_{2,1} (s) > 0) > 0$ for some $s \ge 0$.
The stationarity of $\barz_{2,1}$ implies that $\barz_{2,1} (0) \equalDist \barz_{2,1} (s)$.
Hence, we can equivalently assume that $P (\barz_{2,1} (0) > 0) > 0$.

Let $B_{2,1}$ denote the set in the underlying probability space, of all sample paths of $\barx$
with $\barz_{2,1} (0) > 0$. By the contradictory assumption, $P (B_{2,1}) > 0$.
Following the arguments in Lemma \ref{lmZ12}, if $\barz_{2,1} (0) > 0$, then $\barz_{1,2} (t) > 0$
for all $t \ge 0$, which implies that $\barz_{2,1} (t) > 0$ for all $t \ge 0$ in $B_{2,1}$.

Consider a sample path in $B_{2,1}$.
Because of one-way sharing, $\barz_{1,2} = 0e$, so that
the only departures from $\barq_1$ are due to service completions in pool $1$
and abandonment. Since $\barz_{1,1} \le m_1$ w.p.1, $\barq_1$ is stochastically bounded
from below, in sample-path stochastic order, by the fluid limit of an Erlang-A model with $m_1$ servers.
At the same time, $\barq_2$ is stochastically bounded from above by the fluid limit of an
Erlang-A model with $m_2$ servers. Hence, there exists $\ep > 0$ such that, for some $s \ge 0$,
\bequ \label{ineqStat}
r \barq_2 (t) - \ep \le r q^a_2 < q^a_1 \le \barq_1 (t) + \ep \qforallq t \ge s,
\eeq
where $r q^a_2 < q^a_1$ by Assumption \ref{assA}.
Now, because $\barq_1$ and $\barq_2$ are bounded with probability 1, we can find
$s_0 > 0$ such that \eqref{ineqStat} holds with this $s_0$ {\em for all possible initial condition in $B_{2,1}$}.
However, this implies that $\barz_{2,1}$ is strictly decreasing for all $t \ge s_0$ and for all sample paths in $B_{2,1}$
(because no fluid can flow from queue $1$ to pool $2$), so that
$\barz_{2,1} (t) < \barz_{2,1} (s_0)$ for all $t > s_0$ in $B_{2,1}$, contradicting the stationarity of $\barz_{2,1}$.
Thus, $P (B_{2,1}) = 0$.

(ii) $P(\bar{I}_1 = \bar{I}_2 = 0e) = 1.$
We follow the proof of Theorem \ref{thQpos},
building on the result $P(\barz_{2,1} (\infty) = 0) = 1$ just established.
Recall that $L^n_i \equiv Q^n_i + Z^n_{i,1} + Z^n_{i,2} - m^n_i$ in \eqref{net},
representing the excess number of class-$i$ customers in the system,
is stochastically bounded from below, in sample-path stochastic order, by the process $L^n_{b,i}$,
with $L^n_{i,b}$ defined by imposing a reflecting upper barrier at $k^n_{1,2}$ for $i = 1, 2$
and letting $L^n_{i,b} (0) = L^n_i (0) \wedge k^n_{1,2}$.
However, here we working with stationary versions.  Since,
the processes $(L^n_1, L^n_2)$ are strictly stationary, so are the reflected processes
$(L^n_{1,b}, L^n_{2,b})$.
Then $U^n \ge_{st} U^n_b$, where $U^n$ and $U^n_b$ and the linear functions of
$(L^n_1, L^n_2)$ and $(L^n_{1,b}, L^n_{2,b})$, respectively, defined in \eqref{s102}.
These processes $U^n$ and $U^n_b$ are also stationary processes.

However, just as in the proof of Theorem \ref{thQpos}, $U^n_b$ is a birth and death process on the integers in $(- \infty, 0]$,
which is independent of $Z^n_{1,2}$, with drift $\delta^n_b$ in \eqref{driftn}.
Since $\delta^n_b / n \ra \delta_b > 0$, the birth and death $U^n_b$ has a positive drift for all $n$ large enough.
Consequently, $-U^n_b$ has the structure of the stationary queue length process in a stable $M/M/1$ queue.
Since the traffic intensity converges to a limit strictly less than one, the initial value, and the value at any time, is
stochastically bounded.
Moreover, we can apply the essentially the same extreme value argument used
in the proof of Theorem \ref{thQpos} to conclude that, for any $\tau > 0$,
that $P(\| \bar{I}^n_i\|_{\tau} > 0) \ra 0$ as $n \ra \infty$.
We thus conclude that $P(\|\bar{I}_i\|_{\tau} = 0) = 1$, from which the conclusion follows,
because the interval $[0, \infty)$ can be represented as the countable union of finite intervals of finite length,
and the countable sum of $0$ probabilities is itself zero.
\end{proof}

A function $\bar{Y} : \RR_m \ra \RR_k$, $m, k \ge 1$,
is said to be locally Lipschitz continuous if for any compact set $B$,
there exists a constant $K (B)$ such that, for any $s, t \in B$,
$|\bar{Y} (s) - \bar{Y} (t)| \le K (B) |s - t|$.
A locally Lipschitz continuous function is absolutely continuous and is thus differentiable almost everywhere.

In the following we will consider two locally Lipschitz continuous functions
$V : \RR^+_k \ra \RR^+$ and $\bar{Y} : \RR^+ \ra \RR^+_k$, and the ``Lie derivative''
$\dot{V} (\bar{Y} (t)) \equiv \nabla V \cdot \bar{Y}'$, where $\nabla V$ denotes the gradient of $V$,
and $\nabla V \cdot \bar{Y}'$ is the usual inner product of vectors.
Since both functions are differentiable almost everywhere,
$\dot{V}(\bar{Y} (t))$ is understood to be be taken at points $t$ for which both $V$ and $\bar{Y}$
are differentiable.

For a vector $x \in \RR_k$, we let $\|x\|$ denote its $L_1$ norm, although any other norm in $\RR_k$
can be used in the following.  For the following we draw on \S 8.3 of \cite{PeW09c}.

\begin{lemma} \label{lyapunov}
Let $V : \RR^+_k \ra \RR_+$ be locally Lipschitz continuous, such that $V (x) = 0$
if and only if $x = 0$.
Let $\bar{Y} : \RR^+ \ra \RR^+_k$ be Lipschitz continuous with constant $N$, such that
$\bar{Y} \le M$, for some $M > 0$.
If $\dot{V} (\bar{Y} (t)) < 0$ for all $t \ge 0$ for which $\bar{Y} (t) \ne 0$, then $\bar{Y}(t) \ra 0$ as $t \tinf$.
\end{lemma}

\begin{proof}
Let $f : \RR_+ \ra \RR_+$ be absolutely continuous, such that $f (t) = 0$ if and only if $t = 0$.
If for almost every $t \ne 0$, $f' (t) < 0$, then $f (t) \ra 0$ as $t \tinf$.
We can apply that argument to the function $f (t) \equiv V (\bar{Y} (t))$
if we show that $V (\bar{Y} (t))$ is locally Lipschitz continuous.
To see that is the case, note that for all $t > 0$,
\bes
\bar{Y} (s) \le \bar{Y} (0) + |\bar{Y} (s) - \bar{Y} (0)| \le M + N t, \quad 0 \le s \le t.
\ees
Therefore, for any $0 \le u \le s \le t$,
$|f (s) - f (u)| \le K (B) N (s - u)$,
where $B \equiv \{x \in \RR_k^+ : \|x\| \le M + N t\}$.
\end{proof}

\begin{proofof}{Theorem \ref{thStatLim}}
We consider a converging subsequence $\{\barx^{n'}_s: n' \ge 1\}$ of the sequence $\{\barx^n_s: n \ge 1\}$ with limit $\barx$.
By Lemma \ref{lmStatLim}, we can consider it to be three-dimensional
with components $\barq_1$, $\barq_2$ and $\barz_{1,2}$.
By Corollary \ref{corLip}, the limit $\barx$ is Lipschitz continuous in $t$, so that it is differentiable almost everywhere.
Specifically, $\barx$ is a limit of the sequence represented by \eqref{FluidScaled}, with $\barx^n (0) \deq \barx^n (\infty)$
for all $n \ge 1$. Then each component of the converging subsequence $\{\barx^{n'}_s\}$ in \eqref{FluidScaled}
converges to its respective limit (e.g., $\barq^{n'}_1$ to $\barq_1$).
For our purposes here, it is sufficient to conclude that $\barx (t) \ra x^*$ w.p.1 as $t \tinf$.
Hence, we will not characterize the limit of $\barx (t)$ as $t \ra \infty$.

For the representation of the converging subsequence $\{\barx^{n'}_s\}$ in \eqref{FluidScaled}, let
\bes
\bar{W}^{n'} \equiv (C-1) \barz^{n'}_{1,2} + C \barq^{n'}_1 + \barq^{n'}_2,
\ees
where $C \ge 1$ will be specified later.
Since $\barx^{n'}_s \Ra \barx$ by assumption, we can apply the continuous mapping theorem to conclude that
$\bar{W}^n \Ra \bar{W}$ in $\D_3$ as $n \tinf$, where
\bes
\bsplit
\bar{W} (t) & \equiv (C-1) \barz_{1,2} (t) + C \barq_1 (t) + \barq_2 (t) \\
& = (C-1) \barz_{1,2} (0) - (C\mu_{1,2} - \mu_{2,2}) \int_0^t \barz_{1,2} (s)\ ds - \mu_{2,2} m_2 t \\
& \quad + C \barq_1 (0) + C (\lm_1 - \mu_{1,1} m_1)t - C \theta_1 \int_0^t \barq_1 (s)\ ds
+ \barq_2 (0) + \lm_2 t - \theta_2 \int_0^t \barq_2 (s)\ ds.
\end{split}
\ees
with derivative
\bes
\bsplit
\bar{W}' (t) = -(C \mu_{1,2}-\mu_{2,2}) \barz_{1,2} (t) - C \theta_1 \barq_1 (t) - \theta_2 \barq_2 (t)
+ C(\lm_1 - \mu_{1,1} m_1) + \lm_2 - \mu_{2,2} m_2.
\end{split}
\ees

For $x \in \RR_3$, let $V (x) \equiv C x_1 + x_2 + (C-1) x_3$ and note that
$\dot{V} (\barx) \equiv \nabla V \cdot \barx' = \bar{W}'$, and that $V$ is locally Lipschitz continuous.
where $\nabla V$ denotes the gradient of $V$,
and $\nabla V \cdot \barx'$ is the usual inner product of vectors.
Now let $\bar{Y}$ denote the derivative of $\barx$ shifted by $x^*$ for $x^*$ in \eqref{statPt}, i.e., $\bar{Y}' \equiv \barx' + x^*$.
Hence, as in the
proof of Theorem 8.3 in \cite{PeW09c}, we have
\bes
\dot{V} (\bar{Y}) \equiv -(C \mu_{1,2}-\mu_{2,2}) \barz_{1,2} (t) - C \theta_1 \barq_1 (t) - \theta_2 \barq_2 (t).
\ees
If $\mu_{1,2} > \mu_{2,2}$, then we let $C = 1$, and if $\mu_{1,2} \le \mu_{2,2}$ we let $C$ be any number such that $C > \mu_{2,2}/\mu_{1,2} \ge 1$.
With this choice of $C$, we see that $\dot{V} (\bar{Y}) < 0$.
By Lemma \ref{lyapunov}, $\bar{Y}(t) \ra 0$, which implies that $\barx (t) \ra x^*$ as $t \tinf$ w.p.1.
(Note that $\bar{Y}$ and $\barx$ are Lipschitz continuous and bounded, as required.)

For $\af > 0$, let $\beta_V (\af) \equiv \{x \in \RR_3 : \|V (x) - V (x^*)\| \le \af\}$.
By the monotonicity of $V(\barx)$ established above, for each $\af > 0$ there exists $T (\af, \barx(0))$, such that
$\barx (t) \in \beta_V (\af)$ for all $t \ge T(\af, \barx(0))$.
Since the queues $\barq_1$ and $\barq_2$ are bounded w.p.1, we can uniformly bound $T (\af, \barx (0))$ (uniformly in $\barx (0)$).
Hence, there exists $T \equiv T (\af)$, such that $\barx (t) \in \beta_V (\af)$ w.p.1 for all $t \ge T$.
It follows from the stationarity of $\barx$ that $\barx (0) \in \beta_V (\af)$.
Since this is true for all $\af > 0$, it must hold that $\barx (0) = x^*$ which, by the equality in distribution
$\barx (\infty) \deq \barx (0)$, implies that $\barx (\infty) = x^*$ w.p.1.
We have thus shown that the limit of all converging subsequences of $\{\barx^n (\infty): n \ge 1\}$ is $x^*$
in \eqref{statPt}, which implies the full convergence $\barx^n (\infty) \Ra x^*$ as $n \ra \infty$.
Moreover, since $x^*$ is the limit of a stationary sequence, $x^*$ itself must be a stationary point for each fluid limit $\barx$
(i.e., if $\barx(0) = x^*$, then $\barx(t) = x^*$ for all $t \ge 0$), and it is globally asymptotically stable, because $\barx(t) \ra x^*$
as $t \tinf$, as was shown above.
\end{proofof}

Note that none of the proofs in this section used the initial condition in Assumption \ref{assC}, or the rationality of the queue ratios
in Assumption \ref{assE}.

\subsection{Proof of Theorem \ref{thConvToFast}.}\label{secProof43}

The claimed convergence in $\D$ is less complicated than it might appear, because there is no spatial scaling.
Consequently, all processes are pure-jump processes, having piecewise constant sample paths,
with only finitely many discontinuities in any bounded interval.  Both processes have the same four possible transitions from each state:
$\pm 1$ or $\pm r$ for $r \equiv r_{1,2}$.  Hence, it suffices to show convergence in distribution of
the first $k$ pairs of jump times and jump values for any $k \ge 1$.  That can be done by mathematical induction on $k$.

We can start by applying the Skorohod representation theorem to replace the assumed convergence
in distribution $\Gamma^n_6/n \Ra \gamma_6$ and $D^n_e(\Gamma^n, 0) \Ra D(\gamma, 0)$ as $n \ra \infty$
by convergence w.p.1 for alternative random vectors having the same distribution.
Hence, it suffices to assume that $\Gamma^n_6/n \ra \gamma_6$ and
$D^n_e(\Gamma^n, 0) \ra D(\gamma, 0)$ as $n \ra \infty$ and w.p.1.  (We do not introduce new notation.)
Thus, for each sample point, we will have $D^n_e(\Gamma^n, 0) = D(\gamma, 0)$ for all $n$ sufficiently
large.  Hence, we can assume that we start with equality holding.

Moreover, we can exploit the structure of pure-jump Markov
processes. The FTSP is directly such a pure-jump Markov process
with transition rates given in \eqn{bd1}-\eqn{bd4}.  These
rates were defined to be the limit of the transition rates of
the queue difference processes $D^n_{1,2}$, after dividing by
$n$.  The queue difference processes $D^n_{1,2}$ in
\eqn{Dprocess} are not Markov, but they are simple linear
functions of the pure-jump Markov process $X^n_6$.

The transition rates for $D^n_{1,2}$ closely parallel \eqn{bd1}-\eqn{bd4}.
Because of the assumptions, we can work with the three-dimensional random state $\Gamma^n \equiv (Q^n_1, Q^n_2, Z^n_{1,2}) \in \RR_3$
with the understanding that the remaining components of $\Gamma^n_6$ are $Z^n_{1,1} = m^n_1$, $Z^n_{2,1} = 0$ and $Z^n_{2,2} = m^n_2 - Z^n_{1,2}$.
Let $D^n(\Gamma^n, t_0) \equiv D^n_{1,2} (t_0)$ under the condition that $X^n (t_0) = \Gamma^n$.
When $D^n(\Gamma^n, t_0) \le 0$, let the transition rates  be $\lm^{(r)}_-(n, \Gamma^n)$, $\lm^{(1)}_- (n, \Gamma^n)$,
$\mu^{(r)}_- (n, \Gamma^n)$ and $\mu^{(1)}_- (n, \Gamma^n)$ in \eqref{BDfrozen1}, for transitions of $+r$, $+1$, $-r$ and $-1$, respectively.
When $D^n(\Gamma^n, t_0) > 0$, let the transition rates be
$\lm^{(r)}_+(n, \Gamma^n)$, $\lm^{(1)}_+ (n, \Gamma^n)$, $\mu^{(r)}_+ (n, \Gamma^n)$ and $\mu^{(1)}_+ (n, \Gamma^n)$ in \eqref{BDfrozen2},
for transitions of $+r$, $+1$, $-r$ and $-1$, respectively.

The many-server heavy-traffic scaling in \eqn{MS-HTscale}
and the condition $\Gamma^n/n \ra \gamma$ imply that the transition rates of $X^n_6$ are of order $O(n)$ as $n \ra \infty$.
However, the time expansion in \eqn{fast102} brings those transition rates back to order $O(1)$.
Indeed, from \eqn{bd1}-\eqn{bd4} and \eqref{BDfrozen1}-\eqref{BDfrozen2}, we see that the transition rates of $D^n_{e}(\Gamma^n, \cdot)$,
which change with every transition of the CTMC $X^n_6$, actually converge to the transition rates of the FTSP $D(\gamma, \cdot)$,
which only depend on the region ($D > 0$ and $D \le 0$).

The time until the first transition in the pure jump Markov process $D(\gamma, \cdot)$ is clearly exponential.
Since the queue lengths can be regarded as strictly positive by Theorem \ref{thQpos}, the first transition of $D^n_e (\Gamma^n, \cdot)$
coincides with the first transition time of the underlying CTMC $X^n_6$, which also is exponential.  Since the transition rates converge,
 the time until the first transition of $D^n_e (\Gamma^n, \cdot)$ converges in distribution to the exponential time until the first transition
of the FTSP $D(\gamma, \cdot)$.  Moreover, in both processes the jump takes one of four values $\pm 1$ or $\pm r$.  The probabilities of these values converges as well.
Hence the random first pair of jump time and jump value converges to the corresponding pair of the FTSP.  The same reasoning applies to successive
pairs of jump times and jump values, applying mathematical induction.  That completes the proof.
\hfill \qed

\subsection{Auxiliary Results for FTSP's.} \label{secAuxFTSP}

Before proving Theorems \ref{thDSB} and \ref{thAPlocal}, we prove some auxiliary lemmas that we employ in the proofs.
We use stochastic bounds by frozen queue-difference stochastic processes as in \S \ref{secAux}.
The following lemma is proved much like Lemma \ref{lmD21extreme}, exploiting the bounds in \S \ref{secStoBd}.


\begin{lemma}{$($bounding frozen processes$)$} \label{lmD12PR}
Suppose that $x(t) \in \AA$, $t_1 \le t \le t_2$. For any $t \in [t_1, t_2]$
and $\epsilon > 0$, there exist positive constants $\delta, \eta$, state
vectors $x_m, x_M  \in \AA$ and random state vectors $X^n_m, X^n_M$, $n \ge 1$, such that
$\|x_m - x(t)\| < \epsilon$, $\|x_M - x(t)\| < \epsilon$, $n^{-1}X^n_m \Ra x_m$, $n^{-1}X^n_M \Ra x_M$ as $n \ra \infty$,
\begin{eqnarray}\label{DstoBd}
D^n_f (X^n_m, \cdot) & \le_{r} & D^n_f (X^n(t),\cdot) \le_{r} D^n_f (X^n_M, \cdot) \qandq \nonumber \\
D^n_f (X^n_m, \cdot) & \le_{r} & D^n_{1,2} (t) \le_{r} D^n_f (X^n_M, \cdot) \qinq \D([t_1 \vee (t-\delta), (t + \delta) \wedge t_2]) \qforallq n \ge 1,
\end{eqnarray}
and $P(B^n (\delta, \eta)) \ra 1$ as $n\ra \infty$, where
\bequ \label{posrec-delta}
B^n (\delta, \eta) \equiv \{\delta^n_+(X^n_M) < - \eta,\; \delta^n_-(X^n_M) > \eta,\; \delta^n_+(X^n_m) < - \eta,\; \delta^n_-(X^n_m) > \eta \}
\eeq
As a consequence, the bounding frozen processes $D^n_f (X^n_M,
\cdot)$ and $D^n_f (X^n_m, \cdot)$ in {\em \eqn{DstoBd}}, and
thus also the interior frozen processes $D^n_f (X^n(t),\cdot)$,
satisfy {\em \eqn{Aset2}} on $B_n (\delta, \eta)$ and are thus
positive recurrent there, $n \ge 1$.
\end{lemma}

\begin{proof}
We use the convergence $\barx^n \Ra x$ provided by Theorem \ref{th1}.
Consider one $t \in [t_1, t_2]$.
By the linearity of the drift functions $\delta_+$ and
$\delta_-$ in \eqref{drifts}, 
$\delta^n_+(X^n(t))/n \Ra \delta_+(x(t))$ and
$\delta^n_-(X^n(t))/n \Ra \delta_-(x(t))$ for $x(t) \in \AA$,
so that \eqref{Aset2} holds. Hence there exists $\eta > 0$ such
that
\bes
\lim_{n \tinf} P(\delta^n_+(X^n(t)) < - \eta \qandq
\delta^n_-(X^n(t)) > \eta) = 1,
\ees
i.e.,
\eqref{posrecDn}
holds for $\Gamma^n = X^n (t)$ with probability converging to
$1$ as $n \tinf$.

We now bound the drifts in \eqref{driftDn}.  We do that by
bounding the change in the components of $X^n (t)$ in a short
interval around time $t$.  To do that, we use the
stochastic-order bounds in \S \ref{secStoBd}, constructed over the interval
$[t_1 \vee (t - \delta), (t + \delta) \wedge t_2]$, with the construction beginning at the left endpoint
$t_1 \vee (t - \delta)$, just as in the construction after time $0$ in \S \ref{secStoBd}.
Let $I_{(t_1,\delta)} \equiv [t_1 \vee (t - \delta), (t + \delta) \wedge t_2]$ and let $\| \cdot\|_{\delta}$
denote the norm over the interval $I_{(t_1,\delta)}$.
To construct $X^n_{M}$, let
\bes 
X^n_{M^+} \equiv (Q^n_{1,M}, Q^n_{2,M}, Z^n_{M^+}) \qandq X^n_{M^-} \equiv (Q^n_{1,M}, Q^n_{2,M}, Z^n_{M^-}),
\ees
where
\bes 
\begin{array}{lll}
Q^n_{1,M} & \equiv \inf_{t \in I_{(t_1,\delta)}} Q^n_{1,a}(t) \vee 0, & Q^n_{2,M} \equiv \|Q^n_{2,b}\|_{\delta}, \\
Z^n_{M^+} & \equiv \inf_{t \in I_{(t_1,\delta)}} Z^n_{+}(t),        & Z^n_{M^-} \equiv \|Z^n_{-}\|_{\delta},
\end{array}
\ees
with $Z^n_{+}(t) \equiv Z^n_{b} (t)$ and $Z^n_{-}(t) \equiv Z^n_{a} (t)$ if $\mu_{2,2} \ge \mu_{1,2}$,
and $Z^n_{+}(t) \equiv Z^n_{a} (t)$ and $Z^n_{-}(t) \equiv Z^n_{b} (t)$ otherwise.
We work with the final value $X^n_{M^+} \equiv X^n_{M^+} ((t + \delta) \wedge t_2)$, and similarly for $X^n_{M^-}$.
Let $\{D^n_f (X^n_M, s) : s \ge 0\}$ have the rates determined by $X^n_{M^-}$ when $D^n_f (X^n_M, s) \le 0$,
and the rates determined by $X^n_{M^+}$ when $D^n_f (X^n_M, s) > 0$.

We do a similar construction for $X^n_{m}$.  Let
\bes 
X^n_{m^+} \equiv (Q^n_{1,m}, Q^n_{2,m}, Z^n_{m^+}) \qandq X^n_{m^-} \equiv (Q^n_{1,m}, Q^n_{2,m}, Z^n_{m^-}),
\ees
where
\bes 
\begin{array}{lll}
Q^n_{1,m}  & \equiv \|Q^n_{1,b}\|_{\delta}, & Q^n_{2,m} \equiv \inf_{t \in I_{(t_1,\delta)}} Q^n_{2,a}(t) \vee 0, \\
Z^n_{m^+}  & \equiv \|Z^n_{+}\|_{\delta}, & Z^n_{m^-} \equiv \inf_{t \in I_{(t_1,\delta)}} Z^n_{-}(t).
\end{array}
\ees
with $Z^n_{+}(t) \equiv Z^n_{a} (t)$ and $Z^n_{-}(t) \equiv Z^n_{b} (t)$ if $\mu_{2,2} \ge \mu_{1,2}$,
and $Z^n_{+}(t) \equiv Z^n_{b} (t)$ and $Z^n_{-}(t) \equiv Z^n_{a} (t)$ otherwise (the reverse of what is done in $X^n_M$).
Let $\{D^n_f (X^n_m, s) : s \ge 0\}$ have the rates from $X^n_{m^-}$ when $D^n_f (X^n_m, s) \le 0$,
and the rates from $X^n_{m^+}$ when $D^n_f (X^n_m, s) > 0$.
By this construction, we achieve the ordering in \eqn{DstoBd}.
We cover the rates of $D^n_{1,2} (t)$ too because we can make
the identification: the rates of $D^n_{1,2} (t)$ given $X^n (t)$ coincide with the rates of $D^n_f (X^n (t), \cdot)$.

It remains to find a $\delta$ such that both the processes $\{D^n_f (X^n_m, s) : s \ge 0\}$ and
$\{D^n_f (X^n_M, s) : s \ge 0\}$ are asymptotically positive recurrent. To do so, we use Lemma \ref{lmBoundsLim},
which concludes that the bounding processes as functions of $\delta$ have fluid limits.
By Lemma \ref{lmBoundsLim}, we can conclude that
$\bar{X}^{n}_{m^{+}} \equiv n^{-1} X^{n}_{m^{+}} \Rightarrow x_m^{+}$,
$\bar{X}^n_{m^-} \equiv n^{-1} X^n_{m^{-}} \Rightarrow x_m^{-}$,
$\bar{X}^n_{M^{+}} \equiv n^{-1} X^n_{M^{+}} \Rightarrow x_M^{+}$
and $\bar{X}^n_{M^-} \equiv n^{-1} X^n_{M^{-}} \Rightarrow x_M^{-}$
in $\D_3$, where $x_{m^{+}}$, $x_m^{-}$, $x_M^{+}$ and $x_M^{-}$ are all continuous with
$x_m^{+} (t_1 \vee (t - \delta)) = x_m^{-} (t_1 \vee (t - \delta)) = x_M^{+} (t_1 \vee (t - \delta))
= x_M^{-} (t_1 \vee (t - \delta)) =  x(t_1 \vee (t - \delta)) \in \AA$.
Hence, we can find $\delta'$ such that $x_m (\delta) \in \AA$ and $x_M (\delta) \in \AA$ for all $\delta \in [0, \delta']$.
Hence, we can choose $\delta$ such that the constant vectors $x_m \equiv x_m (\delta)$ and $x_M \equiv x_M (\delta)$
both are arbitrarily close to $x(t_1)$.

Finally, we use the linearity of the drift function to deduce
the positive recurrence of the processes depending upon $n$. As $n \tinf$,
\begin{eqnarray}
\delta^n_-( X^n_{m^{-}}) / n & \Ra & \delta_-(x_m^{-}), \quad \delta^n_+( X^n_{m^{+}}) / n  \Ra  \delta_+(x_m^{+}), \nonumber \\
\delta^n_-( X^n_{M^{-}}) / n & \Ra & \delta_-(x_M^{-}), \qandq \delta^n_+( X^n_{M^{+}}) / n  \Ra  \delta_+(x_M^{+}) \quad \mbox{in } \RR. \nonumber
\end{eqnarray}
\end{proof}

We immediately obtain the following corollary to Lemma \ref{lmD12PR}, exploiting Corollary \ref{corRateOrder}.

\begin{coro}\label{corStochBd}
Let $\zeta \equiv (j \vee k) - 1$ using the QBD representation
based on $r_{1,2} = j/k$ in \S $6$ of {\em \cite{PeW09c}}. If,
in addition to the conditions of Lemma {\em \ref{lmD12PR}},
\begin{eqnarray}\label{DstoBdinit}
 D^n_f (X^n_m, 0) - \zeta& \le_{st} & D^n_f (X^n(0 \vee t - \delta),0) \le_{st} D^n_f (X^n_M, 0) + \zeta \qandq \nonumber \\
D^n_f (X^n_m, 0) - \zeta & \le_{st} & D^n_{1,2} (0 \vee t - \delta) \le_{st} D^n_f (X^n_M, 0)  + \zeta \qinq \RR, \quad n \ge 1,
\end{eqnarray}
then, in addition to the conclusions of Lemma {\em
\ref{lmD12PR}},
\begin{eqnarray}\label{DstoBd2}
 D^n_f (X^n_m, \cdot) - \zeta& \le_{st} & D^n_f (X^n(t),\cdot) \le_{st} D^n_f (X^n_M, \cdot) + \zeta, \nonumber \\
D^n_f (X^n_m, t) - \zeta & \le_{st} & D^n_{1,2} (t) \le_{st} D^n_f (X^n_M, t)  + \zeta \qinq \D ([t_1 \vee (t - \delta), (t + \delta) \wedge t_2]).
\end{eqnarray}
\end{coro}

The following lemmas correspond to the QBD structure of the FTSP.
These results hold for QBD processes much more generally, but we state them in terms of the FTSP.

\begin{lemma}{$($continuity of the stationary distribution of the FTSP$)$} \label{lmContSS}
The FTSP stationary random variable $D(\gamma, \infty)$ is continuous in the metric {\em \eqn{Levy}} as a function the state $\gamma$ in $\AA$
and thus $D(x(t), \infty)$ is continuous in the metric {\em \eqn{Levy}} as a function of the time argument $t$ when $x(t) \in \AA$.
\end{lemma}

\begin{proof}
This result follows from Theorem
7.1 in \cite{PeW09c}, in particular
 from the stronger differentiability of the QBD $R$ matrix in an open
neighborhood of each $\gamma \in \AA$, building on Theorem 2.3 in \cite{H95}.  The second statement
follows from the first, together with the continuity of $x(t)$ as a function of $t$.
\end{proof}

\begin{remark}{$($continuity on $\rS)$}\label{rmContOnS}
{\em
Lemma \ref{lmLimFTSP} shows that the stationary distribution of the FTSP is well defined on all of $\rS$
provided that we extend the set of possible values of the FTSP to the space $E \equiv \ZZ \cup \{+\infty\} \cup \{-\infty\}$, as in \S \ref{secHK}.
Following Theorem 7.1 of \cite{PeW09c}, we can show that Lemma \ref{lmContSS} holds, not only on $\AA$,
but also on $\rS$, but that extension is not needed here.
}
\end{remark}

Let $\tau(t) \equiv \inf\{u > 0 : D(x(t), u) = s\}$, where $s$ is a state in the state space of the QBD $D$.
The next lemma establishes the existence of a finite moment generating function (mgf) for $\tau (t)$ for a positive recurrent FTSP.

\begin{lemma} {$($finite mgf for return times$)$}\label{lmReturnTime}
For $x(t) \in \AA$, let $\tau \equiv \tau (t)$ be the return time of the positive recurrent QBD $D(x(t), \cdot)$ to a specified state $s$.
Then there exists $\theta^* > 0$ such that
$\phi_{\tau} (\theta) \equiv E[e^{\theta \tau}] < \infty$ for all $\theta < \theta^*$.
\end{lemma}

\begin{proof}
As for any irreducible positive recurrent CTMC, a positive
recurrent QBD is regenerative, with successive visits to any
state constituting an embedded renewal process. As usual for
QBD's (see \cite{LR99}), we can choose to analyze the system
directly in continuous time or in discrete time by applying
uniformization, where we generate all potential transitions
from a single Poisson process with a rate exceeding the total
transition rate out of any state. In continuous time we focus
on the interval between successive visits to the regenerative
state; in discrete time we focus on the number of Poisson
transitions between successive visits to the regenerative
state.

Let $N$ be the number of Poisson transitions (with specified Poisson rate).
The number of transitions, $N$, has the generating function (gf) $\psi_{N} (z) \equiv E[z^N]$,
for which there exists a radius of convergence $z^*$ with $0 < z^{*} < 1$ such that $\psi_{N} (z) < \infty$ for $z < z^{*}$
and $\psi_{N} (z) = \infty$ for $z > z^{*}$.

The mgf $\phi_{\tau} (\theta)$ and gf $\psi_{N} (z)$ can be expressed directly in terms of the finite QBD defining
matrices. It is easier to do so if we choose a regenerative
state, say $s^{*}$, in the boundary region (corresponding to
the matrix $B$ in (6.5)-(6.6) of \cite{PeW09c}). To illustrate, we
discuss the gf. With $s^*$ in the boundary level, in addition
to the transitions within the boundary level and up to the next
level from the boundary, we only need consider the number of
transitions, plus starting and ending states, from any level
above the boundary down one level. Because of the QBD
structure, these key downward first passage times are the same
for each level above the boundary, and are given by the
probabilities $G_{i,j} [k]$ and the associated matrix
generating function $G(z)$ on p. 148 of \cite{LR99}.  Given
$G(z)$, it is not difficult to write an expression for the
generating function $\psi_{N^n} (z)$, just as in the familiar
birth-and-death process case; e.g., see \S 4.3 of \cite{LR99}.
\end{proof}

The next lemmas establish results regarding distances between processes in a discrete state apace.
In particular, we consider the state-space $E$ of the processes $D^n_{1,2}$ and the FTSP $D$, which is a countable lattice.
With the QBD representation (achieved by renaming the states in $E$) the state-space $E$ is a subset of $\ZZ$.
To measure distances between probability distributions on $\ZZ$, corresponding to convergence in distribution,
we use the L\'{e}vy metric, defined for any two cumulative distribution functions (cdf's) $F_1$ and $F_2$ by
\bequ \label{Levy}
\sL (F_1, F_2) \equiv \inf{\{ \epsilon > 0 : F_1 (x - \epsilon) - \epsilon \leq F_2 (x) \leq F_1 (x + \epsilon) + \epsilon \, \mbox{ for all} \, x\}}.
\eeq
For random variables $X_1$ and $X_2$, $\sL (X_1, X_2)$ denotes the L\'{e}vy distance between their probability distributions.
(The L\'{e}vy metric is defined for probability distributions on $\RR$, but we will use it for processes defined on
the discrete space $E$.)

\begin{lemma} $($uniform bounds on the rate of convergence to stationarity$)$\label{lmExpErg}
For every $t_0 \in [t_1, t_2]$ there exist $\zeta > 0$ and constants $\beta_0 < \infty$ and $\rho_0 > 0$, such that
\bequ \label{ExpErg}
\sL (D(x(t), s),  D(x(t), \infty)) \le \beta_0 e^{- \rho_0 s} \qinq \RR
\eeq
holds for all $t \in [t_1\vee(t_0 - \zeta), t_0 + \zeta] \subset [t_1, t_2]$.
\end{lemma}


For our proof of Lemma \ref{lmExpErg}, we also
 introduce the concept of {\em uniformly} $(u, \af)-small$ sets on an interval, which generalizes the concept of $(u, \af)$-{\em small} sets;
for background see, e.g., \cite{KM03}, \cite{MT93} and \cite{RobRos96}.
For each (fixed) $t \ge 0$ let $P^t_{i,j}(u)$ denote the transition probabilities of  $\{D(x(t), u) : u \ge 0\}$:
$$P^t_{i,j}(u) \equiv P (D(x(t), u) = j \mid D(x(t), 0) = i),~~~ i,j \in \ZZ.$$

\begin{definition} \label{Def:small}{\em
A set $C \subseteq \ZZ$ is $(u, \af)$-{\em small}, for a time $u > 0$ and for some $\af > 0$, if there exists a probability measure $\varphi^t(\cdot)$
on $\ZZ$ satisfying the {\em minorization condition}:
\bes
P^t_{i, j}(u) \ge \af \varphi^t(j), \quad i \in C, ~ j \in \ZZ.
\ees
We say that a set $C$ is {\em uniformly $(u, \af)$-small} in the interval $I$,
if $C$ is $(u,\af)-small$ for each process in the family $\{D(x(t), \cdot) : t \in I\}$,
i.e., if $C$ is $(u,\af)$-{\em small} for each $D(x(t), \cdot)$ with the same $u$ and $\af$, $t \in I$.
}
\end{definition}
Note that the probability measure $\varphi^t$ in the minorization condition is allowed to depend on $t$
in the definition of {\em uniformly $(u, \af)$-small}.

\begin{proofof}{Lemma \ref{lmExpErg}}
As above, let $\tau (t)$ denote the return time of the process $D (x (t), \cdot)$ to the regeneration state $s^*$,
which, for concreteness, we take to be state $0$, i.e., $s^* = 0$.
Consider the infinitesimal generator matrix $Q (t)$ of the process $D (x (t), \cdot)$.
In a countable state space, every compact set is {\em small} (see, e.g., definition in pp. 11 in \cite{KM03})
and, in particular, $\{0\}$ is a small set. (We show in \eqref{minorization} below that $\{0\}$ is in fact {\em uniformly $(u, \af)$-small}.)
Moreover, by Theorem 2.5 in \cite{KM03}, the existence of a finite mgf for the hitting time of the small set $\{0\}$
is equivalent to the existence of a Lyapunov function $V : \ZZ \ra [1,\infty)$ which satisfies
the exponential drift condition on the generator, Condition $\bf (V4)$ in \cite{KM03}:
\bequ \label{ExpDrift}
Q (t) V \le -c_t V + d_t {\bf 1}_{\{0\}},
\eeq
where $c_t$ and $d_t$ are strictly positive constants.

Consider the time $t_0 \in [t_1, t_2]$. Then \eqref{ExpDrift} holds at $t_0$ with constants $c_{t_0}$ and $d_{t_0}$.
Since $c_{t_0} > 0$ we can decrease it such that \eqref{ExpDrift} holds with strict inequality and the new $c_{t_0}$
is still strictly positive.
We increase $d_{t_0}$ appropriately, such that $\sum_{j} q_{0,j}(t) V (j) < -c_{t_0} V (0) + d_{t_0} {\bf 1_{\{0\}}}$.
The continuity of $Q (t)$ on $[t_1, t_2]$ as a function of $t$, which follows immediately from the continuity of the rates
\eqref{bd1}-\eqref{bd4} as functions of the continuous function $x(t)$, implies that
there exist $\zeta > 0$ and two positive constants $c_0$ and $d_0$, such that \eqref{ExpDrift} holds for all
$t \in [t_0 - \zeta, t_0 + \zeta]$ with the same constants $c_0$ and $d_0$.
However, this is still not sufficient to conclude that the bounds in \eqref{ExpErg} are the same for all
$t \in [t_0 - \zeta, t_0 + \zeta]$; see Theorem 1.1 in \cite{Bax05} (for discrete-time Markov chains).

Recall that for each fixed $t$, $P^t_{i,j} (s)$ denotes the transition probabilities of the CTMC $\{D (x (t), s) : s \ge 0\}$.
We can establish uniform bounds on the convergence rates to stationarity by showing that $\{0\}$
is {\em uniformly $(u, \af)$-small} in an interval $[t_0-\zeta, t_0+\zeta]$ for the family
$\{D(x(t), \cdot) : t \in [t_0-\zeta, t_0+\zeta]\}$, as in Definition \ref{Def:small}.
In particular, we need to show that
\bequ \label{minorization}
P^t_{0,j} (u) \ge \af \varphi^t (j), \quad j \in \ZZ.
\eeq
holds for all $t \in [t_0 - \zeta, t_0 + \zeta]$ with the same $\af$ (but $\varphi^t$ is allowed to change with $t$).
This step, together with the uniform bounds $c_0$ and $d_0$ in \eqref{ExpDrift}
established above, will be shown to be sufficient to conclude the proof.

Hence, it is left to show that \eqref{minorization} holds for all $t \in [t_0 - \zeta, t_0 + \zeta]$ with the same $\af > 0$.
This step is easy because $\{0\}$ is a singleton in a countable state space.
Specifically, for each $t$ we consider, we can fix any $u > 0$ and define $\varphi^t (j) \equiv P^t_{0,j} (u)$.
With this definition of $\varphi^t$ we can take any $\af \le 1$ in \eqref{minorization}.
As in the discrete-time case in \cite{Bax05}
(the {\em strong aperiodicity condition} (A3) in \cite{Bax05} is irrelevant in continuous time),
the bounds on the convergence rates in \eqref{ExpErg}
depend explicitly on $\af$ in the minorization condition \eqref{minorization}, the bounds in the
drift condition \eqref{ExpDrift} and the Lyapunov function $V$ in \eqref{ExpDrift}.
This can be justified by uniformization, but can also be justified directly for continuous-time processes,
e.g., from the expressions in Theorem 3 and Corollary 4 in \cite{RobRos96}.

The uniform bounds on the rate of convergence to steady state established above by applying \cite{RobRos96}
are directly expressed in the total-variation metric.
If the total variation metric can be made arbitrarily small, then so can the Levy metric.
Hence we have completed the proof.
\end{proofof}

\begin{remark} {$($bounds on coupling times$)$} \label{remCoupling}
{\em The bounds on the rate of convergence to stationarity of Markov processes satisfying \eqref{ExpDrift} and \eqref{minorization}
in \cite{RobRos96} are obtained via the coupling inequality. In particular, we have provided explicit bounds for the time it takes
a positive recurrent FTSP $D(\gamma, \cdot)$ (with $\gamma \in \AA$), initialized at some finite time, to couple with its stationary version.}
\end{remark}

The final auxiliary lemma relates the L\'{e}vy distance $\mathcal{L}$ in \eqref{Levy} between $D^n_e(X^n(u), s_0)$ in \eqref{fast102} and the FTSP $D(x(u), s_0)$
at a finite time $s_0$.

\begin{lemma}\label{lmExpandUniform}
Suppose that $x (t) \in \AA$ for $t_1 \le t \le t_2$. Then  and for any fixed $s_0 > 0$,
\beql{expUnif}
\sL(D^n_e (X^n (u), s_0), D(x(u), s_0)) \ra 0 \quad \mbox{as $n\tinf$, uniformly in $u \in [t_1, t_2]$}.
\eeq
\end{lemma}

\begin{proof}
If follows from the proof of Theorem \ref{thConvToFast} that, for the given $\ep > 0$,
there exists $\delta(t_1)$, $n_0 (t_1)$ and $s_0(t_1)$ such that
\beq
\sL(D^n_e (X^n (v), s_0 (t_1)), D(x(v), s_0(t_1))) < \ep \qforallq  t_1 \le v \le (t_1 + \delta(t_1)) \wedge t_2 \qandq n \ge n_0(t_1).
\eeqno
exploiting the convergence $\barx^n \Ra x$ and $D^n (\Gamma^n, 0) = D^n_{1,2} (t_1) \Ra D(x(t_1), 0) = D(\gamma, 0)$
for $\Gamma^n \equiv X^n (t_1)$ and $\gamma \equiv x(t_1)$.
We can apply this reasoning in an open interval about each $u \in [t_1, t_2]$.
In particular, for the given $\ep > 0$ and $u \in [t_1, t_2]$,
there exists $\delta(u)$, $n_0 (u)$ and $s_0(u)$ such that
\beq
\sL(D^n_e (X^n (v), s_0 (u)), D(x(v), s_0(u))) < \ep \qforallq (u - \delta(u)) \vee t_1 \le v \le (u + \delta(u)) \wedge t_2 \qandq n \ge n_0(u).
\eeqno
However, since the interval $[t_1, t_2]$ is compact and the family of intervals $(t_1 \vee (u - \delta(u), (u + \delta(u)) \wedge t_2)$,
taken to be closed on the left at $t_1$ and closed on the right at $t_2$, is an open cover, there is a finite subcover.
Hence, all time points in $[t_1,t_2]$ are contained in only finitely many of these intervals.
Hence, we can achieve the claimed uniformity.  Moreover, since the conclusion does not depend on the subsequence used at the initial time $t_1$,
the overall proof is complete.
\end{proof}

\subsection{Proofs of Theorems \ref{thDSB} and \ref{thAPlocal}.} \label{secMainProofs2}   

We apply the results in \S \ref{secAuxFTSP} to prove the theorems.

\begin{proofof}{Theorem \ref{thDSB}}
First, if $x(t_0) \in \AA$ then there exists $t_2$ such that $x(t) \in \AA$ over $[t_0, t_2]$ because $\AA$ is an open set
and $x$ is continuous. It is possible that the fluid limit never leaves $\AA$ after time $t_0$, in which case $x \in \AA$ over $[t_0, \infty)$.
In the latter case we can consider any $t_2 > t_0$.

$(i)$
We begin by showing that there exists $\delta_0 \equiv \delta (t_0) > 0$,
such that $D^n_{1,2}(t_1)$ is tight in $\RR$ for all $t_1$ satisfying $t_0 < t_1 < t_0 + \delta_0$.
Henceforth, $t_1$ denotes such a time point.
We need to show that, for any $\ep >0$, there exists a constant $K$ such that
$P(|D^n_{1,2} (t_1)| > K) < \ep$ for all $n \ge 1$. Overall we prove the claim in three steps:
in the first step we bound $D^n_{1,2}$ over $[t_0, t_0 + \delta_0]$, in sample-path stochastic order, with positive recurrent QBD's
but with random initial conditions.
In the second step, we show that these bounding QBD's are tight for each $t_1$ as above, by showing that they couple with stationary versions
rapidly enough.
In the third step we show that we can extend the conclusion from the subinterval $[t_1, t_0 + \delta_0]$ to the entire interval $[t_1, t_2]$.

{\em Step One.}
We apply Lemma \ref{lmD12PR} to find a $\delta_0 > 0$
and construct random states $X^n_m$ and $X^n_M$
with the properties stated there, such that rate order holds as in \eqref{DstoBd}
in $\D([t_0, t_0 + \delta_0])$, after we let $D^n_f(X^n_m, t_0) \equiv D^n_f(X^n_M, t_0) \equiv D^n_{1,2} (t_0)$.
Next, for all sufficiently large $n$, we bound the upper bounding process above and the lower bounding process below in rate order,
each by a FTSP with fixed states $\tilde{x}_m$ and $\tilde{x}_M$
 but ordered so that strict rate order in Lemma \ref{lmRateOrder} holds.
Since $x_m$ and $x_M$ are in $\AA$ and $\AA$ is open, these new states $\tilde{x}_m$ and $\tilde{x}_M$
can be chosen to be sufficiently near the initial states $x_m$ and $x_M$
that they too are in $\AA$.
We let these FTSP's be given the same initial conditions depending on $n$, consistent with above.
For that purpose, we put the initial condition in the notation; i.e., we
let $D(\gamma, b, \cdot) \equiv D(\gamma, \cdot)$ given that $D(\gamma, 0) = b$.
We combine the rate order just established with Corollary \ref{corStochBd} to obtain the sample-path stochastic order
\begin{eqnarray}\label{DstoBd30}
&& D(\tilde{x}_m, D^n_{1,2} (0), nt) - \zeta \le_{st} D^n_{1,2} (t)
\le_{st} D(\tilde{x}_M, D^n_{1,2} (0), nt) + \zeta
\end{eqnarray}
for all $n \ge n_0$ for some $n_0$.

So far, we have succeeded in bounding $D^n_{1,2} (t)$ above (and below) stochastically by
a positive recurrent FTSP with random initial conditions, in particular starting at $D^n_{1,2} (0)$.
It suffices to show that each bounding FTSP with these initial conditions in stochastically bounded
at any time $t_1 \in (t_0, t_0+\delta_0]$.  We show that next.

{\em Step Two.}
We now show that the two bounding processes in \eqn{DstoBd30} are indeed stochastically bounded at such $t_1$.
We do that by showing that they couple with the stationary versions of these FTSP's with probability converging to $1$ as $n \ra \infty$.
Since $x(t_0) \in \AA$, we can apply Theorem \ref{th1} to deduce that $\barx^n (t_0) \Ra x(0)$ as $n \ra \infty$,
which implies that $D^n_{1,2} (0) = o(n)$ as $n \ra \infty$.
If we consider the FTSP without the scaling by $n$ in \eqn{DstoBd30}, it suffices to show, for $t_1$ as above,
the bounding processes recover from the $o(n)$ initial condition within $o(n)$ time, which will be implied by showing that
the coupling takes place in time $o(n)$.

We show that by applying Lemma \ref{lmExpErg} and Remark \ref{remCoupling}.
They imply that the time until the FTSP in \eqref{DstoBd30} couples with its stationary version is bounded by constants depending on
$c_{t_0}$ and $d_{t_0}$ in \eqref{ExpDrift} and $\af$ in \eqref{minorization}, as well as the Lyapunov function $V$ in \eqref{ExpDrift}.
Now, $D(\tilde{x}_M, t_0) = D^n_{1,2}(t_0)$ by construction, and $x(t_0) \in \AA$, so that $D(\tilde{x}_M, t_0) = o(n)$.
Hence, $V (D(\tilde{x}_M, t)) = o(n)$ for all $t \ge t_0$, for $V$ in \eqref{ExpDrift}.
In particular, we see that the time until $D(\tilde{x}_M, t_0)$ couples with it stationary version, when initialized
at $D^n_{1,2} (t_0)$, is $o(n)$.
Together with \eqref{DstoBd30}, that implies the claim.
That is, we have shown that $D^n_{1,2}(t_1)$ is tight in $\RR$ for all $t_1$, $t_0 < t_1 < t_0 + \delta_0$.

{\em Step Three.}  To extend the result to the interval $[t_1, t_2]$
we observe that we can repeat the reasoning above for any starting point $t \in [t_1, t_2]$ (since $x(t) \in \AA$ for all $t \in [t_1, t_2]$),
achieving an uncountably-infinite cover for $[t_1, t_2]$ of intervals of the form $[t, t + \delta(t)]$.
Since the interval $[t_1, t_2]$ is compact, the uncountably infinite cover
of these intervals, made open at the left unless the left endpoint is $t_1$ and made open on the right unless
the right endpoint is $t_2$, has a finite subcover.  As a consequence, the entire interval
$[t_1, t_2]$ is covered by only finitely many of these constructions, and
we can work with the finite collection of closures of these intervals.

In particular, since the sequence $\{D^n_{1,2} (t_1): n \ge 1\}$ is stochastically bounded,
we can apply the construction over an interval of the form $[t_1, t']$ for $t_1 < t' \le t_2$.
As a consequence we obtain stochastic boundedness for each $t$ in $[t_1, t']$.
If $t' < t_2$, then we continue. We then can choose a second interval $[t'',t''']$
such that $t_1 < t'' \le t' < t'''$.  We thus can carry out the construction over $[t'',t''']$.
Since $t'' \le t'$, we already know that the sequence of random variables $\{D^n_{1,2} (t''); n \ge 1\}$ is stochastically bounded
from the first step.  Thus, in finitely many steps, we will deduce the first conclusion in $(i)$.
If $x(0) \in \AA$ and we take $t_0 = 0$, then $D^n_{1,2}(t_0)$ is tight in $\RR$ because $D^n_{1,2}(0) \Ra L$ by Assumption \ref{assC},
so the result holds on $[t_0, t_2] \equiv [0, t_2]$.

$(ii)$ To prove the statement in $(ii)$ we observe that Theorem \ref{thConvToFast} implies that the oscillations are asymptotically too rapid
for the sequence $\{D^n_{1,2}: n \ge 1\}$ to be tight in $\D$ over any finite interval.  Moreover,
there is a finite interval over which a single frozen-difference process serves as a lower bound, by Lemma \ref{lmD12PR}.
Because of the scaling by $n$, the maximum in the lower bound QBD over this interval will be
unbounded above (actually of order $O(\log{n})$ by reasoning as in Lemma \ref{lmQBDextreme}).
Thus, the sequence of stochastic processes $\{D^n_{1,2}: n \ge 1\}$ is not even stochastically bounded over any finite subinterval of $[t_0, t_2]$.

$(iii)$ We apply Lemma \ref{lmQBDextreme} to establish \eqn{DlogBd}.  We can bound the supremum of $D^n_{1,2} (t)$
over the interval $[t_1, t_2]$ by the supremum of the finitely many frozen queue-difference processes.
Since the rates are of order $n$, Lemma \ref{lmQBDextreme} implies \eqn{DlogBd}.
Once again, the result can be extended to hold on $[t_0, t_2] \equiv [0, t_2]$ if $x(0) \in \AA$ by the assumed convergence
$D^n_{1,2}(0) \Ra L$ in Assumption \ref{assC}.
\end{proofof}

\begin{proofof}{Theorem \ref{thAPlocal}}
For any given $t$ with $t_1 < t \le t_2$ and $\ep > 0$, we will show that we can choose $n_0$ such that
$\sL(D^n_{1,2} (t), D(x(t), \infty)) < \ep$ for all $n \ge n_0$, where $\sL$ is the L\'{e}vy metric in \eqn{Levy}.
Since $\{D^n_{1,2} (t_1): n \ge 1\}$ is stochastically bounded, it is tight.
Hence, we start with a converging subsequence, without introducing subsequence notation.
The result will not depend on the particular converging subsequence we choose.

Hence, we start with $D^n_{1,2} (t_1) \Ra L$, where $L$ is a proper (almost surely finite) random variable.
We will then let $D(x(t_1), 0) = L$ to obtain $D^n_{1,2} (t_1) \Ra D(x(t_1), 0)$, as required to apply
Theorem \ref{thConvToFast} at $t_1$.
If $t > t_1$, then for all sufficiently large $n$, $t- (s/n) \ge t_1$,
in which case we can write $D^n_{1,2} (t) = D^n_{1,2} (t - (s/n) +(s/n)) = D^n_e (X^n (t - (s/n), s)$, where
$D^n_e$ is the expanded queue-difference process defined in \eqn{fast102}, with $t - (s/n) \ge t_1$ so that $x(t- (s/n)) \in \AA$.
Hence we can write
\begin{eqnarray}\label{triangle}
&& \sL(D^n_{1,2} (t), D(x(t), \infty)) \le \sL(D^n_e (X^n (t - (s/n), s), D(x(t - (s/n)), s)) \nonumber \\
&& \quad +  \sL(D(x(t - (s/n)), s), D(x(t - (s/n)), \infty))
+  \sL(D(x(t - (s/n)), \infty), D(x(t), \infty)).
\end{eqnarray}
For $t$ and any $\ep > 0$ given, we first choose $\delta$ so that $t - \delta \ge t_1$ and $x$ remains in $\AA$ throughout $[0, t + \delta]$,
which is always possible because $x(t) \in \AA$, $t > t_1$, $x$ is continuous and $\AA$ is an open subset of $\rS$.
(We use the condition that $t > t_1$ here.) In addition, we choose $\delta$ sufficiently small that
the third term in \eqn{triangle} is bounded as
\bes 
\sL(D(x(u), \infty), D(x(t), \infty) < \ep/3 \qforallq u, \quad t - \delta \le u \le t+ \delta,
\ees
which is possible by virtue of Lemma \ref{lmContSS}.
Given $t$, $\ep$ and $\delta$, we choose $s_0$ sufficiently large that
the second term in \eqn{triangle} is bounded as
\bes 
\sL(D(x(u), s_0), D(x(u), \infty) < \ep/3 \qforallq u, \quad t - \delta \le u \le t+ \delta,
\ees
which is possible by Lemma \ref{lmExpErg}.
Finally, we choose $n_0$ sufficiently large that $s_0/n < \delta$ for all $n \ge n_0$ and
the first term in \eqn{triangle} is bounded as
\bes 
\sL(D^n_e (X^n (u, s_0), D(x(u), s_0)) < \ep/3 \qforallq u, \quad t - \delta \le u \le t+ \delta,
\ees
which is possible by Lemma \ref{lmExpandUniform}.
That choice of $\delta$, $s_0$ and $n_0$ makes each term in \eqn{triangle} less than or equal to $\ep/3$ for all $n \ge n_0$,
thus completing the proof.
\end{proofof}

\section{Conclusion and Further Research.}

In this paper we proved a FWLLN (Theorem \ref{th1}) for an overloaded X model operating under the fixed-queue-ratio-with-thresholds (FQR-T) control
(\S \ref{secFQRorig}), with many-server heavy-traffic scaling (\S \ref{secHT}).
Theorem \ref{th1} shows that the fluid-scaled version of the
six-dimensional Markov chain $X^n_6$,
whose sample-path representation appears in Theorem \ref{thRep}, converges to a deterministic limit,
characterized by the
unique solution to the three-dimensional ODE \eqref{odeDetails},
which in turn is driven by the fast-time-scale stochastic process (FTSP) in \S \ref{secFTSP}.
We also proved a WLLN for the stationary distributions (Theorem \ref{thStatLim}) that justifies a limit interchange in great generality (Theorem \ref{thIntchange}).

Finally, in \S \ref{secQD} we presented results regarding the queue difference process and SSC for the queues when the fluid limit is in $\AA$.
in particular, Theorem \ref{thDSB} proved statements regarding the tightness of $D^n_{1,2}(t)$ in $\RR$, and non-tightness in $\D$.
Corollary \ref{corSSCfull} shows that SSC for the queues hold under any scaling larger than $O(\log n)$, depending only whether the fluid limit
$x$ is in $\AA$. Theorem \ref{thAPlocal} establishes a pointwise AP result, which is not an immediate corollary of the AP in the FWLLN.

{\em Proof of the FWLLN.} We proved the FWLLN in three steps, showing that:
(i) the sequence of processes $\{\barx^n_6: n \ge 1\}$ is tight and every limit is continuous (Theorem \ref{lmTight});
(ii) simplifying the representation in Theorem \ref{thRep} to the essentially three-dimensional
representation in \eqref{FluidScaled} (Corollary \ref{thSSCextend}); and
(iii) characterizing the limit via the averaging principle (AP) (\S \ref{secProofs}).
Characterizing the fluid limit of that three-dimensional process was challenging because the sequence of queue-difference processes
$\{\{D^n_{1,2} (t): t \ge 0\}: n \ge 1\}$ in \eqref{Dprocess} does not converge to any limiting process as $n \ra \infty$.
Instead, we have the AP, which can be better understood through
Theorems \ref{thConvToFast}-\ref{thAPlocal}.

Due to the AP, the indicator functions $1_{\{D^n_{1,2}(s) > 0\}}$ and $1_{\{D^n_{1,2}(s) \le 0\}}$
in the representation \eqref{FluidScaled} are replaced in the limit with appropriate steady-state quantities
related to the FTSP $D(x(t), \cdot)$, e.g., $1_{\{D^n_{1,2}(s) > 0\}}$ is replaced in the limit with
\bes
\pi_{1,2}(x(s)) \equiv P(D(x(s), \infty) > 0) = \lim_{t \tinf} \frac{1}{t}\int_0^t 1_{\{D(x(s), u) > 0\}} du.
\ees

Our proof of the AP in \S \ref{secProofs} is based on the framework established by Kurtz in \cite{K92}.
In the appendix we present a different proof of the AP which is
weaker, since it only characterizes the FWLLN in the set $\AA$. However, it has the advantage of being intuitive,
with an explicit demonstration of the separation of time scales which takes place in $\AA$.
Moreover, this proof continues the bounding logic of \S\S \ref{secSupport} and \ref{secProofsSSCserv}, and thus follows naturally from
previous results established in the paper.
Both proofs have merits, and can be useful in other models where separation of time scales occur in the limit.

{\em Results for QBD Processes.} In the process of proving the SSC in Theorem \ref{thSSCextend} and Theorems \ref{thDSB} and \ref{thAPlocal},
we established some general results for QBD's, which are interesting in their own right. First, we established an extreme-value result
in Lemma \ref{lmQBDextreme}. Then, in \S \ref{secAuxFTSP} we established the continuity of stationary distributions (Lemma \ref{lmContSS}),
finite mgf for return times (Lemma \ref{lmReturnTime}) and uniform ergodicity for a family of ergodic QBD processes (Lemma \ref{lmExpErg}).
To the best of our knowledge, those results are not stated in the existing literature.

{\em The Control.} The FQR-T control is appealing, not only because it is optimal during the overload incident in the fluid limit \cite{PeW09a},
but also because the FQR-T control produces significant simplification of the limit, since it produces
a strong form of state space collapse (SSC).
Specifically, by Theorem \ref{thSSCextend}, the four-dimensional service process is asymptotically one dimensional,
with $Z^n_{1,2}$ alone characterizing limits under any scaling.
By Theorem \ref{thDSB}, the two-dimensional queue process is asymptotically one dimensional under any scaling larger than $O(\log n)$
(in particular, under fluid and diffusion scaling).
This latter result holds unless class $1$ is so overloaded, that there is not enough service capacity in both pools
to keep the desired ratio between the two queues (in which case the fluid limit will not be in $\AA$).
Hence, the six-dimensional Markov chain describing the system during overloads in the prelimit, is replaced by a simplified
lower dimensional deterministic function, which is easier to analyze.
In addition, this SSC results are crucial to proving the FCLT for the system \cite{PeW10}.

{\em Further Research.}
As Lemma \ref{lmZ12} and the proof of Theorem \ref{thSSCextend} reveal, one-way sharing, which we suggested in \cite{PeW09a} as a means to prevent
unwanted simultaneous sharing of customers in finite systems (where the thresholds themselves may not be sufficient to prevent two-way sharing)
has shortcomings; see Remark \ref{remShare}.
For example, in the limiting fluid model, if after some time the overload switches over, so that queue $2$ should start receiving help from pool $1$,
then the one-way sharing rule will not allow class-$2$ customers to be sent to pool $1$ since,
as was shown in the proof of Theorem \ref{thSSCextend}, $z_{1,2} (t) > 0$ for all $t > 0$.
As a consequence, in large systems, a significant amount of time must pass before sharing in the opposite direction is allowed.
This problem with one-way sharing can be remedied by dropping that rule completely, or by introducing lower thresholds on the service process; again
see Remark \ref{remShare}.
We have begun studying such alternative controls.


\begin{appendix}

\begin{center}
\textbf{Appendix}




\end{center}


This Appendix has six sections.  In \S \ref{secSupport} we establish important technical results to be used to prove the results in \S \ref{secSSCserv}.
In \S \ref{secProofsSSCserv} we apply the results in \S \ref{secSupport} to prove the results in \S \ref{secSSCserv}.
In \S \ref{secAltProof} we present an alternative proof of Lemma \ref{lmKey} and thus
an alternative proof of the FWLLN in Theorem \ref{th1}.
In \S \ref{appQBDbdd} we show how to present the process $D^n_*$
in step one of the proof of Lemma \ref{lmD21extreme} as a QBD for each $n$.
In \S \ref{appPosQ} we explain why Assumption \ref{assC} about the initial conditions is reasonable.
Finally, in \S \ref{AppAcro} we list all the acronyms used in this paper.

\section{Supporting Technical Results.}\label{secSupport}

In this section we establish supporting technical results that we will apply to prove the results in \S \ref{secSSCserv}.
In \S \ref{secAux} we introduce a frozen queue difference process, which
`freezes'' the state of the slow process ($\barx^n$ in \eqn{FluidScaled}),
so that the fast process (the queue difference process $D^n_{1,2}$ in \eqn{Dprocess}) can be considered separately.
Like the FTSP in \S \ref{secFTSP}, the frozen process is a pure-jump Markov process.
To carry out the proofs, we exploit stochastic bounds.
Hence, we discuss them next in \S\S \ref{secStoBd}, \ref{secRateOrder} and \ref{secAbanBd}.
Since the FTSP and the frozen processes can be represented as QBD processes,
as indicated in \S 6 of \cite{PeW09c},
we next establish extreme value limits for QBD's in \S \ref{secQBDextreme}.
We establish continuity results for QBD's, used in the alternative proof of characterization, in \S \ref{secCont}.

\subsection{Auxiliary Frozen Processes.} \label{secAux}

Our proof of Theorem \ref{thZ21} will exploit stochastic bounds for
the queue difference process $D^n_{2,1}$.  Our alternative proof of Theorem \ref{th1} in the appendix
we will exploit similar stochastic bounds for the queue difference process
$D^n_{1,2}$.  These processes $D^n_{2,1}$ and $D^n_{1,2}$ are
 non-Markov processes whose rates at each time $t$ are determined by the state of the system at time $t$, i.e., by $X^n_6 (t)$,
and is thus hard to analyze directly.
In order to circumvent this difficulty,
 we consider related auxiliary processes, with {\em constant} rates determined by a fixed initial state $\Gamma^n \equiv X^n_6(0)$.
 The construction is essentially the same for the two processes $D^n_{2,1}$ and $D^n_{1,2}$.
Since we are primarily concerned with $D^n_{1,2}$, we carry out the following in that context,
with the understanding that there is a parallel construction for
the process $D^n_{2,1}$.

When we work with $D^n_{1,2}$, we exploit the reduced representation involving $X^n$ in \eqn{FluidScaled}.
Let $D^n_f (\Gamma^n) \equiv \{D^n_f (\Gamma^n, t): t \ge 0\}$ denote this new process with fixed state $\Gamma^n \equiv (Q^n_1, Q^n_2, Z^n_{1,2})$.
Conditional on $\Gamma^n$, $D^n_f(\Gamma^n)$ is a QBD with the same fundamental structure as the FTSP defined in \S \ref{secFTSP}.
We use the subscript $f$ because we refer to this constant-rate pure-jump Markov process as the {\em frozen queue-difference process},
or alternatively, as the {\em frozen process},
thinking of the constant transition rates being achieved because the state has been frozen at
the state $\Gamma^n$.

As in \S \ref{secFTSP}, the rates of the frozen process are determined by $\Gamma^n \equiv (Q^n_1, Q^n_2, Z^n_{1,2})$
and by its position.  Also, as in \S \ref{secFTSP}, let $r \equiv r_{1,2}$.
The frozen process $D^n_f(\Gamma^n)$ has jumps of size $+r, +1, -r, -1$ with respective rates
$\lm^{(r)}_+(n,\Gamma^n)$, $\lm^{(1)}_+(n, \Gamma^n)$, $\mu^{(r)}_- (n, \Gamma^n)$, $\mu^{(1)}_- (n, \Gamma^n)$ when $D^n_f \le 0$,
and jumps of size $+r, +1, -r, -1$ with respective rates
$\lm^{(r)}_+(n,\Gamma^n)$, $\lm^{(1)}_+(n, \Gamma^n)$, $\mu^{(r)}_+(n, \Gamma^n)$, $\mu^{(1)}_+ (n, \Gamma^n)$ when $D^n_f > 0$.

Analogously to \eqref{bd1}-\eqref{bd4}, in the non-positive state space these rates are equal to
\bequ \label{BDfrozen1}
\begin{split}
\lm^{(1)}_- (n, \Gamma^n)  & \equiv  \lm^n_1 \qandq
\lm_-^{(r)} (n, \Gamma^n)  \equiv  \mu_{1,2}Z^n_{1,2} + \mu_{2,2} (m^n_2 - Z^n_{1,2}) + \theta_2 Q^n_2, \\
\mu^{(1)}_- (n, \Gamma^n) & \equiv  \mu_{1,1}m^n_{1}  + \theta_1 Q^n_1 \qandq \mu_-^{(r)} (n, \Gamma^n) \equiv \lm^n_2,
\end{split}
\eeq
and in the positive state space, these rates are equal to
\bequ \label{BDfrozen2}
\bsplit
\lm^{(1)}_+ (n, \Gamma^n) & \equiv \lm^n_1 \qandq  \lm_{+}^{(r)} (n, \Gamma^n) \equiv \theta_2 Q^n_2, \\
\mu^{(1)}_+ (n, \Gamma^n) & \equiv  \mu_{1,1}m^n_1  + \mu_{1,2}Z^n_{1,2} + \mu_{2,2}(m^n_2 - Z^n_{1,2}) + \theta_1 Q^n_1 \qandq
\mu_+^{(r)} (n, \Gamma^n) \equiv \lm^n_2.
\end{split}
\eeq

Using these transition rates, we can define the {\em drift rates} for $D^n (\Gamma^n)$, paralleling \eqn{drifts}.
Let these drift rates in the regions $(0, \infty)$ and $(-\infty, 0]$ be denoted by $\delta^n_+ (\Gamma^n)$ and $\delta^n_- (\Gamma^n)$, respectively,
Then
\bequ \label{driftDn}
\bsplit
\delta^n_+ (\Gamma^n) & \equiv  [ \lm^n_1 - \mu_{1,1} m^n_1 + (\mu_{2,2} - \mu_{1,2}) Z^n_{1,2}(t) - \mu_{2,2} m^n_{2}(t) - \theta_1 Q^n_1(t) ]
- r [ \lm^n_2 - \theta_2 Q^n_2(t) ], \\
\delta^n_- (\Gamma^n) & \equiv  [\lm^n_1 - \mu_{1,1}m^n_1 - \theta_1 Q^n_1(t) ]
- r [\lm^n_2 + (\mu_{2,2} - \mu_{1,2}) Z^n_{1,2}(t) - \mu_{2,2}m^n_2 - \theta_2 Q^n_2(t) ].
\end{split}
\eeq
Just as for the FTSP, $D^n_f(\Gamma^n)$ is, conditional on $\Gamma^n$, a positive recurrent QBD
if an only if 
\bequ \label{posrecDn}
\delta^n_+(\Gamma^n) < 0 < \delta^n_-(\Gamma^n).
\eeq

The constant-rate pure-jump Markov process $D^n_f(\Gamma^n)$ will frequently appear with $\Gamma^n$ being a state of some process, such as $X^n (t)$,
We then write $D^n_f (X^n (t)) \equiv \{D^n_f (X^n (t), s): s \ge 0\}$, where it is understood that
$D^n_f (X^n (t)) \deq  D^n_f (\Gamma^n)$ under the condition that $\Gamma^n \deq X^n (t)$.
It is important that this frozen difference process $D^n_f (\Gamma^n)$ can be directly identified with
a version of the FTSP defined in \S \ref{secFTSP}, because both are pure-jump Markov processes with the same structure.  Indeed,
the frozen-difference process can be defined as a version of the FTSP with special state and basic model parameters
$\lambda^n_i$ and $m^n_j$, and transformed time.
In order to express the relationship, we first introduce appropriate notation.
Let $D(\lambda_i, m_j, \gamma, s)$ denote the FTSP defined as in \S \ref{secFTSP}
with transition rates in \eqn{bd1}-\eqn{bd4}
as a functions of the arrival rates $\lambda_i$, $i = 1,2$ and the staffing levels $m_j$, $j = 1,2$,
as well as the state $\gamma$, where $s$ is the time parameter as before.  (We now will allow the parameters $\lambda_i$ and $m_j$ to vary as well as the state.)
With that new notation, we see that the frozen process is
equal in distribution to the corresponding FTSP with new parameters,
in particular,
\bequ \label{ident}
\{D^n_f (\lambda^n_i, m^n_j, \Gamma^n, s): s \ge 0\} \deq \{D(\lambda^n_i/n, m^n_j/n, \Gamma_n/n, n s): s\ge 0\},
\eeq
with the understanding that the initial differences coincide, i.e.,
\bes 
D(\lambda^n_i/n, m^n_j/n, \Gamma_n/n, 0) \equiv D^n_f (\lambda^n_i, m^n_j, \Gamma^n, 0). 
\ees
This can be checked by verifying that the constant transition rates are indeed identical for the two processes,
referring to \eqn{bd1}-\eqn{bd4} and \eqref{BDfrozen1}-\eqref{BDfrozen2}.
Since $\lambda^n_i/n \ra \lambda_i$, $i = 1,2$ and $m^n_j/n \ra m_j$, $j = 1,2$,
by virtue of the many-server heavy-traffic scaling in \eqn{MS-HTscale}, we will have the transition rates
of $D(\lambda^n_i/n, m^n_j/n, \Gamma_n/n, \cdot)$ converge to those of $D(\gamma) \equiv D(\lambda_i, m_j, \gamma, \cdot)$
whenever $\Gamma_n/n \ra \gamma$.

\subsection{Bounding Processes.}\label{secStoBd}

We will use bounding processes in our proof of Theorem \ref{thZ21} and later results.  We construct the bounding processes
so that they have the given initial conditions at time $0$
and satisfy FWLLN's with easily determined continuous fluid limits,
coinciding at time $0$.  We thus control the initial behavior.

We first construct w.p.1 lower and upper bounds for $Z^n_{1,2}$.
Recall that $N^s_{i,j}$, $i,j = 1,2$, are the independent rate-$1$
Poisson processes used in \eqn{rep1}.  Let
\bequ \label{Zbounds} \bsplit
Z^n_a (t) & = Z^n_{1,2}(0) - N^s_{1,2} \left( \mu_{1,2} \int_0^t Z^n_{a}(s)\ ds \right),\\
Z^n_b (t) & = Z^n_{1,2}(0) + N^s_{2,2} \left( \mu_{2,2} \int_0^t (m^n_2 - Z^n_b(s))\ ds \right).
\end{split}
\eeq

\begin{lemma}\label{lmZbds}
For all $n \ge 1$ and $t \ge 0$, $Z^n_a (t) \le Z^n_{1,2} (t)  \le Z^n_b (t)$ w.p.1.
\end{lemma}

\begin{proof}
The bounding processes $Z^n_a$ and $Z^n_b$ are both initialized
as $Z^n_{1,2}(0)$ at time $0$.  They are defined in terms of the same rate-$1$ Poisson processes as $Z^n_{1,2}$,
so that the three processes can be compared for each sample path.
The lower-bound process $Z^n_a$ is the pure death process
obtained by routing no new class-$1$
customers to pool $2$ and letting all initial ones depart after receiving service. Hence, $Z^n_a$ is decreasing
monotonically to $0$.
The upper-bound process $Z^n_b$ is a pure birth process obtained by having pool $2$ not serve any of its initial
class $1$ customers and by assigning every server in pool $2$ that completes service
of a class $2$ customer to a new class $1$ customer, assuming that such customers are always available.
Hence, $Z^n_b$ is increasing
monotonically to $m^n_2$.  The given process $Z^n_{1,2}$ necessarily falls in between, where $Z^n_{1,2}$ is defined in \eqn{rep3}
with the asterisk omitted.  That is so
 because, whenever $Z^n_{1,2}$ is equal to
$Z^n_b$, every jump up in $Z^n_{1,2}$ is also a jump up in $Z^n_b$, but not vice versa;
whenever $Z^n_{1,2}$ is equal to $Z^n_a$ every jump down in $Z^n_{1,2}$
is also a jump down in $Z^n_a$, but not vice versa. This is because jumps are generated by the same Poisson processes,
but for $Z^n_{1,2}$ jumps occur only if the indicator functions in the respective Poisson processes are equal to $1$.
\end{proof}

We next construct w.p.1 lower and upper bounds for $Q^n_i$. The
upper bound processes will have the specified arrivals but no
departures, while the lower bound process will have no arrivals
but maximum possible departures. Both processes will start at
the initial values.

Let
$N^a_i$, $N^s_{i,j}$ and $N^u_i$ be the previously specified independent rate-$1$ Poisson processes used in \eqn{rep1}.  Let
\begin{eqnarray} \label{Qbounds}
Q^n_{i,a}(t) & = & Q^n_{i}(0) - \sum_{i = 1}^{2} \sum_{j = 1}^{2} N^s_{i,j} ( \mu_{i,j} m^n_j t)
- N^u_i (\theta_i Q^n_i (0)) t ), \nonumber \\
Q^n_{i,b}(t) & = & Q^n_{i}(0) + N^a_i(\lm^n_i t), \quad t \ge 0.
\end{eqnarray}

\begin{lemma}\label{lmQbds}
For all $n \ge 1$ and $t \ge 1$,
$$(Q^n_{1,a} (t), Q^n_{2,a}(t))  \le (Q^n_{1}(t), Q^n_2 (t))   \le (Q^n_{1,b} (t), Q^n_{2,b} (t) ) \quad \mbox{w.p.1.}$$
\end{lemma}

\begin{proof}
Just as in Lemma \ref{lmZbds}, we get a w.p.1 comparison because we construct both systems using the same rate-$1$ Poisson processes.
The upper bound is immediate, because the two systems being compared have the same initial value and the same arrivals,
but the upper bound system has no service completions or abandonments.
For the lower bound, we separately consider the fate of new arrivals, customers initially in service and customers initially in queue.
However, we allow the identity of departing customers to shift, which does not affect the result.
In the lower bound system new arrivals never enter, so they necessarily leave sooner.
Allowing for identity shift, customers initially in queue abandon as rapidly as possible in the lower bound system,
since the rate is fixed at the initial rate $\theta_i Q^n_i (0)$.  Of course, some customers from queue may enter service.
But all customers in service leave at least as quickly in the lower bound system because all servers are working continuously.
In the lower bound system we act as if all servers in each pool are simultaneously serving customers of both classes.
Hence, we do not need to pay attention to the service assignment rule.
The identity of customers may change in this comparison, but the order will hold for the numbers.  Hence the proof is complete.
\end{proof}

Unlike the processes in Corollary \ref{corRep3}, we can easily establish stochastic-process limits for the associated
fluid-scaled bounding processes, and these continuous limits coincide at time $0$.
Let $X^n_{a,b} \equiv (Q^n_{1,a}, Q^n_{2,a},Z^n_a, Q^n_{1,b}, Q^n_{2,b}, Z^n_b)$.

\begin{lemma}\label{lmBoundsLim}
as $n \tinf$,
$\barx^n_{a, b} \Ra x_{a, b}$ in $\D_6$,
where $x_{a,b} \equiv (q_{1,a}, q_{2,a}, z_a, q_{1,b}, q_{2,b}, z_b)$ is an element of $\C_6$ with
\begin{eqnarray}
z_a(t) & = & z_{1,2}(0) - \mu_{1,2} \int_0^t z_a(s)\ d s, \nonumber \\
z_b(t) & = & z_{1,2}(0) + \mu_{2,2} \int_0^t z_b(s)\ d s, \label{zbounds} \\
q_{i,a}(t) & = & q_{i}(0) - \sum_{i = 1}^{2} \sum_{j = 1}^{2} \mu_{i,j} m_j t  + \theta_i q_{i}(0)) t,   \nonumber \\
q_{i,b}(t) & =  &  q_{i}(0) + \lm_i t, \quad t \ge 0. \label{qzbounds}
\end{eqnarray}
\end{lemma}

\begin{proof}
The stated convergence is a relatively simple application of the continuous mapping theorem. In particular,
we first exploit the continuity of the integral representation, Theorem 4.1 in \cite{PTW07}, to establish the convergence
$(\barz^n_a, \barz^n_b) \Rightarrow (z_a, z_b)$.  The queue length bounds are simple linear functions.
\end{proof}

We will apply the bounding results above in \S \ref{secProofThZ21},
starting with
the proof of Lemma \ref{thZ21}.

\subsection{Rate Order for FTSP's.}\label{secRateOrder}

Given that we can represent frozen processes as FTSP's with appropriate state parameters $\gamma$, as shown in
\eqn{ident} above, it is important to be able to compare
FTSP's with different state parameters.  We establish such a comparison result here for the
FTSP in \S \ref{secFTSP} using rate order.  We say that one pure-jump Markov process $Y_1$ is less than or equal to another $Y_2$ in rate order, denoted by $Y_1 \le_r Y_2$,
if all the upward transition rates (with same origin and destination states) are larger in $Y_2$ and all the downward transition rates
(with same origin and destination states) are larger for $Y_1$.
The following lemma is an immediate consequence of the definition of rates in \eqn{bd1}-\eqn{bd4}.

\begin{lemma}{$($rate order for FTSP's$)$}\label{lmRateOrder}
Consider the FTSP in \S {\em \ref{secFTSP}} for candidate states
$\gamma^{(i)} \equiv (q^{(i)}_1, q^{(i)}_2, z^{(i)}_{1,2})$, $i = 1,2$.
$($a$)$ If $\mu_{1,2} \ge \mu_{2,2}$ and
\beq
(-q^{(1)}_1, q^{(1)}_2, z^{(1)}_{1,2}) \le (-q^{(2)}_1, q^{(2)}_2, z^{(2)}_{1,2}) \qinq \RR^3,
\eeqno
then
\beq
D(\gamma^{(1)}, \cdot) \le_r D(\gamma^{(2)}, \cdot).
\eeqno
$($b$)$ If $\mu_{1,2} \le \mu_{2,2}$ and
\beq
(-q^{(1)}_1, q^{(1)}_2, -z^{(1)}_{1,2}) \le (-q^{(2)}_1, q^{(2)}_2, -z^{(2)}_{1,2}) \qinq \RR^3,
\eeqno
then
\beq
D(\gamma^{(1)}, \cdot) \le_r D(\gamma^{(2)}, \cdot).
\eeqno
\end{lemma}

We will apply the rate order to get sample path stochastic order,
 involving coupling; see \cite{KKO77,W81}, Ch. 4 of
\cite{L92} and \S 2.6 of \cite{MS02}.   We briefly discuss
those bounds for a sequence of stochastic processes $\{Y^n : n
\ge 1\}$. We will bound the process $Y^n$, for each $n \ge 1$,
by a process $Y^n_b$; i.e., for each $n$, we will establish
conditions under which it is possible to construct stochastic
processes $\tilde{Y}^n_b$ and $\tilde{Y}^n$ on a common
probability space, with $\tilde{Y}^n_{b}$ having the same
distribution as $Y^n_{b}$, $\tilde{Y}^n$ having the same
distribution as $Y^n$, and every sample path of
$\tilde{Y}^n_{b}$ lies below (or above) the corresponding
sample path of $\tilde{Y}^n$. We will then write $Y^n_b
\le_{st} (\ge_{st}) Y^n$. However, we will not introduce this
``tilde'' notation; Instead, we will use the original notation
$Y^n$ and $Y^n_b$. As a first step, we will directly give both
processes, $Y^n$ and $Y^n_b$ identical arrival processes, the
Poisson arrival processes specified for $Y^n$. We will then
show that the remaining construction is possible by increasing
(decreasing) the departure rates so that, whenever $Y^n =
Y^n_{b}$, any departure in $Y^n$ also leads to a departure in
$Y^n_{b}$. That is justified by having the conditional
departure rates, given the full histories of the systems up to
time $t$, be ordered.

When $r_{1,2} = 1$, rate order directly implies the stronger sample path stochastic order, but not more generally, because the upper (lower) process
can jump down below (up above) the lower (upper) process when the lower process is at state $0$ or below, while the upper process is just above state $0$.
Nevertheless, we can obtain the following stochastic order bound, involving a finite gap.
For the following we use the rational form $r_{1,2} = j/k$, $j, k \in \ZZ_+$ and the associated integer-valued
QBD, as in \S 6 of \cite{PeW09c}.  (Recall Assumption \ref{assE}.)
There is no gap when $r_{1,2} =1$ because then $j = k = 1$ and the jump Markov process and associated QBD process both are
equivalent to a simple birth-and-death process.

\begin{coro}{$($stochastic bounds from rate order for FTSP's$)$}\label{corRateOrder}
Consider the FTSP in \S {\em \ref{secFTSP}} with QBD representation based on $r_{1,2} = j/k$ as in \S $6$ of {\em \cite{PeW09c}}, for candidate states
$\gamma^{(i)} \equiv (q^{(i)}_1, q^{(i)}_2, z^{(i)}_{1,2})$, $i = 1,2$.
$($a$)$ If $\mu_{1,2} \ge \mu_{2,2}$ and
\beq
(-q^{(1)}_1, q^{(1)}_2, z^{(1)}_{1,2}) \le (-q^{(2)}_1, q^{(2)}_2, z^{(2)}_{1,2}) \qinq \RR^3,
\eeqno
then
\beq
D(\gamma^{(1)}, \cdot) \le_{st} D(\gamma^{(2)}, \cdot) + (j \vee k) - 1.
\eeqno
$($b$)$ If $\mu_{1,2} \le \mu_{2,2}$ and
\beq
(-q^{(1)}_1, q^{(1)}_2, -z^{(1)}_{1,2}) \le (-q^{(2)}_1, q^{(2)}_2, -z^{(2)}_{1,2}) \qinq \RR^3,
\eeqno
then
\beq
D(\gamma^{(1)}, \cdot) \le_{st} D(\gamma^{(2)}, \cdot) + (j \vee k) - 1.
\eeqno
\end{coro}

\begin{proof}
We can do the standard sample path construction:  Provided that the processes are on the same side of state 0 in the QBD representation,
we can make all the processes jump up by the same amount whenever the lower one jumps up,
and make all the processes jump down by the same amount whenever the upper one jumps down.
However, there is a difficulty
when the processes are near the state 0 in the QBD representation
(which involves the matrix $B$ for the QBD).
When the upper process is above $0$ and the lower process is at or below $0$,
the lower process can jump over the upper process by at most $(j \vee k) - 1$,
and the upper process can jump below the lower process by this same amount.  But the total discrepancy cannot exceed $(j \vee k) - 1$, because of the rate order.
Whenever the desired order is switched, no further discrepancies can be introduced.
\end{proof}

The complexity of the proof in Appendix \ref{secAltProof} is primarily due to the fact that we allow general rational ratio parameters.
If $r_{1,2} = r_{2,1} = 1$, the proof can be much shorter, directly exploiting the sample path stochastic order in
Corollary \ref{corRateOrder} above (where there is no gap).

\begin{remark}{$($rate order comparisons for queue difference processes$)$}\label{rmRateOrder}
{\em
In the rest of this paper, in particular in the proof of Lemma \ref{lmD21extreme} in \S \ref{secProofThZ21},
Theorem \ref{th1} in \S \ref{secAltProof}
and Theorem \ref{thDSB} in \S \ref{secMainProofs2},
we will combine the results in this subsection and earlier subsections to
establish rate order and sample path stochastic order comparisons between queue difference processes and associated frozen difference processes.
A typical initial rate order statement will be of the form
\bequ \label{rateOrderForm}
D^n_{1,2}  \le_r  D^n_f (\Gamma^n) \qinq \D([0,\delta])
\eeq
for some $\delta > 0$, which we now explain.  First, $D^n_{1,2}$ is a function of the Markov process $X^n_6$, which has state-dependent rates.
Thus the ``transition rates'' of $D^n_{1,2}$ are understood to be functions of time $t$ and $X^n_6 (t)$, the state of the Markov process $X^n_6$ at time $t$,
which includes the value of $D^n_{1,2} (t)$.  However, the right side of \eqn{rateOrderForm} is interpreted quite differently.  We regard
$D^n_f (\Gamma^n)$, conditional on the random state vector $\Gamma^n$, as a homogenous pure-jump Markov process constructed independently of $X^n_6$,
with new rate-$1$ Poisson processes, as in \S \ref{secRepGen}.
However, we deliberately construct the random fixed state vector $\Gamma^n$ as a function of $X^n_6$ in order to facilitate comparison of rates.
Thus, $D^n_f (\Gamma^n)$ and $X^n_6$, and thus $D^n_{1,2}$, are dependent, but they are conditionally independent given the random state vector $\Gamma^n$.  Thus,
the conclusion in
\eqn{rateOrderForm} means that the transition rates
at each time $t$ for each value of $X^n_6 (t)$ are ordered.
If the two processes are in the same state at some time $t$, then the two processes can make transitions to the same states,
and each upward transition rate of $D^n_f (\Gamma^n)$ is greater than or equal to the corresponding upward transition rate of $D^n_{1,2}$,
while each downward transition rate of $D^n_f (\Gamma^n)$ is less than or equal to the corresponding downward transition rate of $D^n_{1,2}$.
That rate ordering then allows the sample path stochastic order comparisons, as in Corollary \ref{corRateOrder}.
}
\end{remark}

\subsection{A Sample Path Stochastic Order Bound.}\label{secAbanBd}

For the proof of Theorem \ref{thStatLim} in \S \ref{secAuxProofs},
we also need an upper bound process, unlike $Q^n_{i,b} (t)$ in \eqn{Qbounds}, that does not explode as $t \ra \infty$.
Hence, we now
establish an elementary sample path stochastic order bound
on the queue lengths that is stronger than the w.p.1 upper bound in Lemma \ref{lmQbds}.  Each of the two upper bound stochastic processes has the structure of the
queue length in an $M/M/\infty$ model, with a service rate equal to the abandonment rate here,
for which asymptotic results have been established \cite{PTW07}.

\begin{lemma}{$($sample path stochastic order for the queue lengths$)$}\label{lmQbds2}
For $i = 1, 2$ and $n \ge 1$, let
\beq
Q^n_{i,bd}(t)  =  Q^n_{i}(0) + N^a_i(\lm^n_i t) - N^u_i (\theta_i \int_{0}^{t} Q^n_{i,bd} (s) \, ds), \quad t \ge 0.
\eeqno
Then
\beq
Q^n_{i,bd}(t) \Rightarrow Q^n_{i,bd}(\infty) \qasq t \ra \infty,
\eeqno
where $Q^n_{i,bd}(\infty)$ has a Poisson distribution with mean $\lambda^n_i/\theta_i$ and
\beq
\bar{Q}^n_{i,bd} \equiv n^{-1} Q^n_{i,bd} \Rightarrow q_{i,bd} \qinq \D ([0, \infty)) \qasq n \ra \infty,
\eeqno
where $q_{i,bd}$ evolves deterministically according to the ODE $\dot{q}_{i,bd} (t) = \lambda_i - \theta_i q_{i,bd} (t)$,
starting at $q_{i,bd} (0) \equiv q_i (0)$ for $q_i (0)$ part of $x(0)$ in Assumption {\em \ref{assC}}.  Thus,
\beq
q_{i,bd} (t) \le q_{i,bd}(\infty) \equiv q_i (0) \vee (\lambda_i/\theta_i).
\eeqno
Moreover,
\beq
(Q^n_{1}, Q^n_2) \le_{st} (Q^n_{1,bd},Q^n_{2,bd}) \qinq \D_2 ([0, \infty)).
\eeqno
\end{lemma}

\begin{proof}  We apply Assumption \ref{assC} to get the intial queue lengths to converge.  Just as for $Q^n_{i,b}$ in Lemma \ref{lmQbds},
the upper bound system here provides no service completion at all.
However, unlike the upper bound $Q^n_{i,b}$ in Lemma \ref{lmQbds}, here abandonment from queue is allowed.  Here we have sample path stochastic order
because we can construct the two systems together, keeping the upper bound system greater than or equal to $Q^n_i (t)$ for all $t$.
Whenever the constructed processes are equal, they can have the same abandonments, because the abandonment rate in both systems will
be identical.
\end{proof}

\subsection{Extreme-Value Limits for QBD Processes.}\label{secQBDextreme}

In order to prove Theorem \ref{thZ21} in \S \ref{secProofThZ21}, we exploit extreme-value limits for QBD processes.
Since we are unaware of any established extreme-value limits for QBD processes, we establish the following result here.
Recall that a QBD has states $(i,j)$, where $i$ is the level and $j$ is the phase.  If we only consider the level
we get the level process; it is an elementary function of a QBD.

\begin{lemma}{$($extreme value for QBD$)$} \label{lmQBDextreme}
If\ $\mathcal{L}$ is the level process of a positive recurrent $($homogeneous$)$ QBD process $($with a finite number of phases$)$, then there
exists $c > 0$ such that
\bes
\lim_{t \tinf} P \left( \| \mathcal{L} \|_t / \log t > c \right) = 0.
\ees
\end{lemma}

\begin{proof}
Our proof is based on regenerative structure.  The intervals between successive visits
to the state $(0,j)$ constitute an embedded renewal process for the QBD.  Since the QBD is positive recurrent,
these cycles have finite mean.
Given the regenerative structure, our proof is based on the observation that,
if the process $\mathcal{L}$ were continuous real-valued with an exponential tail, instead of integer valued with a geometric tail,
then we could establish the conventional convergence in law of $\| \mathcal{L} \|_t - c  \log{t}$ to the Gumbel distribution, which implies
our conclusion.  Hence, we bound the process $\mathcal{L}$ above w.p.1 by another process $\mathcal{L}_b$ that is
continuous real-valued with an exponential tail and which inherits the regenerative structure of $\mathcal{L}$.

We first construct the bounding process $\mathcal{L}_b$ and then afterwards explain the rest of the reasoning.
To start, choose a phase determining a
specific regenerative structure for the level process $\mathcal{L}$.  let $S_i$ be the epoch cycle $i$ ends, $i \ge -1$,
with $S_{-1} \equiv 0$,
and let $L(n)$ be the set of states in level $n$.
For each cycle $i$,
we generate an independent exponential random variable $X_i$ and take the maximum between $\mathcal{L}(t)$ and $X_i$ for all
$S_{i-1} \le t < S_i$ such that $\mathcal{L}(t) \notin L(0)$; i.e., letting $\{X_i: i \ge 0\}$ be an i.i.d.
sequence of exponential random variables independent of $\mathcal{L}$ and letting $C(t)$ be the cycle in progress at time $t$,
$\mathcal{L}_b (t) \equiv \mathcal{L}(t) \vee X_{C(t)} 1_{\{\mathcal{L}(t) \notin L(0)\}}$.
Clearly, $\mathcal{L}_b$ inherits the regenerative structure of $\mathcal{L}$ and satisfies
$\mathcal{L} \le \mathcal{L}_b$ almost surely.
Moreover, by the assumed independence, for each $x > 0$ and $t \ge 0$,
\bes
P(\mathcal{L}_b (t) > x) = P(\mathcal{L}(t) > x) + P(X > x) - P(\mathcal{L}(t) > x)P(X > x),
\ees
where $X$ is an exponential random variable distributed as $X_i$ that is independent of $\mathcal{L}(t)$.
We now consider the stationary version of $\mathcal{L}$, which makes $\mathcal{L}_b$ stationary as well.
We let the desired constant $c$ be the mean of the exponential random variables $X_i$.
If we make $c$ sufficiently large, then we clearly have $P(\mathcal{L}_b (t) > x) \sim e^{-x/c}$ as $x \ra \infty$,
because the first and third terms become asymptotically negligible as $x \ra \infty$.  (We choose $c$ to make
$\mathcal{L} (t)$ asymptotically negligible compared to $X$.)

It now remains to establish the conventional extreme-value limit for the bounding process $\mathcal{L}_b$.
For that, we exploit the exponential tail of the stationary distribution, just established, and regenerative structure.
There are two approaches to extreme-value limits for regenerative processes, which are intimately related,
as shown by Rootz\'{e}n \cite{Ro88}.  One is based on stationary processes, while the other is based on the
cycle maxima, i.e., the maximum values achieved in successive regenerative cycles.
First, if we consider the stationary version, then we can apply classical extreme-value limits for stationary processes
as in \cite{LLR83}.  The regenerative structure implies that the mixing condition in \cite{LLR83}
is satisfied; see Section 4 of \cite{Ro88}.

However, the classical theory in \cite{LLR83} and the analysis in \cite{Ro88}
applies to sequences of random variables as opposed to
continuous-time processes.
In general, the established results for stationary sequences in \cite{LLR83} do not
extend to stationary continuous-time processes.
That is demonstrated by extreme-value limits for positive recurrent diffusion processes in
\cite{BK98, D82}.  Proposition 3.1, Corollary 3.2 and Theorem 3.7 of \cite{BK98}
show that, in general, the extreme-value limit is not determined by the stationary distribution of the process.

However, continuous time presents no difficulty in our setting,
because the QBD is constant between successive
transitions, and the transitions occur in an asymptotically regular way.
It suffices to look at the embedded discrete-time process at transition epochs.
That is a standard discrete-time Markov chain associated with the continuous-time Markov chain
represented as a QBD.
Let $N(t)$ denote the number of transitions over the interval $[0,t]$.
Then $\mathcal{L}_b (t) = \mathcal{L}_{d} (N(t))$, where $\mathcal{L}_d (n)$ is the embedded discrete-time process associated
with $\mathcal{L}_b$.  Since $N(t)/t \ra c' > 0$ w.p.1 as $t \ra \infty$ for some constant $c' > 0$,
the results directly established for the discrete-time process $D_d$ are inherited with minor modification by $\mathcal{L}_b$.
Indeed, the maximum over random indices already arises when relating extremes for regenerative sequences to extremes of i.i.d. sequences;
see p. 372 and Theorem 3.1 of \cite{Ro88}.  In fact, there is a substantial literature on extremes with a random index, e.g., see
Proposition 4.20 and (4.53) of \cite{R87} and also \cite{ST98}.
Hence, for the QBD we can initially work in discrete time, to be consistent with \cite{LLR83, Ro88}.
After doing so, we obtain extreme-value limits in both discrete and continuous time,
which are essentially equivalent.

So far, we have established an extreme-value limit for the stationary version of $\mathcal{L}_b$,
but our process $\mathcal{L}_b$ is actually not a stationary process.  So it is natural to apply the second approach
based on cycle maxima, which is given in \cite{Ro88,A98} and Section VI.4 of \cite{Asmussen}.
We would get the same extreme-value limit for the given version of $\mathcal{L}_b$ as the stationary version
if the cycle maximum has an exponential tail.  Moreover, this reasoning would apply directly to continuous time as well as discrete time.
However, Rootz\'{e}n \cite{Ro88} has connected the
two approaches (see p. 380 of \cite{Ro88}), showing that all the versions of the regenerative process have the same
extreme-value limit.  Hence, the given version of the process $\mathcal{L}_b$ has the same extreme-value limit
as the stationary version, already discussed.
Moreover, as a consequence, the cycle maximum has an exponential tail if and only if the
stationary distribution has an exponential tail.  Hence, we do not need to consider the cycle maximum directly.
\end{proof}

\section{Proofs of Three Theorems from \S \ref{secSSCserv}.}\label{secProofsSSCserv}

In this section we prove the three theorems in \S \ref{secSSCserv}.

\subsection{Proof of Theorem \ref{thZ21}.}\label{secProofThZ21}

 Theorem \ref{thZ21} is an immediate consequence of the following three lemmas: Lemmas \ref{lmZ12}, \ref{lmD21} and \ref{lmD21extreme}.

\begin{lemma} \label{lmZ12}
If $z_{1,2}(0) > 0$, then, for all $T > 0$, $P(\inf_{0 \le t \le T} \barz^n_{1,2}(t) > 0) \ra 1$ as $n \tinf$.
As a consequence, $Z^n_{2,1} \Ra 0$ as $n \tinf$.
\end{lemma}

\begin{proof}
By Assumption \ref{assC}, $\barz^n_{1,2} (0) \Rightarrow z_{1,2}(0)$.
From Lemma \ref{lmZbds} we know that $Z^n_{1,2} \ge
Z^n_a$ in $\D$ for all $n \ge 1$ w.p.1. From Lemma \ref{lmBoundsLim},
we know that $\barz^n_a \Ra z_a$ in $\D$ as $n\tinf$, for $z_a$
in \eqref{zbounds}. However, the integral equation for $z_a$ in
\eqref{zbounds} is equivalent to the ODE $\dot{z}_a (t) = -
\mu_{1,2} z_a (t)$ with initial value $z_a (0) = z_{1,2} (0)$.
Since $z_{1,2} (0) > 0$ by assumption, it follows that
${z}_{a}(t) \ge {z}_{a}(0) e^{-\mu_{1,2} t} > 0$ for all $t \ge
0$. Thus $P(\inf_{0 \le s \le t}{Z}^n_{a}(s) > 0) \ra 1$ as $n
\tinf$. Lemma \ref{lmZbds} implies that the same is true for
$Z^n_{1,2}$, which proves the first claim of the lemma. The
second claim that $Z^n_{2,1} \Ra 0$ as $n \tinf$ follows from
the first together with the one-way sharing rule.
\end{proof}

\begin{remark} \label{remShare}{$($implications for one-way sharing rule$)$}
{\em
The conclusion of Lemma \ref{lmZ12} reveals a disadvantage of the one-way sharing rule for very large systems.
The lemma concludes that, for large $n$, if for some $\epsilon > 0$ and $t_0 \ge 0$
 $Z^n_{1,2}(t_0) > \epsilon n$,
then $Z^n_{1,2}(t)$ is very likely not to reach $0$ for a long time, thus
preventing sharing in the opposite direction, even if that would prove beneficial to do so at a later time, e.g.,
because there is a new overload incident in the opposite direction.

In practice, we thus may want to relax the one-way sharing rule.
One way of relaxing the one-way sharing rule is by
dropping it entirely, and relying only on the thresholds $k^n_{1,2}$ and $k^n_{2,1}$
to prevent sharing in both directions simultaneously (at least until the arrival rates change again).
Another modification is to introduce lower thresholds on the service processes,
denoted by $s^n_{i,j}$, $i \ne j$, such that pool $2$ is allowed to start
helping class $1$ at time $t$ if $D^n_{2,1} > k^n_{2,1}$ and $Z^n_{1,2}(t) < s^n_{1,2}$,
and similarly in the other direction.
We do not analyze either of these modified controls in this paper.
}
\end{remark}

Given Lemma \ref{lmZ12}, it remains to consider only the case $z_{1,2}(0) = 0$.
Hence, we assume that $z_{1,2}(0) = 0$ for the rest of this section. Here is the outline of the proof:
We first prove (Lemmas \ref{lmD21} and \ref{lmD21extreme} below) that $Z^n_{2,1}$ is asymptotically null over an interval $[0, \tau]$,
for some $\tau > 0$. We then prove that $\bar{Z}^n_{1,2}(t)$ must become strictly positive before time $\tau$ in fluid scale.
By Assumption \ref{assRatio}, the optimal ratios for FQR-T satisfy $r_{1,2} \ge r_{2,1}$.
In Lemma \ref{lmD21} we consider the cases (i) $x(0) \in \AA \cup \AA^+$ with $r_{1,2} > r_{2,1}$ and $q_1(0) > 0$ and $(ii)$ $x(0) \in \SS^+$;
in Lemma \ref{lmD21extreme} we consider the remaining cases, i.e., $x(0) \in \AA \cup \AA^+$ with
$r_{1,2} = r_{2,1}$ or $q_1(0) = 0$.  Unlike the definition of $D^n_{1,2}$ in \eqn{Dprocess}, let $D^n_{2,1}$ be defined by
\bequ \label{Dprocess2}
D^n_{2,1} (t) \equiv r_{2,1} Q^n_2 (t) -  Q^n_1 (t), \quad t \ge 0.
\eeq

\begin{lemma} \label{lmD21}
Assume that $z_{1,2}(0) = 0$. If either one of the following two conditions hold:
$(i)$ $x(0) \in \AA \cup \AA^+$, $r_{1,2} > r_{2,1}$ and $q_1(0) > 0$, or
$(ii)$ $x(0) \in \SS^+$,
then there exists $\tau$, $0 < \tau \le \infty$, such that
\bes
\lim_{n \tinf} P \left(\sup_{t \in [0, \tau]}D^n_{2,1}(t) \le 0 \right) = 1
\ees
for $D^n_{2,1}$ in {\em \eqn{Dprocess2}},
so that $\|Z^n_{2,1}\|_{\tau} \Ra 0$ as $n \tinf$.
\end{lemma}

\begin{proof} We first show that the appropriate conditions hold in fluid scale at the origin.
We start by assuming that $x(0) \in \AA \cup \AA^+$, which implies that $d_{1,2} (0) \equiv q_1 (0) - r_{1,2} q_2 (0) = 0$.
Since $q_1 (0) > 0$ by assumption, $q_2 (0) > 0$ too.
Since $r_{1,2} > r_{2,1}$ by assumption, $d_{2,1} (0) = r_{2,1} q_2 (0) - q_1 (0) < r_{1,2} q_2 (0) - q_1 (0) = 0$, so that
we also have $d_{2,1} (0) < 0$.
If $(ii)$ holds, so that $x(0) \in \SS^+$, then $d_{2,1}(0) < 0$ by definition of $\SS^+$

Given Assumption \ref{assC}, we also have $\barx^n (0) \Rightarrow x(0)$ in $\AA \cup \AA^+ \cup \SS^+$.
Hence, the fluid-scaled queueing processes converge to these initial values.
In particular, we necessarily have $D^n_{2,1} (0)/n \Rightarrow d_{2,1} (0) < 0$ as $n \ra \infty$.
Hence, there exists $c > 0$ such that $P(D^n_{2,1} (0) < -cn) \ra 1$ as $n \ra \infty$.
Our goal now is to show that there exists $\tau > 0$ such that
$P(\sup_{0 \le t \le \tau} D^n_{2,1} (t) > 0) \ra 0$.  That will imply the desired conclusion.

It only remains to show that the change in these quantities has to be continuous in fluid scale.
For the purpose of bounding $D^n_{2,1} = r_{2,1} Q^n_2 - Q^n_1$ above, it suffices to bound
$Q^n_2$ above and $Q^n_1$ below, as we have done in Corollary \ref{corRateOrder}.  Hence
$D^n_{2,1} \le _{st} D^n_{u, 2,1} \equiv r_{2,1} Q^n_{2,b} - Q^n_{1,a}$.
By Lemma \ref{lmBoundsLim}, $\bar{D}^n_{u, 2,1} \Rightarrow d^u_{2,1}  \equiv r_{2,1} q_{2,b}  - q_{1,a}$,
where the limit function
 $d^u_{2,1}$ evolves continuously, starting with $d^u_{2,1} (0) = d_{2,1} (0) < 0$.  Hence there is a time $\tau'$ such that $d^u_{2,1} (t) < 0$ for all $0 \le t \le \tau'$.
Asymptotically, by the FWLLN, the same will be true for the fluid scaled queue difference process
$\bar{D}^n_{u,2,1}$.  Hence,
we deduce that
$P(\sup_{0 \le t \le \tau}{\{D^n_{u,2,1} (t)\}} > 0) \ra 0$ as $n \ra \infty$ for any $\tau$ with $0 < \tau < \tau'$.
Since $D^n_{2,1}  \le_{st} D^n_{u,2,1}$, the same is true for $D^n_{2,1}$.
That completes the proof.
\end{proof}

The proof of Lemma \ref{lmD21} relies on a fluid argument since, under its assumptions and Assumption \ref{assC}, $D^n_{2,1}(0)/n$ converges
to a strictly negative number as $n \tinf$. In particular, the difference $D^n_{2,1}(0)$ without centering by $k^n_{2,1}$
is order $O_P(n)$ away from the threshold $k^n_{2,1}$.
That fluid reasoning fails when $r_{2,1} = r_{1,2} \equiv r$ or when $q_1(0) = 0$ since in either of these cases,
$q_1(0) - r_{1,2} q_2(0) = q_1(0) - r_{2,1} q_2(0) = 0$. (By Assumption \ref{assC}, $q_2(0) = 0$ if $q_1(0) = 0$.)
In these cases we will rely on the threshold $k^n_{2,1}$ to prevent class-$2$ customers to be sent to pool $1$,
and construct a finer sample-path stochastic-order bound for the stochastic system.
Below we remove the centering by the threshold in $D^n_{2,1}$.

\begin{lemma} \label{lmD21extreme}
Assume that $z_{1,2}(0) = 0$ and that $q_{1}(0) - r_{2,1} q_2(0) = 0$ $($necessarily, $x(0) \notin \SS^+$, i.e., $x(0) \in \AA \cup \AA^+)$.
In this case, there exists $\tau$, $0 < \tau \le \infty$, such that
\bes
\lim_{n \tinf} P \left(\sup_{t \in [0, \tau]} D^n_{2,1}(t) \le k^n_{2,1} \right) = 1,
\ees
for $D^n_{2,1}$ in {\em \eqn{Dprocess2}}.  Hence, $\|Z^n_{2,1}\|_{\tau} \Ra 0$ as $n \tinf$.
\end{lemma}

\begin{proof} We start by showing that $P(D^n_{2,1} (0) \le 0) \ra 1$ as $n \ra \infty$.
That follows because, by Assumption \ref{assRatio} and the definitions \eqn{Dprocess} and \eqn{Dprocess2},
\beq
D^n_{2,1} (0) \equiv r_{2,1} Q^n_2 (0) -  Q^n_1 (0) \le r_{1,2} Q^n_2 (0) -  Q^n_1 (0) = - D^n_{1,2} (0) - k^n_{1,2}.
\eeqno
Assumptions \ref{assThresh} and \ref{assC} then imply that $D^n_{1,2} (0) \Ra L$ and $k^n_{1,2} \ra \infty$ as $n \ra \infty$,
which together imply the initial conclusion.
Going forward, it suffices to assume that we initialize by $D^n_{2,1} (0) = 0$.

The rest of
our proof follows three steps:
In the first step, paralleling Lemma \ref{lmRateOrder} and Corollary \ref{corRateOrder}, we construct a QBD that
bounds $D^n_{2,1}$ (without centering by $k^n_{2,1}$) in rate order, which enables us to obtain a stochastic order bound for $D^n_{2,1}$;
see Remark \ref{rmRateOrder}.
The bound is constructed over an interval $[0, \tau]$.  In the second step, we show how to choose $\tau$ small enough so that
 the QBD bound is asymptotically positive recurrent.
In the third step, we ``translate'' the QBD bound to a time-accelerated QBD, in the spirit of \eqref{ident}, and employ the
extreme-value result for the time-accelerated QBD in Lemma \ref{lmQBDextreme} to conclude the proof.

{\em Step One:} We construct a stochastic-order bound for $D^n_{2,1}$, building on rate order.
For $\tau > 0$, let
\bequ \label{X*}
 X^n_*(t) \equiv (Q^n_{1,b}(t), Q^n_{2,a}(t), Z^n_b(t))
\qandq \Gamma^n_\tau \equiv X^n_\tau \equiv (Q^n_{1,\tau}, Q^n_{2,\tau}, Z^n_\tau),
\eeq
 where
\bes 
Q^n_{1,\tau} \equiv \|Q^n_{1,b}\|_{\tau}, \quad Q^{n}_{2,\tau} \equiv \inf_{0 \le t \le \tau} Q^n_{2,a}(t) \vee 0
\qandq Z^n_\tau \equiv \|Z^n_b\|_{\tau}.
\ees
using the processes defined in \eqn{Zbounds} and \eqn{Qbounds}.
By Lemmas \ref{lmZbds} and \ref{lmQbds}, for all $n \ge 1$,
\bequ \label{stateIneq}
(Q^n_1(t), -Q^n_2(t), Z^n_{1,2}(t)) \le (Q^n_{1,\tau}, -Q^n_{2,\tau}, Z^n_{\tau}) \mbox{ in $\RR_3$ for all } t \in [0, \tau] \quad \mbox{w.p.1}.
\eeq
As in \S \ref{secAux}, let $D^{n}_f (\Gamma^n_\tau)$
be the frozen difference process associated with $D^n_{2,1}$ and $\Gamma^n_\tau$ in \eqn{X*}.
(Recall that we are considering $D^n_{2,1}$ here and not $D^n_{1,2}$.)
However, to obtain positive results, we want to consider the process $D^{n}_f (\Gamma^n_\tau)$
only for nonnegative values.  We obtain such a process
by working with the associated reflected process, denoted by
\bes
D^{n,*}_f \equiv D^{n,*}_f (\Gamma) \equiv \{D^{n,*}_f (\Gamma, t) : t \ge 0\}, \quad \Gamma \in \RR_3
\ees
obtained by imposing a reflecting lower barrier at $0$, where $\Gamma$ specifies the fixed rates of $D^{n,*}_f (\Gamma)$.
We omit $\Gamma$ from the notation for statements that hold for all $\Gamma \in \RR_3$.
The reflected process $D^{n,*}_f$ is always nonnegative and has the same state space
as the nonnegative part of the state space of $D^n_{2,1}$.
Within the QBD framework used in \S 6 of \cite{PeW09c},
we obtain the reflected process by omitting all transitions down below level $0$;
the specific QBD construction is given in Appendix \ref{appQBDbdd}.

It follows from \eqref{stateIneq} and Theorem \ref{lmRateOrder} that we have rate order.
By the analog of Lemma \ref{corRateOrder}, there exists a constant $K$
such that we have sample path stochastic order in $\D([0, \tau])$, i.e.,
\bequ \label{RateBd102}
D^n_{2,1} \le_{st} D^{n,*}_f (\Gamma^n_\tau) + K \quad \mbox{in $\D([0, \tau])$ for all $n \ge 1$},
\eeq
for $\Gamma^n_\tau$ in \eqn{X*}. 

{\em Step Two:}
We now show that we can choose $\tau > 0$ so that there exist sets $E^n$ in
the underlying probability space such that
$D^{n,*}_f (\Gamma^n_{\tau})$ is positive recurrent in $E^n$ and $P(E^n) \ra 1$ as $n \ra \infty$.
Paralleling \eqn{posrecDn} (for the frozen process associated with $D^n_{1,2}$),
the set $E^n$ here is
\bequ \label{Edef}
E^n \equiv \{\delta_{*} (\Gamma^n_{\tau}) < 0\},
\eeq
where $\delta_{*}$ is the drift for $D^{n,*}_f$.
To find $\tau > 0$ such that $P(E^n) \ra 0$ for $E^n$ in \eqn{Edef},
 we analyze the asymptotic behavior of $\Gamma^n_{\tau} \equiv X^n_\tau$ in \eqref{X*}.

First, by Lemma \ref{lmBoundsLim}, $\barx^n_* \Ra x_* \equiv (q_{1,b}, q_{2,a}, z_b)$ in $\D_3$ as $n\tinf$,
where the components of $x_*$ are given in \eqref{qzbounds} and  $x_*(0) = x(0)$ by construction.
Then, by the continuous mapping theorem for the supremum function, e.g., Theorem 12.11.7 in \cite{W02},
\bequ \label{RateBd103}
\barx^n_\tau \equiv X^n_\tau/n \Ra x_\tau \equiv (q_{1,\tau}, q_{2,\tau}, z_\tau) \quad \mbox{in $\RR_3$, as $n \tinf$},
\eeq
where 
\bes 
q_{1,\tau} \equiv \|q_{1,b}\|_{\tau},\quad q_{2,\tau} \equiv \inf_{0 \le t \le \tau} q_{2,a}(t) \vee 0 \qandq z_\tau \equiv \|z_b\|_{\tau}.
\ees
Since $x(0) \in \AA \cup \AA^+$ by the assumption of the lemma,
$\delta_{-} (x(0)) > 0$, where $\delta_{-}$ is the drift for the FTSP in \eqn{drifts}
associated with $D^n_{1,2}$.

For $\gamma \equiv \gamma(t) \equiv (q_{1,b}(t), q_{2,a}(t) \vee 0, z_b(t))$ let $\hatlm_{r}(\gamma) \equiv \lm_2$,
$\hatmu_{r}(\gamma) \equiv \mu_{2,2} (m_2 - z_b(t)) + \theta_2 (q_{2,a}(t) \vee 0)$,
$\hatlm_{1}(\gamma) \equiv \mu_{1,1} m_1 + \mu_{1,2} z_b(t) + \theta_1 q_{1,b}(t)$ and $\hatmu_{1}(\gamma) \equiv \lm_1$.
Let $D^* (\gamma)$ be the reflected FTSP corresponding to $D^{n,*}_f (\Gamma^n_\tau)$.
The process $D^*(\gamma) \equiv \{D^*(\gamma,t) : t \ge 0\}$ has upward jumps of size $r_{2,1}$ with rate
$\hatlm_{r}(\gamma)$ and downward jumps of size $r$ with rate $\hatmu_{r}(\gamma)$.
It has upward jumps of size $1$ with rate $\hatlm_{1}(\gamma)$,
and downwards jumps of size $1$ with rate $\hatmu_{1}(\gamma)$.
By Theorem 7.2.3 in \cite{LR99}, $D^*(\gamma)$ is positive recurrent if and only if
$\delta_*(\gamma) < 0$, where
\bequ \label{drift*}
\delta_*(\gamma) \equiv r_{2,1}(\hatlm_{r}(\gamma) - \hatmu_{r}(\gamma)) + (\hatlm_{1}(\gamma) - \hatmu_{1}(\gamma)), \quad \gamma \in \RR_3.
\eeq
Replacing $\gamma$ with $\gamma_\tau \equiv x_\tau$, we have that $D^*(\gamma_\tau)$ is positive recurrent if and only if
$\delta_*(\gamma_\tau) < 0$.
Hence, it suffices to show that $\delta_*(x_\tau) < 0$ for some $\tau > 0$.  We do that next.

We consider the two possible cases of the condition imposed in the lemma:
(i) $r_{2,1} = r_{1,2}$ and (ii) $q_1(0) = q_2(0) = 0$ (with $z_{b}(0) = z_{1,2}(0) = 0$ in both cases).
First, in case (i),
$\delta_*(x_*(0)) = - \delta_-(x(0))$ for $\delta_-(\gamma)$ in \eqref{drifts}, $\delta_*(\gamma)$ in \eqn{drift*} and $\gamma = x(0)$.
Since we have already observed that $\delta_-(x(0)) > 0$, we necessarily have $\delta_*(x_*(0)) < 0$.

In case (ii) with $r_{1,2} > r_{2,1}$ and $q_1(0) = q_2(0) = 0$ (and again $z_{1,2}(0) = 0$),
$\delta_*(x_*(0)) = r_{2,1} (\lm_1 - \mu_{1,1} m_1) +  (\lm_2 - \mu_{2,2} m_2)$,
so that $\delta_*(x_*(0)) < 0$ if and only if $ (\lm_1 - \mu_{1,1} m_1) + r_{2,1} (\lm_2 - \mu_{2,2} m_2) < 0$.
However, this inequality must hold because, by Assumption \ref{assC}, $\delta_-(x(0)) > 0$, so that
$\delta_-(x(0)) = (\lm_1 - \mu_{1,1} m_1) - r_{1,2} (\lm_2 - \mu_{2,2} m_2) > 0$.
Since $r_{2,1} < r_{1,2}$, it follows that here too $\delta_*(x_*(0)) < 0$.

Finally, the continuity of $x_*$ and $\delta_*(x_*)$ imply that we can find $\tau > 0$ and $\eta > 0$ such that
$\sup_{s \in [0, \tau]} \delta_*(x_*(s)) < -\eta < 0$. In particular, for that choice of $\tau$, $\delta_*(x_\tau) < -\eta$.
Hence, $P(E^n) = P (\delta_*(\Gamma^n_\tau) < 0) \ra 1$,
because $\delta_*(\barx^n_*) \Ra \delta_*(x_*)$ in $\D$ and $\delta_*(\Gamma^n_\tau) \Ra \delta_*(\gamma_\tau)$ in $\RR$ as $n \tinf$
by the continuous mapping theorem.

{\em Step Three:} Finally, we apply the extreme-value result in Lemma \ref{lmQBDextreme}
to the stochastic upper bound $K + D^{n,*}_f(\Gamma^n_\tau)$ in \eqn{RateBd102} for $\Gamma^n_\tau$ in \eqn{X*}.
For that purpose, observe that, paralleling \eqref{ident},
$\{D^{n,*}_f(\Gamma^n_\tau, s) : s \ge 0\} \deq \{D^*(\Gamma^n_\tau/n, ns) : s \ge 0\}$ for $D^{n,*}_f$ and $D^*$ defined above, with
$\Gamma^n_\tau/n \Ra x_\tau$ in \eqn{RateBd103}.
For the rest of the proof, we apply the Skorohod representation theorem to replace the convergence in distribution by convergence
$\Gamma^n_\tau/n \ra x_\tau$ w.p.1, without changing the notation.

By the arguments above
\bequ \label{IneqRate}
\bsplit
P(\|D^n_{2,1}\|_{\tau} / \log n > c) \le
P((K + \|D^{n,*}_f(\Gamma^n_\tau)\|_{\tau}) / \log n > c)
& = P((K + \|D^*(\Gamma^n_\tau/n)\|_{n \tau}) / \log n > c).
\end{split}
\eeq
In order to apply Lemma \ref{lmQBDextreme} to the final term in \eqn{IneqRate}, we want to replace $\Gamma^n_\tau/n \equiv \barx^n_\tau$ by a
vector independent of $n$, say $x_{\epsilon}$, such that $D^* (x_{\epsilon})$ is positive recurrent and,
for some $n_0$, $\barx^n_\tau \le x_{\epsilon}$ for all $n \ge n_0$.
Then we can apply Lemma \ref{lmQBDextreme} to get, for $n \ge n_0$,
\beq
P((K + \|D^*(\Gamma^n_\tau/n)\|_{n \tau}) / \log n > c) \le P ((2K + \|D^*(x_\ep)\|_{n \tau})/ \log n > c) \ra 0 \mbox{\, as } n \tinf.
\eeqno
That implies the claim because of the way the thresholds are scaled in Assumption \ref{assThresh}.

We conclude by showing how to construct the vector $x_{\epsilon}$ such that $D^* (x_{\epsilon})$ is positive recurrent and, for some $n_0$,
$\Gamma^n_\tau/n = \barx^n_{\tau} \le x_{\epsilon}$
for all $n \ge n_0$.
If $q_{2,\tau} > 0$, then choose $\ep$ such that $0 <\ep < q_{2,\tau}$ and let
$x_\ep \equiv (q_{1,\ep}, q_{2,\ep}, z_\ep) \equiv (q_{1,\tau} + \ep, q_{2,\tau} - \ep, z_b + \ep)$.
Otherwise, let $x_\ep \equiv (q_{1,\tau} + \ep, 0, z_b + \ep)$.  Clearly $x_\ep \le x_{\tau}$ for all $ \ep > 0$, so that
$\{D^*(x_\tau, t) : 0 \le t \le \tau\} \le_{st} K + \{D^*(x_\ep, t) : 0 \le t \le \tau\}$ and,
by the choice of $\tau$, we can find $\ep > 0$ small enough so that $\delta^*(x_\ep) < 0$,
so that $D^*(x_\ep)$ is positive recurrent.
\end{proof}

\subsection{Proof Theorem \ref{thQpos}.}\label{secProofThQpos}

Define the processes
\bequ \label{net} L^n_1 \equiv Q^n_1 +
Z^n_{1,1} + Z^n_{1,2} - m^n_1 \qandq L^n_2 \equiv Q^n_2 +
Z^n_{2,1} + Z^n_{2,2} - m^n_2,
\eeq
representing the excess
number in system for each class. Note that $(L^n_i)^+ = Q^n_i$, $i = 1,2$.
For all $n$ sufficiently large, we will bound the two-dimensional process
$(L^n_1, L^n_2)$ below in sample-path stochastic order by another two-dimensional process
$(L^n_{1,b}, L^n_{2,b})$, $n \ge 1$.

We construct the lower-bound process $(L^n_{1,b}, L^n_{2,b})$ by having $L^n_{i,b}(0) = L^n_{i}(0) \wedge k^n_{1,2}$, $i = 1,2$, and
by increasing the departure rates in both processes $L^n_1$ and $L^n_2$, making it so that each goes down at least as fast,
regardless of the state of the pair.
First, we place reflecting upper barriers on the two queues.  This is tantamount to making the death rate infinite in these states and all higher states.
We place the reflecting upper barrier on
$L^n_i$ at $k^n_{1,2}$, where $k^n_{1,2} \ge 0$.
This necessarily produces a lower bound for all $n$.
By the initial conditions assumed for the queue lengths in Assumption \ref{assC},
we have $P((L^n_{1,b} (0), L^n_{2,b} (0)) = (k^n_{1,2},k^n_{1,2})) \ra 1$ as $n \ra \infty$.

With the upper barrier at $k^n_{1,2}$, the departure rate of $L^n_1(t)$ at time $t > 0$ is bounded above by
$\mu_{1,1} m^n_{1}  +  \mu_{1,2} Z^n_{1,2}(t) + \theta_1 k^n_{1,2}$,
based on assuming that pool $1$ is fully busy serving class $1$, 
that $L^n_1(t)$ is at its upper barrier,
and that $Z^n_{1,2}(t)$ agents from pool $2$ are currently busy serving class $1$ in the original system.
With the upper barrier at
$k^n_{1,2}$, the departure rate of $L^n_2(t)$ at time $t$ is bounded above by
$\mu_{2,2} (m^n_{2} - Z^n_{1,2}(t)) + \theta_2 k^n_{1,2}$,
based on assuming that pool $2$ is fully busy
with $Z^n_{1,2} (t)$ agents from pool $2$ currently busy serving class $1$,
and that $L^n_2(t)$ is at its upper barrier $k^n_{1,2}$.
Thus, we give $L^n_{1,b}$ and $L^n_{2,b}$ these bounding rates at all times.

As constructed, the evolution of $(L^n_{1,b}, L^n_{2,b})$ depends on the process $Z^n_{1,2}$
associated with the original system, which poses a problem for further analysis.  However, we can avoid this difficulty
by looking at a special linear combination of the processes.  Specifically, let
\beql{s102}
U^n \equiv \mu_{2,2} (L^n_1 - k^n_{1,2}) + \mu_{1,2} (L^n_2 - k^n_{1,2}) \qandq
U^n_b \equiv \mu_{2,2} (L^n_{1,b} - k^n_{1,2}) + \mu_{1,2} (L^n_{2,b} - k^n_{1,2}).
\eeq

By the established sample-path stochastic order $(L^n_{1,b}, L^n_{2,b}) \le_{st} (L^n_1, L^n_2)$,
the initial conditions specified above
and the monotonicity of the linear map in \eqref{s102},
we get the associated sample-path stochastic order $U^n_b \le_{st} U^n$.
The lower-bound stochastic process
 $U^n_b$
has constant birth rate $\lambda^n_b = \mu_{2,2} \lambda^n_1 + \mu_{1,2} \lambda^n_2$ and
constant death rate
\begin{eqnarray}
\mu^n_b & \equiv &\mu_{2,2}(\mu_{1,1} m^n_{1}  + \mu_{1,2} Z^n_{1,2} (t) + \theta_1 k^n_{1,2})
+ \mu_{1,2}(\mu_{2,2} m^n_{2} - \mu_{2,2} Z^n_{1,2} (t)) + \theta_2 k^n_{1,2}) \nonumber \\
& = & \mu_{2,2}(\mu_{1,1} m^n_{1} + \theta_1 k^n_{1,2}) + \mu_{1,2}(\mu_{2,2} m^n_{2} + \theta_2 k^n_{1,2}). \nonumber
\end{eqnarray}
In particular, unlike the pair of processes $(L^n_{1,b}, L^n_{2,b})$,
the process
 $U^n_b$ is independent of the process $Z^n_{1,2}$.  Consequently,
  $U^n_b$ is a birth and death process on the
set of all integers in $(-\infty, 0]$.
Since $P((L^n_{1,b} (0), L^n_{2,b} (0)) = (k^n_{1,2},k^n_{1,2})) \ra 1$ as $n \ra \infty$,
$P(U^n_{b} (0) = 0) \ra 1$ as $n \ra \infty$.

The drift in $U^n_b$ is
\bequ\label{driftn}
\delta^n_b \equiv \lambda^n_b - \mu^n_b = \mu_{2,2}(\lambda^n_1 - m^n_1 \mu_{1,1} -\theta_1 k^n_{1,2})
+ \mu_{1,2}(\lambda^n_2 - m^n_2 \mu_{2,2} - \theta_2 k^n_{1,2}).
\eeq
Hence, after scaling, we get $\delta^n_b/n \ra \delta$ (recall that $k^n_{1,2}$ is $o(n)$), where
\beql{s4}
\delta_b \equiv \mu_{2,2}(\lambda_1 - m_1 \mu_{1,1})
+ \mu_{1,2}(\lambda_2 - m_2 \mu_{2,2}) > 0,
\eeq
with the inequality following from Assumption \ref{assA}.

Now we observe that $-U^n_b$ is equivalent to the number in system
in a stable $M/M/1$ queueing model with traffic intensity $\rho^n_{*} \ra \rho_{*} < 1$, starting out empty, asymptotically.
Let $Q_*$ be the number-in-system process in an $M/M/1$ system
having arrival rate equal to $\lambda_* \equiv \mu_{2,2} \mu_{1,1}m_1 + \mu_{1,2}\mu_{2,2}m_2$,
service rate $\mu_* \equiv \mu_{2,2}\lambda_1 + \mu_{1,2}\lambda_2 $ and traffic intensity $\rho_* \equiv \lambda_*/\mu_* < 1$.
Observe that the scaling in $U^n_b$
is tantamount to accelerating time by a factor of order $O(n)$
in $Q_*$. That is, $\{-U^n_b (t) : t \ge 0\}$ can be represented as
$\{Q_*(c_n t) : t \ge 0\}$, where $c_n / n \ra 1$ as $n \tinf$.
We can now apply the extreme-value result in Lemma \ref{lmQBDextreme} for the $M/M/1$ queue above
(since the $M/M/1$ birth and death process is trivially a QBD process) to conclude that
$\|Q_*\|_t = O_P(\log(t))$. This implies that $-U^n_b/ \log(n)$ is stochastically bounded.

From the way that the reflecting upper barriers were constructed,
we know at the outset that $L^n_{1,b} (t) \le k^n_{1,2}$ and $L^n_{2,b} (t) \le k^n_{1,2}$.
Hence, we must have both $L^n_{1,b} - k^n_{1,2}$ and $L^n_{2,b} - k^n_{1,2}$ non-positive.
Combining this observation with the result that $(-U^n_b)/\log{n}$ is stochastically bounded, we deduce that both
$(k^n_{1,2} - L^n_{1,b})/\log{n}$ and $(k^n_{1,2}-L^n_{2,b})/\log{n}$ are stochastically bounded, i.e.,
the fluctuations of $L^n_{i,b}$ below $k^n_{1,2}$ are $O_P(\log n)$.
The result follows because $-L^n_i \le_{st} - L^n_{i,b}$, $i = 1,2$, and from the choice
of $k^n_{1,2}$, which satisfies $k^n_{1,2}/\log{n} \ra \infty$ as $n \tinf$ by Assumption \ref{assThresh}.

\subsection{Proof of Theorem \ref{thSSCextend}.}\label{secGlobalProofs}

We now show that the interval $[0, \tau]$ over which the conclusions in \S \ref{secSSCserv} are valid
can be extended from $[0, \tau]$ to $[0, \infty)$ after Theorem \ref{th1} has been proved over the interval $[0, \tau]$.

For this purpose, we use the processes $L^n_1$ and $L^n_2$ defined in \eqref{net}.
By Lemma \ref{lmZ12}, we only need to consider the case $z_{1,2}(0) = 0$.
By Lemmas \ref{lmD21} and \ref{lmD21extreme}, there exists $\tau > 0$ such that
\bes
\lim_{n \tinf} P \left( \|D^n_{2,1}\|_{\tau} < k^n_{2,1} \right) = 1.
\ees
Hence, the claim of the theorem will follow from Lemma \ref{lmZ12}
if we show that for some $t_0$ satisfying $0 < t_0 \le  \tau$ it holds that $z_{1,2}(t_0) > 0$,
 where
$z_{1,2}$ is the (deterministic) fluid limit of $\barz^n_{1,2}$ as $n \tinf$, which exists by Theorem \ref{th1}.
We will actually show a somewhat stronger result, namely that
for any $0 < \ep \le \delta$, where $\delta$ is chosen from Lemma \ref{lmXinA} and thus Theorem \ref{th1},
there exists $t_0 < \ep$ such that $z_{1,2}(t_0) > 0$.
We prove that by assuming the contradictory statement:
for some $0 < \ep \le \delta$ and for all $t \in [0, \ep]$, $z_{1,2}(t) = 0$.

By our contradictory assumption above, $\barz^{n}_{1,2} = o_P(1)$, i.e., $\|Z^{n}_{1,2}/n\|_{\ep} \Ra 0$.
Recall also that $Z^{n}_{2,1} = o_P(1)$ over $[0, \ep]$ (since $\ep \le \tau$, and $\tau$ is chosen according
to Lemmas \ref{lmD21} and \ref{lmD21extreme}).
Hence, by our contradictory assumption and by our choice of $\ep$, there is negligible sharing of customers over the interval $[0, \ep]$.
We can thus represent $L^n_i$ in \eqref{net}, $i = 1,2$ by
\bequ \label{Qbdd}
\bsplit
L^{n}_i(t) = L^{n}_i(0) + N^a_i(\lm^{n}_i t) - N^s_{i,i} \left( \mu_{i,i} \int_0^t (L^{n}_i(s) \wedge 0)\ ds \right)
- N^u_i \left( \theta_i \int_0^t (L^{n}_i(s) \vee 0)\ ds \right) + o_P(n), 
\end{split}
\eeq
for $i = 1,2$ and  $0 \le t \le \delta$ as $n \ra \infty$,
where $N^a_i$, $N^s_{i,i}$ and $N^u_{i}$ are independent rate-$1$ Poisson processes.
The $o_P(n)$ terms are replacing the (random-time changed) Poisson processes related to $Z^{n}_{1,2}$ and $Z^{n}_{2,1}$,
which can be disregarded when we consider the fluid limits of \eqref{Qbdd}.
The negligible sharing translates into the $o_P(n)$ term in \eqn{Qbdd} by virtue of the continuous mapping theorem
and Gronwall's inequality, as in \S 4.1 of \cite{PTW07}.

Letting $\bar{L}^{n}_i \equiv L^{n}_i / n$, $i = 1,2$, and applying the continuous mapping theorem
for the integral representation function in \eqref{Qbdd}, Theorem 4.1 in \cite{PTW07},
(see also Theorem 7.1 and its proof in \cite{PTW07}), we have $(\bar{L}^{n}_1, \bar{L}^{n}_2) \Ra (\bar{L}_1, \bar{L}_2)$ in $\D([0, \ep])$
as $n \tinf$, where, for $i = 1,2$,
\beq
\bsplit
\bar{L}_i(t) & = \bar{L}_i(0) + (\lm_i - \mu_{i,i} m_i)t - \int_0^t [\mu_{i,i}(\bar{L}_i(s) \wedge 0) + \theta_i(\tilde{L}_i(s) \vee 0)]\ ds,
\quad 0 \le t \le \ep,
\end{split}
\eeqno
so that
\beq
\bar{L}'_i(t) \equiv \frac{d}{dt}\bar{L}_i(t) = (\lm_i - \mu_{i,i} m_i) - \mu_{i,i}(\bar{L}_i(t) \wedge 0) - \theta_i(\tilde{L}_i(t) \vee 0),
\quad 0 \le t \le \ep.
\eeqno

By Assumption \ref{assC}, both pools are full at time $0$, so that $L_i(0) \ge 0$.
Moreover, for $i = 1,2$, $\bar{L}_i^e \equiv (\lm_i - \mu_{i,i}) / \theta_i$ is an
equilibrium point of the ODE $\bar{L}'_i$, in the sense that,
if $\bar{L}_i(t_0) = \bar{L}^e_i$, then $\bar{L}_i(t) = \bar{L}^e_i$ for all $t \ge t_0$.
(That is, $\bar{L}^e_i$ is a fixed point of the solution to the ODE.)
It also follows from the derivative of $\bar{L}_i$ that $\bar{L}_i$ is strictly increasing if $\bar{L}_i(0) < \bar{L}^e_i$,
and strictly decreasing if $\bar{L}_i(0) > \bar{L}^e_i$, $i = 1,2$.

Recall that $\rho_1 > 1$, so that $\lm_1 - \mu_{1,1} m_1 > 0$. Together with the initial condition, $L_1(0) \ge 0$,
we see that, in that case, $\bar{L}_1(t) \ge 0$ for all $t \ge 0$.
First assume that $\rho_2 \ge 1$ .
Then, by similar arguments, $\bar{L}_2(t) \ge 0$ for all $t \ge 0$.
In that case, we can replace $\bar{L}_i$ with $q_i(t) = (\bar{L}_i(t))^+$, $i = 1,2$,
 where $q_i$ is the fluid limit of $\barq^{n}_i$ over $[0, \ep]$.
We can then write, for $t \in [0, \ep]$,
\bes
\bsplit
q_1(t) & = q_1(0) - (\lm_1 - \mu_{1,1} m_1) t - \theta_1 \int_0^t q_1(s)\ ds, \\
q_2(t) & = q_2(0) - (\lm_2 - \mu_{2,2} m_2) t - \theta_2 \int_0^t q_2(s)\ ds,
\end{split}
\ees
so that, for $t \in [0, \ep]$,
\bequ \label{BarD12}
\bsplit
d_{1,2}(t) & = q^a_1 + (q_1(0) - q^a_1) e^{- \theta_1 t} - r \left( q^a_2 + (q_2(0) - q^a_2) e^{- \theta_2 t} \right) \\
& = (q^a_1 - r q^a_2) + (q_1(0) - q^a_1) e^{- \theta_1 t} - r (q_2(0) - q^a_2) e^{- \theta_2 t}.
\end{split}
\eeq
First assume that $x(0) \in \AA \cup \AA^+$, so that $d_{1,2}(0) = 0$.  From \eqn{BarD12},
\bes
d'_{1,2}(t) \equiv \frac{d}{dt} d_{1,2}(t) = -\theta_1 (q_1(0) - q^a_1) e^{- \theta_1 t} + r \theta_2 (q_2(0) - q^a_2) e^{- \theta_2 t}.
\ees
Hence, $d'_{1,2}(0) = \lm_1 - \mu_{1,1}m_1 - \theta_1 q_1(0) - r (\lm_2 - \mu_{2,2}) + r \theta_2 q_2(0)$.
If follows from \eqref{drifts} and the assumption $z_{1,2}(0) = 0$, that $d'_{1,2}(0) = \delta_-(x(0))$.
By Assumption, $x(0) \in \AA \cup \AA^+$, so that $d'_{1,2}(0) = \delta_-(x(0)) > 0$ (that follows from the definition of $\AA$ and $\AA^+$
in \eqref{Aset2} and \eqref{AplusSet}, and the fact that $\delta_- > \delta_+$).
Hence, $d_{1,2}$ is strictly increasing at $0$.
Now, since $d_{1,2}(0) = 0$, we can find $t_1 \in (0, \ep]$, such that $d_{1,2}(t) > 0$ for
all $0 < t < t_1$.
This implies that $P (\inf_{0 < t \le t_1}D^{n}_{1,2}(t) > 0) \ra 1$ as $n \tinf$.
The same is true if $x(0) \in \rS^+$.

Since $\|Z^{n}_{1,2}/n\|_{\ep} \Ra 0$, as a consequence of our contradictory assumption,
it follows from the representation of $Z^n_{1,2}$ in Theorem \ref{thRep} that
\bequ \label{z1}
\barz^{n}_{1,2}(t) = n^{-1}N^s_{2,2} \left( \mu_{2,2} m^{n}_2 t \right) + o_P(1) \qasq n \ra \infty.
\eeq
However, by the FSLLN for Poisson processes, the fluid limit $z_{1,2}$ of $\barz^{n}$ in \eqref{z1} satisfies
$z_{1,2}(t) = \mu_{2,2} m_2 t > 0$ for every $0 < t \le t_1$.
We thus get a contradiction to our assumption that $z_{1,2} (t) = 0$ for all $t \in [0, \ep]$.

For the case $\rho_2 < 1$ the argument above still goes through, but we need to distinguish between two cases:
$\bar{L}_2 = 0$ and $\bar{L}_2 > 0$. In both cases $\bar{L}_2$ is strictly decreasing. In the first case, this implies that
$\bar{L}_2$ is negative for every $t > 0$. It follows immediately that $q_1(t) - r q_2(t) > 0$ for every $t > 0$.
If $\bar{L}_2(0) > 0$, then necessarily $\bar{L}_1(0) > 0$, and we can replace $\bar{L}_i$ with $q_i$, $i = 1,2$,
on an initial interval (before $\bar{L}_2$ becomes negative). We then use the arguments used in the case $\rho_2 \ge 1$ above.
\hfill \qed

In the proof of Theorem \ref{thSSCextend} we have shown that, if $z_{1,2}(0) = 0$, then $z_{1,2}(t) > 0$ for all $t > 0$.
We prove in Appendix \ref{appPosQ} that the two queues must also become strictly positive right after time $0$,
if they are not strictly positive at time $0$.

\section{Alternative Proof of Characterization Using Stochastic Order Bounds.}\label{secAltProof}

In this appendix we present an alternative proof of Lemma \ref{lmKey} and thus
an alternative proof of the FWLLN in Theorem \ref{th1}.
At the beginning of \S \ref{secProofs} we observed that it suffices to
characterize the limiting integral terms in \eqn{FS}; i.e., it suffices to prove Lemma \ref{lmKey}.
In \S \ref{secProofs} we accomplished that goal by using the martingale argument
of \cite{HK94,K92}.  Here we show that same goal can be achieved with stochastic bounds,
exploiting Lemma \ref{lmD12PR} and similar reasoning.
However, we prove less than the full Theorem \ref{th1} here.
Our proof here is under the special case of Assumption \ref{assC} for which
$x(0) \in \AA$.  For this special case we carry out the characterization over the full interval $[0, \infty)$ if $x (t)$ remains within $\AA$ for all $t \ge 0$.
Otherwise, we complete the characterization proof over $[0, T_A]$, where
\beql{TA}
T_A \equiv \inf{\{t > 0: x(t) \not\in \AA\}}.
\eeq
Since $x$ is continuous and $\AA$ is an open subset of $\rS$, we know that $T_A > 0$.
A first step is to do the characterization over an interval $[0, \delta]$ for some $\delta > 0$.
We start in \S \ref{secExpand} by indicating how the interval of convergence can be extended given that the first step has been carried out.
Next in \S \ref{secProofAP2} we prove Theorem \ref{th1} subject to Lemma \ref{lmConvInt}, establishing convergence of integral terms over the interval $[0, \delta]$.
For that purpose, we state Lemma \ref{lmNewBds} establishing a sample path stochastic order bound
that we will use to prove Lemma \ref{lmConvInt}.
in \S \ref{secCont} we establish continuity results for QBD processes.
We then prove Lemmas \ref{lmConvInt} and \ref{lmNewBds}, respectively, in
\S \ref{secEP4.5} and \S \ref{secEP4.4}.

\subsection{Extending the Interval of Convergence.}\label{secExpand}

Unlike the first proof, with this second proof we only establish convergence
in $\D_{14}([0,T_A])$ if $T_A < \infty$,
for $T_A$ in \eqn{TA}.
We now show how we achieve this extension.

As in the first proof, after establishing the convergence over an initial interval $[0, \delta]$ with $\delta \le \tau$,
we apply Theorem \ref{thSSCextend} to conclude that
any limit point of the tight sequence $\barx^n_6$ is again a limit of the tight sequence $\barx^{n,*}_6$ in \eqref{red1} over the entire half line $[0,\infty)$,
showing that $\tau$ places no constraint on expanding the convergence interval.
Moreover, by part (ii) of Theorem 5.2 in \cite{PeW09c}, any solution to the ODE, with a specified initial condition, can be extended indefinitely,
and is unique. Hence that places no constraint either.

However, for this second method of proof, we do critically use that fact that $x(t) \in \AA$, $0 \le t \le \delta$, in order to prove
the characterization.  (This is proved in Lemma \ref{lmXinA} below.)  However, we can extend the interval of convergence further.
Given that we have shown that $\barx (t) = x(t) \in \AA$ for $\AA$ in \eqn{Aset2} over a time interval $[0, \delta]$,
and thus established the desired convergence $\barx^n \Ra x$ over that time interval $[0, \delta]$, we can always extend the time interval
to a larger interval $[0,\delta']$ for some $\delta'$ with $\delta' > \delta$.
To do so, we repeat the previous argument treating time $\delta$ as the new time origin.
That directly yields $\barx (t) = x(t) \in \AA$, and thus convergence $\barx^n \Ra x$, over the time interval $[\delta, \delta']$.  However,
we can combine that with the
previous result to obtain $\barx (t) = x(t) \in \AA$, and thus convergence $\barx^n \Ra x$, over the longer time interval $[0, \delta']$.
Let $\nu$ be the supremum of all $\delta$ for which the expansion of convergence to $[0,\delta]$ is valid.
We must have $\barx (t) = x(t) \in \AA$, and thus convergence $\barx^n \Ra x$, over the interval $[0, \nu)$, open on the right.
The interval is $[0, \infty)$ if $\nu = \infty$.  Suppose that $\nu < \infty$.
In that case, we can next apply continuity to extend the interval of convergence to the closed interval $[0,\nu]$.
Since $\barx (t) = x (t)$, $0 \le t < \nu$, $x$ is continuous and $\barx$ is almost surely continuous,
we necessarily have $\barx (\nu) = x (\nu)$ w.p.1 as well.
We claim that $\nu \ge T_A \equiv \inf\{t \ge 0 : x(t) \notin \AA\}$.
If not, we can do a new construction yielding $\barx (t) = x (t)$, first for $\nu \le t < \nu'$ and then for $0 \le t \le \nu'$, $\nu' > \nu$,
contradicting the definition of $\nu$.
Hence, we have extended the domain of convergence to $[0,T_A]$ if $T_A < \infty$ and to $[0, \infty)$ otherwise, as claimed.

\subsection{Reduction to Convergence of Integral Terms.} \label{secProofAP2}

Since each of these integrals in \eqn{FS} can be treated in essentially the
same way, we henceforth focus only on the single term
$\bar{I}_{q,1,1} (t)$ and establish \eqn{key}.

Recall that $\barx$ is the limit of the converging subsequence of $\{\barx^n: n \ge 1\}$
in $\D_3 ([0, \tau])$ for $\tau$ in Theorem \ref{thFluidScaled}.
An important first step is to identify an initial interval $[0, \delta]$, $0 < \delta < \tau,$ over which $\barx \in \AA$.
We apply the proof of Lemma \ref{lmD12PR} to
prove the following lemma.
\begin{lemma}{$($state in $\AA$ over $[0, \delta])$}\label{lmXinA}
There exists $\delta$ with $0 < \delta \le \tau$ for $\tau$ in Theorem {\em \ref{thFluidScaled}} such that
\beq
P(\bar{X}^n (t) \in \AA, \quad 0 \le t \le \delta) \ra 1 \qasq n \ra \infty,
\eeqno
for $\AA$ in {\em \eqn{Aset2}}, so that $P(\barx (t) \in \AA \quad 0 \le t \le \delta) = 1$.
\end{lemma}

\begin{proof}
Recall that we have assumed that Assumption \ref{assC} holds with $x(0) \in \AA$.
We can apply the first step of the proof of Lemma \ref{lmD12PR} to obtain the stochastic bound over
an initial interval $[0, \delta]$ for some $\delta > 0$.
Since the FTSP $D(\gamma, \cdot)$ is positive recurrent
if and only if $\gamma \in \AA$ for $\AA$ in \eqn{Aset2}.  By
Lemma \ref{lmD12PR}, $D_f (X^n_m, \cdot)$ and $D_f(X^n_M,
\cdot)$ are positive recurrent on $B^n (\delta, \eta)$ in
\eqn{posrec-delta}. Thus, by Corollary \ref{corStochBd},
$D_f (X^n (t), \cdot)$ is positive recurrent on $B^n (\delta, \eta)$
as well for $0 \le t \le \delta$. Hence, the claim holds.
\end{proof}

The next important step
is the following lemma, proved in \S \ref{secEP4.5}, after
establishing preliminary bounding lemmas.
\begin{lemma}{\em $($convergence of the integral terms$)$} \label{lmConvInt}
There exists $\delta$ with $0 < \delta \le \tau$ for
$\tau$ in Theorem \ref{thFluidScaled}, such that for any
$\epsilon > 0$ and $t$ with $0 \le t < \delta$, there exists
$\sigma \equiv \sigma(\epsilon, \delta, t)$ with $0 < \sigma <
\delta - t$ and $n_0$ such that \bequ \label{limInt} P
\left(|\frac{1}{\sigma}\int_t^{t + \sigma} 1_{\{D^n_{1,2}(s) >
0\}}\barz^n_{1,2}(s)\ ds -  \pi_{1,2}(\barx(t))\barz_{1,2}(t)|
>  \epsilon \right) < \epsilon \qforallq n \ge n_0. \eeq
\end{lemma}

In order to apply Lemma \ref{lmConvInt} to prove Lemma
\ref{lmKey}, we exploit the absolute continuity of
$\bar{I}_{q,1,1}$, established now.
\begin{lemma}{$($absolute continuity of $\bar{I}_{q,1,1})$}\label{lmAC}
The limiting integral term $\bar{I}_{q,1,1}$ almost surely
satisfies \bequ \label{AC} 0 \le \bar{I}_{q,1,1} (t + u) -
\bar{I}_{q,1,1} (t) \le m_2 u \qforallq 0 \le t < t + u < \tau,
\eeq and so $\bar{I}_{q,1,1}$ is the cumulative distribution
function corresponding to a finite measure, having a density
$h$ depending on $\bar{X}$.  As a consequence, for all $\ep >
0$, there exists $u_0 > 0$ such that \bequ \label{density}
\left|\frac{\bar{I}_{q,1,1} (t + u) - \bar{I}_{q,1,1} (t)}{u} -
h(t)\right| < \ep \eeq for all $u < u_0$.
\end{lemma}

\begin{proof}
Since \beq \bar{I}^n_{q,1,1} (t+ u) - \bar{I}^n_{q,1,1} (t) =
\int_{t}^{t+ u} 1_{\{D^n_{1,2} (s) > 0\}} \barz^n_{1,2} (s) \,
ds \le u m^n_2/n, \eeqno $I^n_{q,1,1} (t)$ is a nondecreasing
function with $0 \le I^n_{q,1,1} (t+ u) - I^n_{q,1,1} (t) \le
m^n_2$ for all $0 \le t < t + u \le \tau$.  Hence, $I^n_{q,1,1}
(t)$ is a a cumulative distribution function associated with a
finite measure. The convergence obtained along the subsequence
based on tightness then yields \beq 0 \le \bar{I}_{q,1,1} (t +
u) - \bar{I}_{q,1,1} (t) \le m_2 u \qforallq 0 \le t < t + u
\le \tau. \eeqno Hence, $\bar{I}_{q,1,1}$ has a density with
respect to Lebesgue measure, as claimed.
\end{proof}

\begin{proofof}{Lemma \ref{lmKey}}
Given Lemma \ref{lmConvInt}, for any $\ep > 0$, we can find
$\sigma$ and $n_0$ such that \eqn{limInt} is valid for all $n
\ge n_0$. Hence, we can let $n \ra \infty$ and conclude that,
for any $\ep > 0$, we can find $\sigma$ such that \bequ
\label{limInt200} P \left(|\frac{1}{\sigma}(\bar{I}_{q,1,1} (t
+ \sigma) - \bar{I}_{q,1,1} (t)) -
\pi_{1,2}(\barx(t))\barz_{1,2}(t)| >  \epsilon \right) <
\epsilon. \eeq However, given that \eqn{density} and
\eqn{limInt200} both hold, we conclude that we must almost
surely have $h(t) = \pi_{1,2}(\barx(t))\barz_{1,2}(t)$, which
completes the proof.
\end{proofof}

We now apply bounds to prove Lemma \ref{lmConvInt}.
The comparisons in Lemma \ref{lmD12PR} and Corollary
\ref{corStochBd} are important, but they are not directly
adequate for our purpose. The sample-path stochastic order
bound in Corollary \ref{corStochBd} enables us to prove Lemma
\ref{lmConvInt}, and thus Theorem \ref{th1},
 for the special case of $r_{1,2} = 1$, because then $\zeta = 0$, where $\zeta$ is the gap in Corollary \ref{corStochBd}, but not more generally
 when the gap is positive.
However, we now show that an actual gap will only be present
rarely, if we choose the interval length small enough and $n$
big enough. We use the construction in the proof of Lemma \ref{lmD12PR}.
exploiting the fact that we have rate order, where the bounding
rates can be made arbitrarily close to each other by choosing
the interval length suitably small.  We again construct a
sample path stochastic order, but for the time averages of the
time spent above $0$. We prove the following lemma in \S
\ref{secEP4.4}.  We use it to prove Lemma \ref{lmConvInt} in \S
\ref{secEP4.5}.

\begin{lemma} \label{lmNewBds}
Suppose that conditions of Lemma {\em \ref{lmD12PR}} hold.
Then, for any $\ep > 0$, there exists $n_0$ and $\xi$ with $0
<\xi \le \delta$  for $\delta$ in Lemma {\em \ref{lmD12PR}},
such that, in addition to the conclusions of Lemma {\em
\ref{lmD12PR}} and Corollary {\em \ref{corStochBd}}, the states
$x_m, x_M \in \AA$ and random vectors $X^n_m, X^n_M$ and
associated frozen processes $D^n_f (X^n_m)$ and $D^n_f (X^n_M)$
can be chosen so that on $B_n (\xi, \eta)$ $($defined as in
{\em \eqn{posrec-delta}} with $\xi$ instead of $\delta$
 and necessarily
$P(B_n (\xi, \eta)) \ra 1$ as $n \ra \infty)$, first \bequ
\label{initialLim}
 D^n_f (X^n_m, 0) = D^n_f (X^n_M, 0) = D^n_{1,2} (0)
\eeq and, second, \bequ \label{DstoBd3} \bsplit
\frac{1}{t} \int_{0}^{t} 1_{\{D^n_f (X^n_m, s) > 0\}} \, ds -
\epsilon & \le_{st} \frac{1}{t} \int_{0}^{t} 1_{\{D^n_{1,2} (s)
> 0\}} \, ds \le_{st} \frac{1}{t} \int_{0}^{t} 1_{\{D^n_f
(X^n_M, s) > 0\}} \, ds + \epsilon
\end{split}
\eeq for $n \ge n_0$ and $0 \le t \le \xi$; i.e., there is
sample path stochastic order in $\D([0, \xi])$ for $n \ge n_0$.
\end{lemma}


\subsection{Continuity of the FTSP QBD.}\label{secCont}

In the current proof of Lemma \ref{lmKey} and thus Theorem \ref{th1}, we
will ultimately reduce everything down to the behavior of the
FTSP $D$. First, we intend to analyze the inhomogeneous
queue-difference processes $D^n_{1,2} (\Gamma^n)$ in \eqn{Dprocess} in terms of
associated frozen processes $D^n_f (\Gamma^n)$ introduced in \S
\ref{secAux}, obtained by freezing the transition rates at the
transition rates in the initial state $\Gamma^n$.  In
\eqn{ident} above, we showed that the frozen-difference
processes can be represented directly in terms of the FTSP, by
transforming the model parameters $(\lambda_i, m_j)$ and the
fixed initial state $\gamma$ and scaling time.  We will appropriately bound the
queue-difference processes $D^n (\Gamma^n)$ above and below by
associated frozen-queue difference processes,
 and then transform them into
versions of the FTSP $D$.  For the rest of the proof of Theorem
\ref{th1} in \S \ref{secProofs}, we will exploit a continuity
property possessed by this family of pure-jump Markov
processes, which exploits their representation as QBD processes
using the construction in \S 6 of \cite{PeW09c}.  We will be
applying this continuity property to the FTSP $D$.

To set the stage, we review basic properties of the QBD
process.  We refer to \S 6 of \cite{PeW09c} for important
details.  From the transition rates defined in
\eqn{bd1}-\eqn{bd4}, we see that there are only $8$ different
transition rates overall. The generator $Q$ (in (65) of
\cite{PeW09c}) is based on the four basic $2m \times 2m$
matrices $B$, $A_0$, $A_1$, and $A_2$, involving the $8$
transition rates (as shown in (66) of \cite{PeW09c}). By
Theorem 6.4.1 and Lemma 6.4.3 of \cite{LR99}, when the QBD is
positive recurrent, the FTSP steady-state probability vector
has the matrix-geometric form $\af_n = \af_0 R^n$, where
$\alpha_n$ and $\alpha_0$ are $1 \times 2m$ probability vectors
and $R$ is the $2m \times 2m$ rate matrix, which is the minimal
nonnegative solutions to the quadratic matrix equation $A_0  +
R A_1  + R^2 A_2 = 0$, and can be found efficiently by existing
algorithms, as in \cite{LR99}; See \cite{PeW09c} for
applications in our settings. If the drift condition
\eqn{Aset2} holds, then the spectral radius of $R$ is strictly
less than $1$ and the QBD is positive recurrent
 (Corollary 6.2.4 of \cite{LR99}).  As a consequence, we have
$\sum_{n=0}^{\infty}{R^n} = (I-R)^{-1}$. Also, by Lemma 6.3.1
of \cite{LR99}, the boundary probability vector $\af_0$ is the
unique solution to the system $\af_0(B + R A_2) = 0$ and $\af
\textbf{1} = \af_0 (I-R)^{-1} {\bf 1} = 1$.  See \S 6.4 of
\cite{PeW09c} for explicit expressions for $\pi_{1,2} (\gamma)$
for $\gamma \in \AA$.

As in Lemma \ref{lmReturnTime}, we also use the return time to a fixed
state, $\tau$, and its mgf $\phi_{\tau} (\theta)$ with
a critical value
$\theta^* > 0$ such that $\phi_{\tau} (\theta) < \infty$ for
$\theta < \theta^{*}$ and $\phi_{\tau} (\theta) = \infty$ for
$\theta > \theta^{*}$.
We will be interested in the {\em cumulative process}
\bequ \label{cumulative}
C(t) \equiv \int_0^{t} (f(D(s)) - E[f(D(\infty))]) \, ds \quad t \ge 0,
\eeq
for the special
function $f (x) \equiv 1_{\{x \ge 0\}}$.  Cumulative processes
associated with regenerative processes obey CLT's and FCLT's,
depending upon assumptions about the basic cycle random
variables $\tau$ and $\int_0^{\tau} f(D(s)) \, ds$, where we
assume for this definition that $D(0) = s^{*}$; see \S VI.3 of
\cite{Asmussen} and \cite{GW93}.  From \cite{B80}, we have the
following CLT with a Berry-Esseen bound on the rate of
convergence (stated in continuous time, unlike \cite{B80}): For
any bounded measurable function $g$, there exists $t_0$ such
that \bequ \label{roc} |E[g(C(t))/\sqrt{t}] - E[g(N(0, \sigma^2
))]| \le \frac{K}{\sqrt{t}} \qforallq t > t_0, \eeq where \bequ
\label{var} \sigma^2  \equiv E\left[\left(\int_0^{\tau} f(D(s))
- E[f(D(\infty))] \, ds \right)^2\right], \eeq again assuming
for this definition that $D(0) = s^{*}$. The constant $K$
depends on the function $g$ (as well as the function $f$ in
\eqn{cumulative}) and the third absolute moments of the basic
cycle variables defined above, plus the first moments of the
corresponding cycle variables in the initial cycle if the
process does not start in the chosen regenerative state.

There is significant simplification in our case, because the
function $f$ in \eqn{cumulative} is an indicator function.
Hence, we have the simple domination: \bequ \label{CycleVar}
\int_0^{\tau } |f(D(s)| \, ds = \int_0^{\tau } f(D(s)) \, ds
\le \tau \quad \mbox{w.p.1} \eeq As a consequence, boundedness
of absolute moments of both cycle variables reduces to the
moments of the return times themselves, which are controlled by
the mgf.


We will exploit the following continuity result for QBD's,
which parallels previous continuity results for Markov
processes, e.g., \cite{K75,W80}.

\begin{lemma} {$($continuity of QBD's$)$} \label{lmQBDcont}
Consider a sequence of irreducible, positive recurrent QBD's
having the structure of the fundamental QBD associated with the
FTSP in \S {\em \ref{secFTSP}} and \S $6$ in {\em
\cite{PeW09c}} with generator matrices $\{Q_n: n \ge 1\}$ of
the same form. If $Q_n \ra Q$ as $n \ra \infty$ $($which is determined by the convergence of the $8$ parameters$)$, where the
positive-recurrence drift condition {\em \eqn{Aset2}} holds for
$Q$, then there exists $n_0$ such that the positive-recurrence
drift condition  {\em \eqn{Aset2}} holds for $Q_n$ for $n \ge
n_0$. For $n \ge n_0$, the quantities $(R, \alpha_0, \alpha,
\phi_{\tau}, \theta^{*}, \psi_{N}, z^{*}, \sigma^2, K)$ indexed
by $n$ are well defined for $Q_n$, where $\sigma^2$ and $K$ are
given in {\em \eqn{roc}} and {\em \eqn{var}}, and converge as
$n \ra \infty$ to the corresponding quantities associated with
the QBD with generator matrix $Q$.
\end{lemma}

\begin{proof}
First, continuity of $R$, $\alpha_0$ and $\af$ follows from the
stronger differentiability in an open neighborhood of any
$\gamma \in \AA$, which was shown to hold in the proof of
Theorem 5.1 in \cite{PeW09c}, building on Theorem 2.3 in
\cite{H95}.  The continuity of $\sigma^2 $ follows from the
explicit representation in \eqn{var} above (which corresponds
to the solution of Poisson's equation).  We use the QBD
structure to show that the basic cycle variables $\tau$ and
$\int_0^{\tau } f(D(s)) \, ds$ are continuous function of $Q$,
in the sense of convergence in distributions (or convergence of
mgf's and gf's) and then for convergence of all desired
moments, exploiting \eqn{CycleVar} and the mgf of $\tau$ to get
the required uniform integrability. Finally, we get the
continuity of $K$ from \cite{B80} and the continuity of the
third absolute moments of the basic cycle variables, again
exploiting the uniform integrability. We will have convergence
of the characteristic functions used in \cite{B80}.  However,
we do not get an explicit expression for the constants
$K$.\end{proof}

We use the continuity of the steady-state distribution $\alpha$
in \S \ref{secEP4.5}; specifically in \eqn{steadyBds} .  In
addition, we use the following corollary to Lemma
\ref{lmQBDcont} in \eqn{double} in \S \ref{secEP4.5}.  We use
the notation in \eqn{ident}.

\begin{coro}\label{corDouble}
If $(\bar{\lambda}^n_i, \bar{m}^n_j, \bar{\gamma}_n) \ra
(\lambda_i, m_j, \gamma)$ as $n \tinf$ for our FTSP QBD's,
where {\em \eqn{Aset2}} holds for $(\lambda_i, m_j, \gamma)$,
then for all $\epsilon > 0$ there exist $t_0$ and $n_0$ such
that \bes P\left(| \frac{1}{t}\int_{0}^{t} 1_{\{ D(\lambda^n_i,
m^n_j, \gamma_n, s) > 0\}}\, ds - P(D(\lambda_i, m_j, \gamma,
\infty) > 0)| > \epsilon\right) < \epsilon \ees for all $t \ge
t_0$ and $n \ge n_0$.
\end{coro}

\begin{proof}
First apply Lemma \ref{lmQBDcont} for the steady-state
probability vector $\alpha$, to find $n_0$ such that
$|P(D(\lambda^n_i, m^n_j, \gamma_n, \infty) > 0)| -
P(D(\lambda_i, m_j, \gamma, \infty) > 0)| < \epsilon/2$ for all
$n \ge n_0$.  By the triangle inequality, henceforth it
suffices to work with $P(D(\lambda^n_i, m^n_j, \gamma_n,
\infty) > 0)$ in place of $P(D(\lambda_i, m_j, \gamma, \infty)
> 0)$ in the statement to be proved. By \eqn{roc}, for any $M$,
there exists $t_0$ such that for all $t \ge t_0$, \bequ
\label{pf23} \bsplit
& P\left(| \frac{1}{t}\int_{0}^{t} 1_{\{ D(\lambda^n_i, m^n_j, \gamma_n, s) > 0\}}\, ds - P(D(\lambda^n_i, m^n_j, \gamma_n, \infty) > 0)| > \frac{M}{\sqrt{t}} \right) \\
& \quad \quad < P(|N(0, \sigma^2 (\lambda^n_i, m^n_j,
\gamma_n))| > M) + \frac{K(\lambda^n_i, m^n_j,
\gamma_n)}{\sqrt{t}}.
\end{split}
\eeq We get \eqn{pf23} from \eqn{roc} by letting $f (x) =
1_{(0,\infty)}(x)$ in \eqn{cumulative} and letting $g(x) =
1_{\{|x| > M\}} (x)$ in \eqn{roc}.

Next, choose $M$ so that $ P(|N(0, \sigma^2 (\lambda_i, m_j,
\gamma))| > M) < \epsilon/2$. Then, invoking Lemma
\ref{lmQBDcont}, increase $n_0$ and $t_0$ if necessary so that
$|\sigma^2 (\lambda^n_i, m^n_j, \gamma_n)) - \sigma^2
(\lambda_i, m_j, \gamma))|$ and $|K(\lambda^n_i, m^n_j,
\gamma_n) - K(\lambda_i, m_j, \gamma)|$ are sufficiently small
so that the right side of \eqn{pf23} is less than $\epsilon/2$
for all $n \ge n_0$ and $t \ge t_0$. If necessary, increase
$t_0$ and $n_0$ so that $M/\sqrt{t_0} < \epsilon/2$. With those
choices, the objective is achieved.\end{proof}

\subsection{Proof of Lemma \ref{lmConvInt}.} \label{secEP4.5}

In this subsection we give the first of two long proofs of previous lemmas.
We now prove Lemma \ref{lmConvInt}, which completes the alternate proof of Theorem \ref{th1}
in the case $r_{1,2} = 1$, because then Corollary \ref{corStochBd} holds with gap $\zeta = 0$.
In that case, Corollary \ref{corStochBd} directly implies that Lemma \ref{lmNewBds} holds with $\ep = 0$.
Lemma \ref{lmNewBds}, which is proved in the following subsection, is required to treat more general $r_{1,2}$.

First, let $\delta > 0$, $\epsilon > 0$ and $t$ with $0 < t < \delta$ be given, where the $\delta$ is chosen
as in Lemma \ref{lmD12PR} (with $\delta \le \tau$ for $\tau$ in Theorem \ref{thFluidScaled}).
Here we will be introducing a new interval $[t, t + \sigma]$, where $0 < \sigma \equiv \sigma (t) \equiv \sigma (t, \epsilon, \delta) < \delta - t$,
so that $[t, t+\sigma] \subset [0, \delta]$.  Moreover, we will make $\sigma < \xi$ for $\xi$
existing via Lemma \ref{lmNewBds}.
Lemmas \ref{lmD12PR} and \ref{lmNewBds} hold on the interval $[t, t + \xi]$,
where $\xi \equiv \xi (t)$ satisfies $0 < \xi < \delta - t$.  We will be choosing $\sigma$ with $0 < \sigma < \xi$.

Before we started with regularity conditions at time $0$ provided by Assumption \ref{assC}.
We now will exploit tightness to get corresponding regularity conditions at time $t$ here.
In particular, before
 we started with the convergence $\bar{X}^n (0) \Rightarrow x(0)$ in $\RR^3$ where $x(0) \in \AA$ based on Assumption \ref{assC}.
Now, instead, we base the convergence $\bar{X}^n (t) \Rightarrow \barx (t)$ at time $t$
on the convergence established along the converging subsequence (without introducing new subsequence notation).
We apply Lemma \ref{lmXinA} to deduce that $P(\barx (t) \in \AA) = 1$.
 Since the frozen processes to be constructed are Markov processes, we can construct the processes
after time $t$, given only the value of $X^n (t)$, independently of what happens on $[0,t]$.
Before, we also started with $D^n_{1,2} (0) \Rightarrow L$, where $L$ is a finite random variable.
Instead, here we rely on the stochastic boundedness (and thus tightness) of $\{D^n_{1,2} (t): n \ge 1\}$ in
$\RR$ provided by Theorem \ref{thDSB}.
As a consequence, the sequence $\{(X^n (t), D^n_{1,2} (t)): n \ge 1\}$ is tight in $\RR_2$.
Thus there exists a convergent subsequence of the latest subsequence we are considering.
Hence, without introducing subsequence notation, we have $(\barx^n (t), D^n_{1,2} (t)) \Ra (\barx (t), L (t))$ in $\RR_2$ as $n \ra \infty$,
where $(\barx (t), L (t))$ is a finite random vector with $P(\barx (t) \in \AA) = 1$.

We use the same construction used previously in
 the proofs of Lemmas \ref{lmD12PR} and \ref{lmNewBds},
letting $\sigma$ decrease to achieve new requirements in addition to the
conclusions deduced before.
We now regard $t$ as the time origin for the frozen difference processes.  Given $D^n_{1,2} (t)$, let
the two associated frozen difference processes be $\{D^n_f (X^{n}_{M}, s) : s \ge t\}$
and $\{D^n_f (X^{n,\xi}_{m}, s) : s \ge 0\}$.  We directly let
\bequ \label{initT}
D^n_f (X^{n}_{M}, t) = D^n_f (X^{n}_{m}, t) = D^n_f (X^{n} (u), t) = D^n_{1,2} (t), \quad u \ge t,
\eeq
so that we can invoke property \eqn{initialLim} at time $t$ in our application of Lemma \ref{lmNewBds} here.
As before, the initial random states $X^{n}_{M}$ and $X^{n}_{m}$ and their fluid-scaled limits are chosen to achieve the goals before and here.

We now successively decrease upper bounds on $\sigma$ and increase lower bounds on $n$ until we achieve \eqn{limInt}
in Lemma \ref{lmConvInt}.
First, we can apply Lemma \ref{lmD12PR} to find an $n_1$ such that $P(B^n(\delta, \ep) > 1 - \ep/6$ for $n \ge n_1$.
We next apply Lemma \ref{lmNewBds} to conclude that there exists $\sigma_1$
such that the following variants of the integral inequalities in \eqn{DstoBd3} hold with probability
at least $1 - \epsilon/6$ as well:
\bequ \label{newIntIneq}
\bsplit
\frac{1}{\sigma} \int_{t}^{t +\sigma} 1_{\{D^n_f (X^{n}_{m}, s) > 0\}} \, ds - \frac{\epsilon}{6 m_2}
\le \frac{1}{\sigma} \int_{t}^{t + \sigma} 1_{\{D^n_{1,2} (s) > 0\}} \, ds
\le \frac{1}{\sigma} \int_{t}^{t + \sigma} 1_{\{D^n_f (X^{n}_{M}, s) > 0\}} \, ds + \frac{\epsilon}{6 m_2}.
\end{split}
\eeq
for all $\sigma \le \sigma_1$
(We divide by $m_2$ because we will be multiplying by $z_{1,2} (t)$.)

We now present results only for $X^n_M$, with the understanding that corresponding results hold for $X^n_m$.
We represent the bounding frozen queue-difference processes directly in terms
of the FTSP, using the relation \eqn{ident}, with the notation introduced there:
\bequ\label{FTSPrep}
\bsplit
\{D^n_f (\lambda^n_i, m^n_j, X^{n}_{M}, t + s): s \ge 0\} & \deq \{D(\lambda^n_i/n, m^n_j/n, X^{n}_{M}/n, t + sn): s \ge 0\}.
\end{split}
\eeq
Upon making a change of variables, the bounding integrals in \eqn{newIntIneq} become
\bequ\label{FTSPrep2}
\bsplit
\frac{1}{\sigma}  \int_{t}^{t + \sigma} 1_{\{D^n_f (\lambda^n_i, m^n_j, X^{n}_{M}, s) > 0\}} \, ds
& \deq   \frac{1}{n\sigma} \int_{t}^{t +n \sigma} 1_{\{D (\lambda^n_i/n, m^n_j/n,X^{n,\sigma}_{m}/n, s)> 0\}} \, ds.
\end{split}
\eeq
For each integer $k$, we can apply Lemma \ref{lmQBDcont} to obtain the iterated limits
\bequ\label{iterated}
\bsplit
& \lim_{n \ra \infty} \lim_{s \ra \infty} P(D (\lambda^n_i/n, m^n_j/n, X^{n}_{M}/n, s) = k)
= \lim_{s \ra \infty} \lim_{n \ra \infty} P(D (\lambda^n_i/n, m^n_j/n, X^{n}_{M}/n, s) = k),
\end{split}
\eeq
where the limit is 
$P(D (x_M, \infty) = k) \equiv P(D (\lambda_i, m_j, x_M, \infty) = k)$.
In particular, the limit for the model parameters first implies convergence of the generators.
The convergence of the generators then implies convergence of the processes.  Finally Lemma \ref{lmQBDcont} implies convergence of the associated
steady-state distributions.

By Corollary \ref{corDouble}, we also have the associated double limit for the averages over intervals of length $O(n)$ as $n \ra \infty$.  As $n \ra \infty$,
\bequ\label{double}
\bsplit
\frac{1}{n\sigma} \int_{t}^{t +n \sigma} 1_{\{D (\lambda^n_i/n, m^n_j/n, X^{n}_{M}/n, s) > 0\}} \, ds & \Rightarrow  P(D(\lambda_i, m_j, x_M, \infty) > 0)
\equiv \pi_{1,2} (x_M).
\end{split}
\eeq

Invoking Lemma \ref{lmQBDcont}, choose $\sigma_2$ less than or equal to the previous value $\sigma_1$ such that
\bequ \label{steadyBds}
|\pi_{1,2} (x_m) - \pi_{1,2} (\bar{X}(t))| \le \frac{\epsilon}{6 m_2} \qforallq \sigma \le \sigma_2.
\eeq
For that $\sigma_2$, applying \eqn{double}, choose $n_2 \ge n_1$ such that for all $n \ge n_2$,
\bequ\label{IntDiff}
\bsplit
 & P\left( |\frac{1}{n\sigma_2} \int_{t}^{t +n \sigma_2} 1_{\{D (\lambda^n_i/n, m^n_j/n, X^{n}_{M}/n, s) > 0\}} \, ds
 - \pi_{1,2} (x_M)|  > \frac{\epsilon}{6 m_2} \right) < \frac{\epsilon}{6}.
\end{split}
\eeq

We now use the convergence of $\barx^n$ along the subsequence over $[0,t]$ together with the tightness of the
sequence of processes $\{\bar{X}^n: n \ge 1\}$ to control $\bar{Z}^n_{1,2}$ in an interval after time $t$.
In particular, there exists $\sigma_3 \le \sigma_2$ and $n_3 \ge n_2$ such that
\bequ \label{BdafterT}
P(\sup_{u: t \le u \le t + \sigma_3}{\{|\bar{X}^n (u) - \bar{X}(t)|\}} > \epsilon/6) < \epsilon/6 \qforallq n \ge n_3.
\eeq
For the current proof, we will use the consequence
\bequ \label{BdafterT2}
P(\sup_{u: t \le u \le t + \sigma_3}{\{|\bar{Z}^n_{1,2} (u) - \bar{Z}_{1,2}(t)|\}} > \epsilon/6) < \epsilon/6 \qforallq n \ge n_3.
\eeq
We let the final value of $\sigma$ be $\sigma_3$.
We now show the consequences of the selections above.
We will directly consider only the upper bound; the reasoning for the lower bound is essentially the same.
Without loss of generality, we take $\epsilon \le 1\wedge m_2$.
From above, we have the following relations (explained afterwards) holding with probability
at least $1 - \epsilon$ (counting $\epsilon/6$ once each to achieve $\|x_M - \barx (t)\| \le \ep/6$, $\|x_m - \barx (t)\| \le \ep/6$,
\eqn{newIntIneq}, \eqn{BdafterT2} and twice for \eqn{IntDiff}):

\bequ \label{string}
\bsplit
(a)\; & \int_t^{t + \sigma} 1_{\{D^{n}_{1,2}(s) > 0\}}\barz^n_{1,2}(s)\ ds  \le \left(\bar{Z}_{1,2} (t) + \frac{\epsilon}{6}\right) \int_t^{t + \sigma} 1_{\{D^{n}_{1,2}(s) > 0\}} \ ds  \\
(b)\; &  \quad \quad \le \left(\bar{Z}_{1,2} (t) + \frac{\epsilon}{6}\right) \left(\int_{t}^{t + \sigma} 1_{\{D^n_f (\lambda^n_i, m^n_j, X^n_M, s) > 0\}} \, ds
+ \frac{\epsilon \sigma}{6 m_2}\right) \\
(c)\; &  \quad \quad \quad  \deq \left(\bar{Z}_{1,2} (t) + \frac{\epsilon}{6}\right) \left(\int_{0}^{\sigma} 1_{\{D(\lambda^n_i/n, m^n_j/n, X^{n}_{M}/n, t + sn) > 0\}} \, ds
+ \frac{\epsilon \sigma}{6 m_2}\right) \\
(d)\; &  \quad \quad\quad  \deq \left(\bar{Z}_{1,2} (t) + \frac{\epsilon}{6}\right) \sigma
\left(\frac{1}{n \sigma} \int_{0}^{n\sigma} 1_{\{D(\lambda^n_i/n, m^n_j/n, X^{n}_{M}/n,t + s) > 0\}} \, ds + \frac{\epsilon}{6 m_2}\right) \\
(e)\; &  \quad \quad \le  \left(\bar{Z}_{1,2} (t) + \frac{\epsilon}{6}\right) \sigma \left(\pi_{1,2} (x_M) + \frac{2 \epsilon}{6 m_2}\right) \\
(f)\; &  \quad \quad \le \left(\bar{Z}_{1,2} (t) + \frac{\epsilon}{6}\right) \sigma \left(\pi_{1,2} (\bar{X}(t)) + \frac{3\epsilon}{6 m_2}\right) \\
(g)\; &  \quad \quad \le \bar{Z}_{1,2} (t) \pi_{1,2} (\bar{X}(t)) \sigma +  \frac{\pi_{1,2} (\bar{X}(t))}{6} \epsilon \sigma + \frac{1}{2} \epsilon \sigma
+ \frac{\sigma \epsilon^2}{12 m_2} \\
(h)\; &  \quad \quad \le \bar{Z}_{1,2} (t) \pi_{1,2} (\bar{X}(t)) \sigma +  \frac{3}{4} \epsilon \sigma\\
& \quad \quad \le (\bar{Z}_{1,2} (t) \pi_{1,2} (\bar{X}(t)) + \ep) \sigma \qforallq n \ge n_0 \equiv n_3.
\end{split}
\eeq

We now explain the steps in \eqn{string}:  First, for (a) we replace $\barz^n_{1,2} (s)$ by $\bar{Z}_{1,2} (t)$ for $t \le s \le t + \xi$ by applying \eqn{BdafterT2}.
For (b), we apply Lemma \ref{lmNewBds}.  For (c), we use the alternative representation in terms of the FTSP in
\eqn{FTSPrep}.  For (d), we use the change of variables in \eqn{FTSPrep2}.
For (e), we use \eqn{IntDiff}, exploiting the convergence in \eqn{double}.
For (f), we use \eqn{steadyBds}.  Step (g) is simple algebra, exploiting $\bar{Z}_{1,2} (t) \le m_2$.
Step (h) is more algebra, exploiting $\pi_{1,2} (\bar{X}(t)) \le 1$, and $\epsilon \le 1\wedge m_2$.
That completes the proof of the lemma.
\qed

\subsection{Proof of Lemma \ref{lmNewBds}.} \label{secEP4.4}
We will directly show how to carry out the construction of the sample path order for stochastic processes having the same distributions
as the terms in \eqn{DstoBd3}.  We will show that the interval $[0, \xi]$ can be divided into a large number ($O(n)$)
of alternating subintervals, each of length $O(1/n)$, such that full sample path order holds on one subinterval, and then the processes are unrelated
on the next subinterval.  We will then show, for an appropriate choice of $\xi$, that these intervals can be chosen so that the first intervals
where order holds are much longer than the second intervals where the processes are unrelated.  Hence we will deduce for this construction that
the inequalities in \eqn{DstoBd3} holds w.p.1 for all $n$ sufficiently large.
We will specify the steps in the construction and explain why they achieve their goal,
while minimizing the introduction of new notation.  We highlight twelve key steps in the argument.

{\em An integer state space.}  We exploit Assumption \ref{assE} to have $r_{1,2}$ be rational, and then use the
integer state space associated with the QBD representation in \S 6 of \cite{PeW09c}.  However,
we do not directly exploit the QBD structure; instead we directly work with the CTMC's
with transitions $\pm j$ and $\pm k$, where $j$ and $k$ are positive integers.  We will reduce the analysis to consideration of homogeneous
CTMC's.

{\em Exploiting Lemma {\em \ref{lmD12PR}}, but reducing the length of the interval.}  We construct the states $x_m, x_M \in \AA$ and random vectors $X^n_m, X^n_M$
and associated frozen processes $D^n_f (X^n_m)$ and $D^n_f (X^n_M)$
exactly as in Lemma \ref{lmD12PR}, but we adjust the specific values
to make them closer together at time $0$ as needed by choosing $\xi$ to be suitably smaller than the
$\delta$ needed in Lemma \ref{lmD12PR}.  It is significant that here $\xi$ can be arbitrarily small;
we only require that $\xi > 0$.

{\em Initializing at time $0$.}  By Assumption \ref{assC}, we have $\barx^n \Ra x(0) \in \AA \subset \RR_3$ and $D^n_{1,2} (0) \Ra L \in \RR$.
Since $x(0)$ is deterministic, we immediately have the joint convergence $(\barx^n, D^n_{1,2} (0))  \Ra (x(0),L) \in \RR_4$;
apply Theorem 11.4.5 of \cite{W02}.
We then apply the Skorohod representation theorem to replace the convergence in distribution
with convergence w.p.1., without changing the notation.
Hence, we can start with $(\barx^n, D^n_{1,2} (0))  \ra (x(0),L) \in \RR_4$ w.p.1.
With that framework, we condition on $D^n_{1,2} (0)$.  We then
initialize the processes
$D^n_f (X^n_m, 0)$ and $D^n_f (X^n_M, 0)$ 
to satisfy $D^n_f (X^n_m, 0) = D^n_f (X^n_M, 0) = D^n_{1,2} (0)$, as in \eqn{initialLim}.

{\em Coupling over an initial random interval.}  Given identical initial values, we can apply the rate order in Lemma \ref{lmD12PR}
to construct versions of the processes that will be ordered w.p.1 over an initial interval of random positive length, just as in
Corollary \ref{corStochBd}.
Let $\nu_n$ be the first time that the sample path order is violated.  The rate order implies that $\nu_n > 0$ w.p.1 for each $n \ge 1$.
The coupling is performed over the interval $[0, \nu_n]$.
At random time $\nu_n$, we must have all three processes in the boundary set of states
$\{j: -\beta +1 \le j \le \beta\}$, where $\beta \equiv j \vee k$.  That is so, because violation of order only need occur when,
just prior to the order violation, we have
$-\beta < D^n_f (X^n_m) \le 0 < 1 \le D^n_f (X^n_M) < \beta$, where the rates are no longer ordered properly, because the processes are in different
regions.
At time $\nu_n$, either the upper process jumps down below the lower process
or the lower process jumps up over the upper process.

{\em A new construction using independent versions when order is first violated.}
At time $\nu_n$, the order is first violated and would remain violated over an interval thereafter.
However, at this time $\nu_n$, we alter the construction.
At this random time $\nu_n$ we temporarily abandon
the coupling based on rate order.  Instead, going forward from time $\nu_n$, we construct independent versions of the three
processes being considered.  More precisely, the three processes are conditionally independent, given their initial values at time $\nu_n$.
The idea is to let them evolve in this independent manner until the three processes
reach a state in which the desired sample path ordering does again hold.
We do this in a simple controlled manner.  We wait until, simultaneously, the upper bound process
$D^n_f (X^n_M, \nu_n + t)$  exceeds a suitably high threshold,
the lower bound process
$D^n_f (X^n_m, \nu_n + t)$  falls below a suitably low threshold,
and the interior processes
 $D^n_{1,2} (\nu_n + t)$ is in a middle region.
 At such a random (stopping) time, the three processes will necessarily be ordered
 in the desired way.  After that time, we can use the coupling again.

 {\em Avoiding working directly with $D^n_{1,2}$.}
 However, since $D^n_{1,2} (\nu_n + t)$ is an inhomogeneous CTMC, and thus difficult to work with,
 we avoid working with it directly.  Instead,
we simultaneously construct new
upper and lower bound frozen processes, $\tilde{D}^n_f(X^n_M)$ and
$\tilde{D}^n_f(X^n_m)$ starting at time $\nu_n$, coupled with $\{D^n_{1,2} (\nu_n + t):t \ge 0\}$, initialized by stipulating that
$\tilde{D}^n_f(X^n_M, \nu_n) = \tilde{D}^n_f(X^n_m, \nu_n) = D^n_{1,2} (\nu_n)$.
These new processes are coupled with $\{D^n_{1,2} (\nu_n + t):t \ge 0\}$, but conditionally independent of
the independent versions of the other two processes
$\{D^n_f (X^n_M, \nu_n + t): t \ge 0\}$ and  $\{D^n_f (X^n_m, \nu_n + t): t \ge 0\}$.
We again can apply Lemma \ref{lmD12PR} to obtain rate order, but again we cannot have full sample path order, because of the gap.
Nevertheless, we can now apply Corollary \ref{corStochBd} to conclude that we have the sample path stochastic order
\beq
\tilde{D}^n_f(X^n_m, \nu_n + t) - \beta \le_{st} D^n_{1,2} (\nu_n + t) \le_{st} \tilde{D}^n_f(X^n_M, \nu_n + t) + \beta \qinq D([\nu_n, \delta]).
\eeqno
We thus can do a sample path construction to achieve
\bequ \label{gap3}
\tilde{D}^n_f(X^n_m, \nu_n + t) - \beta \le D^n_{1,2} (\nu_n + t) \le \tilde{D}^n_f(X^n_m, \nu_n + t) + \beta \qforq 0 \le t \le \delta - \nu_n \quad w.p.1.
\eeq

{\em Stopping times for desired order to be achieved.}
We now define stopping times for the order to be achieved.  Let
\bequ \label{phiDef}
\phi_{n}  \equiv  \inf{\{t > 0: D^n_f (X^n_M, \nu_n + t) > 2 \beta, \, D^n_f (X^n_m, \nu_n + t) < -2 \beta, \,
 -  \beta < \tilde{D}^n_f(X^n_M, \nu_n + t) <  \beta\}},
\eeq
with the three processes $D^n_f (X^n_M)$, $D^n_f (X^n_m)$ and $\tilde{D}^n_f(X^n_M)$ in \eqn{phiDef} being mutually conditionally independent,
given that all three are initialized at time $\nu_n$ with the final values there obtained from the evolution of the coupled processes up until time $\nu_n$.
(We have conditional independence given the vector of initial values at time $\nu_n$.)
We can combine \eqn{phiDef} with \eqn{gap3} to conclude that we have the appropriate order at time $\nu_n + \phi_n$, in particular,
\beq
D^n_f (X^n_m, \nu_n + \phi_n) \le D^n_{1,2} (\nu_n + \phi_n) \le D^n_f (X^n_M, \nu_n + \phi_n) \quad \mbox{w.p.1.}
\eeqno
We not only obtain the desired order at time $\nu_n + \phi_n$, but the stopping time $\phi_n$ depends on the three processes
 $\{D^n_f (X^n_M, \nu_n + t): t \ge 0\}$,
   $\{D^n_f (X^n_m, \nu_n + t): t \ge 0\}$ and $\{\tilde{D}^n_f(X^n_M, \nu_n + t): t \ge 0\}$,
   which are are independent homogeneous CTMC's conditional on their random state vectors $X^n_M$ and $X^n_m$, respectively.
   Thus the stopping time is with respect to the three-dimensional vector-valued CTMC (conditional on the random state vectors $X^n_M$ and $X^n_m$).

{\em Alternating Cycles using coupled and independent versions.}
Going forward after the time $\nu_n + \phi_n$ for $\phi_n$ in \eqn{phiDef}, we again construct coupled processes just as we did at time $0$.
However, now we initialize by having the specified states coincide at time $\nu_n + \phi_n$; i.e.,
$D^n_f (X^n_M, \nu_n + \phi_n) = D^n_f (X^n_m, \nu_n + \phi_n) = D^n_{1,2} (\nu_n + \phi_n)$.
 Hence, we can repeat the coupling done over the initial interval $[0, \nu_n]$.
In particular, we can apply the ordering of the initial states and the rate ordering to construct new coupled versions
after time $\nu_n + \phi_n$.  We can thus repeat the construction already done over the time interval $[0, \nu_n]$ after time $\nu_n + \phi_n$.
The coupled construction beginning at time $\nu_n + \phi_n$ will end at the first subsequent time that the order is violated.
At this new random time (paralleling $\nu_n$), we must have all three processes in the boundary set of states
$\{j: -\beta +1 \le j \le \beta\}$.  In this way, we produce a sequence of alternating intervals, where we perform coupling on one interval
and then create independent versions on the next interval.

{\em Applying a FWLLN after scaling time by $n$.}
Since the transition rates of $D^n_{1,2}$ and all the other processes are $O(n)$,
there necessarily will be order $O(n)$ of these cycles in any interval $[0, \xi]$.
However, we scale time by $n$, replacing $t$ by $nt$,
just as in \eqn{ident}, \eqn{FTSPrep} and \eqn{FTSPrep2} to represent all processes as FTSP's
with random parameters depending on $n$, converging to finite deterministic limits.
Once, we scale time by $n$ in that way,
 the scaled transition rates converge to finite limits as $n \ra \infty$, but the relevant time interval becomes $[0, n\xi]$ instead of $[0, \xi]$, as in \eqn{FTSPrep2}.
 We can thus apply a FWLLN to complete the proof.

{\em Regenerative Structure as a basis for the FWLLN.}  As a formal basis for the FWLLN, we can apply regenerative structure associated with successive
epochs of order violation starting from coupling, but a regenerative cycle must contain more than two successive intervals; we do not have an alternating renewal process.
Such order-violation epochs necessarily occur in the boundary region
$\{j: -\beta +1 \le j \le \beta\}$, where $\beta \equiv j \vee k$.  Hence, there are at most $(2 \beta)^3$ vectors of values for the three processes.
Moreover, there necessarily will be one of these vectors visited infinitely often.
Successive visits to that particular state vector after coupling thus serve as regeneration points for the entire process.
That is, there is an embedded delayed renewal process.  We next ensure that the times between successive regeneration times
have finite mean values, with the correct asymptotic properties as $n \ra \infty$.  That justifies the FWLLN.
In particular, we can apply a FWLLN for cumulative processes, as in \cite{GW93}.
For the particular QBD processes being considered, there is continuity of the asymptotic parameters,
as indicated in Lemma \ref{lmQBDcont}.

 {\em Finite mean time between regenerations of order $O(1/n)$.}
From basic CTMC theory, it follows that the first passage times
$\phi_{n}$ in \eqn{phiDef} have finite mean values
which are of order $O(1/n)$.
In particular, $E[\phi_n] < c_{n,1} < \infty$, where $n c_{n,1} \ra c_1 > 0$.  It is more difficult to treat the mean time over
which the order is valid during a single coupling, i.e.,
the initial time $\nu_n$ and the subsequent random times that which the order is valid during the coupling.
However, we can truncate the variables at finite constant times in order to ensure that the FWLLN reasoning can be applied;
e.g., we replace $\nu_n$ by $\tilde{\nu}_n \equiv \nu_n \wedge c_{n,2} < \infty$, where
$n c_{n,2} \ra c_2 > 0$.  We can later choose $c_2$ large enough so that the inequalities in \eqn{DstoBd3} hold for the specified $\ep$.
Finally, the total number of these cycles until the designated initial state vector appears again is a random variable.
Since the successive vector of initial state vectors visited on successive cycles is a finite-state discrete-time Markov chain,
the random number of cycles within a regeneration interval is a random variable with finite mean, say $c_{n,3}$, with $c_{n,3} \ra c_{3} > 0$
as $n \ra \infty$.
As a consequence, the mean time between successive regenerations is $c_{n,4} < \infty$, where $n c_{n,4} \ra c_4 > 0$ as $n \ra \infty$.

{\em The proportion of time that order holds.}
Finally, it is important that we can control the proportion of time that order holds in any interval $[0,t]$, $0 < t \le \xi$,
as stated in \eqn{DstoBd3}, for an appropriate choice of $\xi$.
As a basis for that control, we exploit the fact that the distribution of $\phi_n$, the length of the first interval during which order violation is allowed,
and the random lengths of all subsequent intervals on which the desired sample path order is violated, depend on the
states of the three processes at time $\nu_n$, i.e., upon $D^n_f (X^n_m, \nu_n)$, $D^n_f (X^n_M, \nu_n)$ and
$D^n_{1,2} (\nu_n)$.  However, because the violation necessarily occurs in the boundary region
$\{j: -\beta +1 \le j \le \beta\}$, where $\beta \equiv j \vee k$, there are at most $2\beta$ possible values for each process,
and thus at most $(2 \beta)^3$ vectors of values for the four processes.
Since this number is finite, the violation interval lengths (like $\phi_n$) can be stochastically bounded
above, independent of the specific values at the violation point.  Moreover, this can be done essentially independently of $\xi$.
On the other hand, the length of the intervals on which the coupling construction remains valid necessarily increases
as we decrease $\xi$ and appropriately increase $c_2$ and $c_{n,2}$ in the paragraph above.
Thus, for any $\ep > 0$, we can achieve the claimed ordering in \eqn{DstoBd3} for all $n \ge n_0$ by an appropriate choice of $\xi$, $c_2$ and $n_0$.
\qed

\section{The Bounding QBD in Lemma \ref{lmD21extreme}.} \label{appQBDbdd}

We now describe how to present the process $D^n_*$
in Step One of the proof of Lemma \ref{lmD21extreme} as a QBD for each $n$.
We assume that $r_{2,1} = j/k$ for $j,k \in \ZZ_+$, and let $m \equiv j \vee k$. We define the process
$\tilde{D}^n_* \equiv kQ^n_{2,*} - j Q^n_{1,*}$.
Thus, $\tilde{D}^n_*$ is $D^n_*$ with an altered state space. In particular, each process is positive recurrent if and only if the other one is.

We divide the state space $\ZZ_{+} \equiv \{0, 1, 2, \dots\}$ into level of size $m$:
Denoting level $i$ by $L(i)$, we have
\bes
\bsplit
L(0) & = (0, 1, \dots, m-1), \quad
L(1)  = (m, m+1, \dots, 2m-1) \quad \mbox{etc.}
\end{split}
\ees
The states in $L(0)$ are called the boundary states.
Then the generator matrix $Q^{(n)}$ of the process $D^n_*$ has the QBD form
\bes
Q^{(n)}  \equiv
\left( \begin{array}{ccccc}
   B^{(n)}     &  A^{(n)}_0 & 0         & 0          &  \ldots    \\
   A_2^{(n)}   &  A_1^{(n)} & A^{(n)}_0 & 0          & \ldots    \\
   0           &  A^{(n)}_2 & A^{(n)}_1 & A^{(n)}_0  & \ldots    \\
   0           &  0         & A^{(n)}_2 & A^{(n)}_1  & \ldots    \\
\vdots  &  \vdots & \vdots  & \vdots
\end{array} \right).
\ees
(All matrices are functions of $X^n_*$. However, to simplify notation, we drop the argument $X^n_*$,
and similarly in the example below.)

For example, if $j = 2$ and $k = 3$, then
\bes
\begin{array}{lcr}
B^{(n)} =
    \left( \begin{array}{ccc}
    - \sigma^n + \hatmu^n_{\Sigma}           & 0       & \hatlm^n_2  \\
    \hatmu^n_{\Sigma}                        & -\sigma^n & 0         \\
    \hatmu^n_{\Sigma}                        & 0       & - \sigma^n
    \end{array} \right),
& &
A^{(n)}_0 =
     \left( \begin{array}{ccc}
     \hatlm^n_3  & 0 & 0                   \\
     \hatlm^n_2  & \hatlm^n_3 & 0          \\
     0           & \hatlm^n_2 & \hatlm^n_3
     \end{array} \right),
\\
\\
A^{(n)}_1  =
  \left( \begin{array}{ccc}
   -\sigma^n   & 0           &  \hatlm^n_2   \\
  0            & - \sigma^n  &  0            \\
  \hatmu^n_2   & 0           &  - \sigma^n   \\
                    \end{array} \right),
& &
A^{(n)}_2  =
  \left( \begin{array}{ccccc}
   \hatmu^n_3 & \hatmu^n_2 & 0           \\
  0           & \hatmu^n_3 & \hatmu^n_2  \\
  0           & 0          & \hatmu^n_2
                    \end{array} \right),
\end{array}
\ees
where $\hatmu^n_{\Sigma} \equiv \hatmu^n_3 + \hatmu^n_2$ and $\sigma^n \equiv \hatmu^n_{\Sigma} + \hatlm^n_2 + \hatlm^n_3$.

Let $A^{(n)} \equiv A^{(n)}_0 + A^{(n)}_1 + A^{(n)}_2$.
Then $A^{(n)}$ is an irreducible CTMC infinitesimal generator matrix.
It is easy to see that its unique stationary probability vector, $\nu^{(n)}$, is the uniform probability vector,
attaching probability $1 / m$ to each of the $m$ states.
Then by Theorem 7.2.3 in \cite{LR99}, the QBD is positive recurrent if and only if
$\nu A^{(n)}_0 \mathbf{1} < \nu A^{(n)}_2 \mathbf{1}$,
where $\mathbf{1}$ is the vector of all $1$'s.
This translates to the stability condition given in the proof of Lemma \ref{lmD21extreme}.

\section{Positivity of Fluid Limit.} \label{appPosQ}

In this appendix we explain why Assumption \ref{assC} about the initial conditions is reasonable.
First, the assumed convergence $D^n_{1,2}(0) \Ra L$ if $x(0) \in \AA$ is natural, since
that convergence holds whenever $x(t) \in \AA$, provided that $t$ is not a hitting time of the set $\AA$ from $\rS - \AA$,
by virtue of Theorem \ref{thAPlocal} and Remark \ref{remTightInA}.
We can take time $0$ to be a time shortly after a hitting time of $\AA$.

We now elaborate on another part of Assumption \ref{assC}, namely the assumed initial values for each $n$.
Observe that, given the convergence $\barx^n \Ra x$ established in Theorem \ref{th1}, we necessarily have $Z_{1,2}^n(t) = O_P(n)$
(and thus $Z_{2,1}^n (t) = 0$) and $Q^n_{i}(t) = O_P(n)$, $i = 1,2$, if
$z_{1,2} (t) > 0$ and $q_i (t) > 0$, $i = 1,2$, where $x \equiv (q_1, q_2, z_{1,2})$ is the fluid limit.
However, in general we may not have such strict positivity in the components of the fluid limit at time $0$.
(Recall that Assumption \ref{assC} allows $Q^n_i(0) = o(n)$, $i = 1,2$, and $Z^n_{1,2}(0)$ is allowed to be $0$.)
Nevertheless, we now show that we do necessarily have strict positivity of the components of $x(t)$ for all $t > 0$ small enough,
even if that property does not hold at time $0$.

For a vector $y \in \RR_3$, we write $y > 0$ if the (strict) inequality holds componentwise.
\begin{prop} \label{propPosQ}
The fluid limit $x$ satisfies $x(t) > 0$ for all $t > 0$ sufficiently small.
\end{prop}

\begin{proof}
The statement follows immediately for the case $x(0) > 0$ because of the continuity of $x$. We thus assume that at least one
component of $x(0)$ is not strictly positive.
If $z_{1,2}(0) = 0$, then it was shown in the proof of Theorem \ref{thSSCextend} that $z_{1,2}(t) > 0$ for all $t > 0$.
Now note that, by the assumption on the initial condition $x(0)$ in Assumption \ref{assC}, $q_1(0) \ge rq_2(0)$, so that
if $q_1(0) = 0$, then we must have $q_2(0) = 0$ as well. We prove the result for the three possible regions of $x(0)$
in Assumption \ref{assC}, namely  $\SS^+$, $\AA^+$ and $\AA$, separately.

\paragraph{$(i)$ $x(0) \in \SS^+$.}
In this case, $q_1(0) > 0$ and $\pi_{1,2}(x(0)) = 1$, so we need to consider the case $q_2(0) = 0$. Plugging these values of
$\pi_{1,2}(x(0))$ and $q_2(0)$ in the equation for $\dot{q}_2$ in \eqref{odeDetails}, we see that $\dot{q}_2(0) = \lm_2 > 0$,
so that $q_2$ is strictly increasing at time $0$, and the result follows.

\paragraph{$(ii)$ $x(0) \in \AA^+$.}
It is shown in Lemma 7.2 in \cite{PeW09c} that if $x(0) \in \AA^+$, then for all $t > 0$ sufficiently small,
$x(t) \in \SS - \SS^- - \AA^-$, i.e., $q_1(t) \ge rq_2(t)$ for all $t \in (0, \ep_1]$, for some $\ep_1 > 0$.
As in the previous case, $\pi_{1,2}(x(0)) = 1$ and, assuming $q_2(0) = 0$, we have that $q_2$ is strictly increasing
at time $0$, so that $q_2(t) > 0$ for all $t$ satisfying $0 < t \le \ep_2$, for some $\ep_2 > 0$.
Then, with $\varepsilon \equiv \ep_1 \wedge \ep_2$, $q_1(t) > 0$ for all $t \in (0, \varepsilon]$.

\paragraph{$(iii)$ $x(0) \in \AA$.}
By Theorem 5.2 in \cite{PeW09c}, the fluid limit $x$ is right differentiable at $0$ and is differentiable on an open interval $(0, \varepsilon)$
for some $\varepsilon > 0$.
For $i = 1,2$, if $q_i(t) = 0$, then $\dot{q}_i(t)$ cannot be negative by Theorem 5.1 in \cite{PeW09c}.
Hence, since we are considering the case $q_i(0) = 0$, we have two possibilities: Either $\dot{q}_i(0) = 0$ or $\dot{q}_i(0) > 0$.

Our proof builds on the fact that $\dot{x}(t)$ is itself a continuous function of $t$.
That is so because the right-hand side of each component
of $\dot{x}(t)$ in \eqref{odeDetails} includes the system's parameters, the Lipschitz-continuous fluid limit $x(t)$, and the function $\pi_{1,2}(x(t))$.
However, $\pi_{1,2}(x)$ is locally Lipschitz continuous as a function of $x$ by Theorem 7.1 in \cite{PeW09c},
and in particular, $\pi_{1,2}(x(t))$ is a continuous function of $x(t)$, which is itself a continuous function of $t$.
Hence, $\pi_{1,2}(x(t))$ is continuous in $t$, and so is $\dot{q}_i(t)$, $i = 1,2$.
As a consequence, if $\dot{q}_i(0) > 0$, then $\dot{q}_i(t) > 0$ for all $t$ in some neighborhood of $0$,
so that $q_i$ is strictly increasing on a positive interval, $i = 1,2$.

Hence it remains to consider the case in which
$\dot{q}_i(0) = 0$ for $i = 1$ or $i = 2$.  Assuming that to be the case, we will show that $\dot{q}_j(0) = 0$ for $j \ne i$, $j = 1,2$.
Indeed, if $\dot{q}_j (0) > 0$, then $q_j$ must be increasing at a neighborhood of time $0$,
 so that $x(t) \notin \AA$ for all $t > 0$ small enough, contradicting the fact that $\AA$ is an open set.
Hence, the proof of the proposition for the case $x(0) \in \AA$ will follow if we
assume that $q_i(0) = \dot{q}_i(0) = 0$ for both $i = 1$ and $i = 2$, and show that we reach a contradiction.
We consider two subcases, depending on whether $\mu_{1,2} > \mu_{2,2}$ or $\mu_{1,2} \le \mu_{2,2}$.

If $\mu_{1,2} > \mu_{2,2}$, then define the function $V(x(t)) = q_1(t) + q_2(t)$. It is shown in the proof of Theorem 5.1 in \cite{PeW09c}
that $V$ is strictly increasing whenever $x(t) \in \AA$, so that at least one queue is increasing at time $0$. Once again, this means that the second
queue must also be strictly increasing, for otherwise $x$ will leave $\AA$ immediately after time $0$, so the statement is proved.

If $\mu_{1,2} \le \mu_{2,2}$, then define the function $V(x(t)) = (1+\ep)q_1(t) + q_2(t) + \ep z_{1,2}(t)$, for arbitrary $\ep > 0$.
We again use the proof of Theorem 5.1 in \cite{PeW09c} to conclude that $V$ is strictly increasing at time $0$.
However, in that case, even though $\dot{q}_i(0) = 0$, $i = 1,2$, $V$ can be increasing because $\dot{z}_{1,2}(0) > 0$.
Assume that is the case, and consider the ODE \eqref{odeDetails}. (Recall that $q_i(0) = \dot{q}_i(0) = 0$, $i = 1$ and $2$.)

It follows from the assumption $\dot{z}_{1,2}(0) > 0$ and the equation for $\dot{z}_{1,2}$ in \eqref{odeDetails}, that
\bequ \label{eq225}
\pi_{1,2}(x(0))(\mu_{1,2}z_{1,2}(0) + \mu_{2,2} z_{2,2}(0)) > \mu_{1,2} z_{1,2}.
\eeq
Then, by the assumption $\dot{q}_1(0) = q_1(0) = 0$, we have
\bes
0 = \dot{q}_1(0) = \lm_1 - \mu_{1,1}m_1 - \pi_{1,2}(x(0))(\mu_{1,2}z_{1,2}(0) + \mu_{2,2} z_{2,2}(0)) < \lm_1 - \mu_{1,1}m_1 - \mu_{1,2} z_{1,2}(0),
\ees
where the inequality follows from \eqref{eq225}. Hence,
$\mu_{1,2}z_{1,2}(0) > \lm_1 - \mu_{1,1}m_1$,
which further implies that
\bequ \label{eq336}
z_{1,2}(0) > s^a_2,
\eeq
for $s^a_2$ in \eqref{Qalone}.
From the equation for $\dot{q}_2(0)$ in \eqref{odeDetails} (recalling that $q_2(0) = 0$ by assumption), we get
\bes
\dot{q}_2(0) = \lm_2 - \mu_{1,2}z_{1,2}(0) - \mu_{2,2}z_{2,2}(0) + \pi_{1,2}(x(0))(\mu_{1,2}z_{1,2}(0) + \mu_{2,2} z_{2,2}(0))
> \lm_2 - \mu_{2,2}z_{2,2}(0),
\ees
where again, the inequality follows from \eqref{eq225}. Now, since $z_{2,2}(0) = m_2 - z_{1,2}(0)$, we have
\bes
\dot{q}_2(0) > \lm_2 - \mu_{2,2}m_2 + \mu_{2,2}z_{1,2}(0) > \lm_2 - \mu_{2,2}m_2 + \mu_{2,2} s^a_2 \ge 0,
\ees
where the second inequality follows from \eqref{eq336}, and the third from Assumption \ref{assA}.
This contradicts the assumption that $\dot{q}_2(0) = \dot{q}_1(0) = 0$.
Hence the proof is complete.
\end{proof}

\section{List of Acronyms} \label{AppAcro}

In this appendix we list all the acronyms used in the paper, and refer to the sections where they are introduced and discussed.

AP -- averaging principle (\S \ref{secIntro} and \cite{PeW09b,PeW09c})

CTMC -- continuous time Markov chain (\S \ref{secIntro})

FCLT -- functional central limit theorem (\S \ref{secIntro})

FQR-T -- fixed queue ratio {\em control} with thresholds (\S \ref{secIntro} and \S \ref{secFQRorig})

FSLLN -- functional strong law of large numbers (\S \ref{secSSCserv})

FTSP -- fast-time-scale process (\S \ref{secIntro} and \S \ref{secFTSP})

FWLLN -- functional weak law of large numbers (\S \ref{secIntro} and \S \ref{secStatement})

ODE -- ordinary differential equation (\S \ref{secIntro} and \S \ref{secODE})

QBD -- quasi birth and death {\em process} (\S \ref{secODE} and \cite{PeW09c})

QR-T -- {\em not fixed} queue ratio {\em control} with thresholds (\S \ref{secIntro} and \cite{PeW09a})

SSC -- state-space collapse (\S \ref{secIntro} and \S \ref{secSSCserv})

\section*{Acknowledgments.}

This research is part of the first author's doctoral dissertation in the IEOR Department at Columbia University.
Additional work was done subsequently, including while the first author had a postdoctoral fellowship at CWI in Amsterdam.
The research was supported by NSF grants DMI-0457095 and CMMI 0948190 and 1066372.

\end{appendix}

\end{document}